\documentclass[12pt,a4paper]{amsart}

\usepackage{euscript,amsfonts,amssymb,amsmath,amscd,amsthm,enumerate,hyperref}

\usepackage{comment}

\usepackage{graphicx}

\usepackage{color}

\usepackage[margin=1.1in]{geometry}
\newtheorem{theorem}{Theorem} 
\newtheorem*{theorem*}{Theorem}
\newtheorem{lemma}[theorem]{Lemma}
\newtheorem{definition}[theorem]{Definition}
\newtheorem{proposition}[theorem]{Proposition}

\newtheorem{genEx}[theorem]{General Example}

\theoremstyle{remark}
\newtheorem{remark}[theorem]{Remark}

\numberwithin{equation}{section} \numberwithin{theorem}{section}

\newcommand{\la}{\lambda}

\newcommand{\GT}{\mathbb{GT}}

\newcommand{\N}{\mathbb N}

\newcommand{\R}{\mathbb R}

\newcommand{\C}{\mathbb C}

\newcommand{\HH}{\mathbb H}

\newcommand{\F}{ \mathbf F}

\newcommand{\pa}{\partial}

\newcommand{\cov}{\mathrm{cov}}

\newcommand{\mes}{\mathbf m}

\newcommand{\al}{\alpha}

\newcommand{\ep}{\epsilon}

\newcommand{\ii}{{\mathbf i}}

\newcommand{\FF}{\mathcal F}

\newcommand{\FFF}{\mathfrak F}

\newcommand{\GG}{\mathcal G}

\newcommand{\E}{\mathbf E}

\renewcommand{\c}{\mathbf c}

\renewcommand{\d}{\mathbf d}

\newcommand{\pr}{\mathrm{pr}}

\newcommand{\st}{\mathrm{st}}

\newcommand{\eps}{\varepsilon}

\newcommand{\Cov}{\mathrm{Cov}}

\newcommand{\rh}{\mathfrak{r}} 

\title[Fluctuations of particle systems]
{Fluctuations of particle systems determined by Schur generating functions}

\author{Alexey Bufetov}

\address[Alexey Bufetov]{Department of Mathematics, Massachusetts Institute of Technology, Cambridge, MA, USA. E-mail: alexey.bufetov@gmail.com}

\author{Vadim Gorin}

\address[Vadim Gorin]{Department of Mathematics, Massachusetts Institute of Technology, Cambridge, MA, USA, and
 Institute for Information Transmission Problems of Russian Academy of Sciences, Moscow, Russia. E-mail: vadicgor@gmail.com}

\begin{document}

\begin{abstract}
We develop a new toolbox for the analysis of the global behavior of stochastic
discrete particle systems. We introduce and study the notion of the Schur generating
function of a random discrete configuration. Our main result provides a Central
Limit Theorem (CLT) for such a configuration given certain conditions on the Schur
generating function. As applications of this approach, we prove CLT's for several
probabilistic models coming from asymptotic representation theory and statistical
physics, including random lozenge and domino tilings, non-intersecting random walks,
 decompositions of tensor products of representations of unitary groups.
\end{abstract}

\maketitle

\tableofcontents

\section{Introduction}

\subsection{Overview}

This article is about the random $N$--particle configurations on $\mathbb Z$ and their asymptotic
behavior as $N\to\infty$. For each $N=1,2,\dots,$ let $\ell^{(N)}$ be a random $N$--dimensional
vector
\begin{equation}
\label{eq_particle_config}
 \ell^{(N)}=\bigl(\ell_1^{(N)}>\ell_2^{(N)}>\dots>\ell_N^{(N)}\bigr),\quad \ell_i^{(N)}\in\mathbb
 Z.
\end{equation}
Our aim is to deal with \emph{global} fluctuations of $\ell^{(N)}$. One way to make sense of those
is to take an arbitrary test function $f(x)$ and consider linear statistics
\begin{equation}
\label{eq_linear_statistics}
 \mathcal L_f^{(N)}=\sum_{i=1}^{N} f\left(\frac{\ell^{(N)}_i}{N} \right).
\end{equation}
We mostly deal with the case when $f(x)$ is a polynomial (or, more generally, a smooth function),
yet if $f(x)$ is the indicator function of an interval, then \eqref{eq_linear_statistics} merely
counts the number of random particles inside this interval.

Since by its definition,  $\mathcal L_f^{(N)}$ is a sum of $N$ terms, it is
reasonable to expect that it grows linearly in $N$. And, indeed, in the class of
systems that we study, $\frac{1}{N} \mathcal L_f^{(N)}$ converges as $N\to\infty$ to
a \emph{deterministic} limit depending on the choice of $f$. We will refer to such a
phenomenon as the Law of Large Numbers, appealing to the evident analogy with a
similar statement of classical probability dealing with sequences of independent
random variables.

The next natural question is to study the fluctuations $\mathcal L_f^{(N)}-\E \mathcal L_f^{(N)}$
as $N\to\infty$. Such fluctuations would grow as $\sqrt{N}$ in the systems arising from sequences
of independent random variables, but the scale is different in our context. We deal with
probability distributions coming from $2d$ statistical mechanics (lozenge and domino tilings,
families of non-intersecting paths), asymptotic representation theory, random matrix theory, and
for them the typical situation is that $\mathcal L_f^{(N)}-\E \mathcal L_f^{(N)}$ does not grow as
$N\to\infty$. Nevertheless, in all cases the fluctuations are asymptotically Gaussian, which
justifies the name Central Limit Theorem for these kinds of results.

\medskip

The main theme of the present article is to develop a new toolbox for proving the Law of Large
Numbers and Central Limit Theorems, which would be \emph{robust} to perturbations of $\ell^{(N)}$.
It is somewhat hard to concisely describe the class of systems where the toolbox is helpful. One
reason is that we believe our conditions to be in a sense equivalent to the LLN and CLT, see the
end of Section \ref{Section_moments_method} (which, of course, does not make these conditions
immediate to check). Yet we list below an extensive list of available applications.

Again coming back to the classical one-dimensional probability, a universal tool is given there by
the method of \emph{characteristic functions}. For instance, it can be used to prove that averages
of independent random variables converge to a Gaussian limit under very mild assumptions on the
distributions of these variables, cf.\ textbooks \cite{Kallenberg}, \cite{Durrett}.

In our context the characteristic functions were not found to be useful, mostly due to the fact
that the dimension (number of the particles) grows with $N$, while the individual coordinates
$\ell^{(N)}_i$, $i=1,\dots,N$ are very far from being independent. Therefore, we suggest to replace
them by a new notion of \emph{Schur generating function} which we now introduce.

Recall that a Schur function is a symmetric Laurent polynomial in variables $x_1,\dots,x_N$
parameterized by $\lambda=(\lambda_1\ge \lambda_2 \ge \dots\ge \lambda_N)$ and given by
$$
 s_\lambda(x_1,\dots,x_N)=\frac{\det\left[ x_i^{\lambda_j+N-j}\right]_{i,j=1}^N}{\prod_{1\le i<j
 \le N} (x_i-x_j)}.
$$
Let $\delta_N$ denote the $N$--tuple $(N-1,N-2,\dots,0)$ and note that the map
$\lambda\to \lambda+\delta_N$ makes the weakly decreasing coordinates of $\lambda$
strictly decreasing, as in \eqref{eq_particle_config}.

For a random $N$--tuple of strictly ordered integers $\ell^{(N)}$, as in
\eqref{eq_particle_config}, its distribution is a function $\rh_N$ of $N$ weakly
decreasing integers given by $\rh_N (\lambda) := {\rm Prob}
(\ell^{(N)}=\lambda+\delta_N)$.

\begin{definition}
The Schur generating function $S_{\rh_N}$ of a random $\rh_N$--distributed
$N$--particle configuration $\ell^{(N)}$, is a function of $N$ variables
$x_1,\dots,x_N$ given by
\begin{equation}\label{eq_Schur_gen_intro}
 S_{\rh_N} (x_1,\dots,x_N)=\sum_{\lambda} \rh_N \left(\lambda\right)
 \frac{s_\lambda(x_1,\dots,x_N)}{s_\lambda(1,1,\dots,1)}.
\end{equation}
\end{definition}

In \cite{BG} we showed how the Law of Large Numbers can be extracted from the
asymptotic behavior of Schur generating functions. Interestingly, the answer, i.e.\
the exact formula for $\lim_{N\to\infty} \frac{1}{N} \mathcal L_f^{(N)}$ depends
only on $f$ and the $N\to\infty$ asymptotics of $S_{\rh_N}(x_1,1,1,\dots,1)$, that is, \emph{all variables except for one} can be set to $1$ prior to the asymptotic analysis. A similar
phenomenon was also found in \cite{Novak2} by another method.


Here we make the next step and address the Central Limit Theorem for global fluctuations, the precise
statement in this direction is Theorem \ref{theorem:main-one-level}. In fact, we go even further,
and also analyze random \emph{sequences} of $N$--particle configurations forming Markov chains, see
Theorems \ref{theorem:main-projections}, \ref{theorem:main-multiplication},
\ref{th:general-for-domino} below. This time the answer, which is the covariance for
$\lim_{N\to\infty} \left(\mathcal L_f^{(N)}-\E \mathcal L_f^{(N)}\right)$,
depends only on asymptotics of
$S_{\rh_N}(x_1,x_2,1,1,\dots,1)$, that is, \emph{all variables except for two} can be set to $1$ prior to the asymptotic analysis.
For proving the asymptotic Gaussianity we need more.

Our theorems reduce the LLN and CLT to asymptotic behavior of Schur generating
functions, which is \emph{known} in many cases. This leads to proofs of the LLN and
CLT for a variety of stochastic systems of particles, including:
\begin{enumerate}
\item Lozenge tilings of trapezoid domains, cf.\ Figure \ref{Fig_Lozenge_1} in Section \ref{Section_Applications}.
\item Domino tilings of Aztec diamond, cf.\ Figure \ref{Fig_Domino_1} in Section \ref{Section_Applications} (for the application to domino tilings of
more complicated domains see \cite{BK}).
\item Ensembles of non-intersecting random walks, cf.\ Figure \ref{Fig_Poisson_non} in Section \ref{Section_Applications}.
\item $2+1$--dimensional random growth models. 
\item Measures governing the decomposition into irreducible components for tensor products of
irreducible representations of the unitary group $U(N)$.
\item Measures governing the decomposition of restrictions onto $U(N)$ of extreme characters of the
infinite--dimensional unitary group $U(\infty)$.
\item Schur--Weyl measures.
\end{enumerate}
A more detailed exposition of the applications of our method is given in Section
\ref{Section_Applications}.

\subsection{Previous work on the subject}

One advantage of our approach through Schur generating functions is that it is quite general, and
as a result, in each of our applications we can address more general situations than those
rigorously known before. However, particular cases of some of our applications were accessible
previously by other important techniques. Let us list several of those.

\begin{itemize}

\item Determinantal point processes have led to Central Limit Theorems for uniformly random lozenge
tilings of certain domains in \cite{Ken}, \cite{Pet2}, for $2+1$ dimensional random
growth in \cite{BF}, \cite{Duits_2rates}, \cite{Kuan_ips}. Similar results for
domino tilings of the Aztec diamond were announced (without technical details) in
\cite{CJY}.

\item Asymptotic analysis of orthogonal polynomials through the recurrence relations has led to
Central Limit Theorems for ensembles of non-intersecting paths with specific initial
conditions (which also include some tiling models) in \cite{Br_Du},
\cite{Duits_nonc}.

\item Discrete loop equations (also known as Nekrasov equations) have led in \cite{BGG} to Central
Limit Theorems for discrete log--gases, which has overlaps with specific ensembles of
non-intersecting paths and tilings.

\item Various representation--theoretic ideas, involving, in particular, computations in the algebra of shifted symmetric functions and
universal enveloping algebra of $\mathfrak{gl}_N$ have led in \cite{Kerov},
\cite{Ivanov_Olsh}, \cite{Fulman}, \cite{Hora_Obata}, \cite{Borodin_Bufetov},
\cite{BBO}, \cite{Dol_Fer}, \cite{Kuan_noncom}, \cite{Meliot} to several instances of Central
Limit Theorem for the probability distributions of asymptotic representation theory.

\item Differential operators acting in the algebra of symmetric functions in infinitely-many
variables were used in \cite{Moll} for proving the Central Limit Theorem for the
Jack measures.

\end{itemize}

Let us emphasize, that despite the existence of several competing methods, most of our applications
were not previously accessible by any of them. Yet our technique is adapted to the study of the
global behavior of probabilistic systems, while some of these methods are more suitable for the
study of the local behavior.

\subsection{Continuous models}

Replacing $\ell_i^{(N)}\in \mathbb Z$ by $\ell_i^{(N)}\in\mathbb R$ in \eqref{eq_particle_config},
we arrive at continuous analogues of the particle configurations under consideration. In this
fashion, our results are closely related to the global asymptotics for the eigenvalues of random
matrix ensembles.

One precise example is given by the \emph{semiclassical limit}, which degenerates the decomposition
of tensor products of irreducible representations of $U(N)$ (one of our applications) to spectral
decomposition of sums of independent Hermitian matrices, see \cite[Section 1.3]{BG} for the
details. The Central Limit Theorem for this random matrix problem is well-known, see \cite[Section
10]{PS}. It can be put into the context of the \emph{second order freeness} in the free probability
theory, see \cite{MS}, \cite{MSS}. In Section \ref{sec:9-4} we explain how the covariance for our
Central Limit Theorem for tensor products degenerates to the random matrix one.

Another degeneration is the appearance of the Gaussian Unitary Ensemble (GUE) as a scaling limit of
lozenge and domino tilings near the boundary of the tiled domain, see \cite{OR_Birth}, \cite{JN},
\cite{GP}, \cite{Novak}. Recall that GUE is the eigenvalue distribution of $H=\frac12 (X+X^*)$,
where $X$ is $N\times N$ matrix of i.i.d.\ mean $0$ complex Gaussian random variables. And again
for GUE the Gaussian asymptotics for global fluctuations is well--known and can be generalized in
(at least) two directions. The first one is a general Central Limit Theorem for (continous)
log--gases of \cite{Joh_CLT} based on the loop equations. The second generalization is to replace
the Gaussian distributions in the definition of GUE by arbitrary ones and to study the resulting
\emph{Wigner matrix}. Then the global fluctuations can be accessed by the \emph{moments method},
see e.g.\ \cite[Chapter 2]{AGZ} for an exposition. In more details, one computes the moments of the
eigenvalues in the following form
\begin{equation} \label{eq_moments_of_rm}
\E\left( \prod_{k=1}^n {\rm Trace}\bigl(H^{m_k}\bigr) \right) = \E \left(
\prod_{k=1}^n \sum_{i=1}^N (h_i)^{m_k} \right),\quad \{h_i\}_{i=1}^N \text{ are
eigenvalues of }H.
\end{equation}
The independence of matrix elements of $H$ paves a way to find the asymptotic of the left--hand
side of \eqref{eq_moments_of_rm}, which then gives the global asymptotic of linear statistics of
the form \eqref{eq_linear_statistics} with polynomial test functions $f(x)$.

\subsection{Moments method}

\label{Section_moments_method}

The moments method was never available for the \emph{discrete} particle configurations as in
\eqref{eq_particle_config} for a very simple reason: there is no underlying random matrix or an
analogue thereof. Here we change this situation by providing a way to efficiently compute (a discrete
analogue of) the right--hand side in \eqref{eq_moments_of_rm}. Let us briefly state the key idea.

Let $\partial_i$ denote the derivative with respect to the variable $x_i$ and consider the
differential operator
$$
\mathcal D_m=\prod_{1\le i<j \le N} \frac{1}{ x_i-x_j} \left( \sum_{i=1}^N \left (x_i
\partial_i\right)^m \right) \prod_{1\le i<j\le N} (x_i-x_j).
$$
A straightforward computation shows that the Schur functions are eigenvectors of $\mathcal D_m$:
$$
 \mathcal D_m s_\lambda= \left(\sum_{i=1}^{N} (\lambda_i+N-i)^m \right) s_\lambda,\quad
 \lambda=(\lambda_1,\dots,\lambda_N).
$$
Therefore, applying such operators to \eqref{eq_Schur_gen_intro} we get
\begin{equation}
\label{eq_moment_computation}
 \E \left( \prod_{k=1}^n
\sum_{i=1}^N \left(\ell_i^{(N)}\right)^{m_k} \right)= \left[ \left(\prod_{k=1}^n \mathcal
D_{m_k}\right) S_{\rh_N} \right]_{x_1=x_2=\dots=x_N=1}.
\end{equation}
The fact that differential (or difference) operators applied to symmetric functions
can be used for the analysis of random particle configurations is by no means new, see
e.g.\ \cite{BC}, \cite{BCGS} for recent similar statements, and the asymptotic
questions boil down to finding a way to analyze the right--hand side of
\eqref{eq_moment_computation}. This is where the specific and relatively simple
definition of $\mathcal D_m$ shines, as we are able to develop a combinatorial
approach (yet based on several analytic lemmas) to the right--hand side of
\eqref{eq_moment_computation}.

One important observed feature is that the right--hand side of
\eqref{eq_moment_computation} depends only on the values of the Schur generating
function $S_{\rh_N}$ at points $(x_1,\dots,x_N)$ such that \emph{all but
a bounded number} of coordinates (i.e.\ the total number is not growing with $N$) are
equal to $1$. First, this reduces a problem in growing (with $N$) dimension to a
much more tractable finite--dimensional form. Second, the values of Schur
generating functions at such points are very robust and not too sensitive to small
perturbations for $\ell^{(N)}$. This is indicated by the results of \cite{GM},
\cite{GP} on the asymptotics of Schur functions, on which we elaborate in Section
\ref{sec:asymp-Schur-func}. In particular, these results give enough control on the
values of Schur generating functions to give the asymptotic expansion for the
left--hand side of \eqref{eq_moment_computation} needed for the Central Limit
theorem. In contrast to our method, previous results and related approaches in the
area, such as those of \cite{Borodin_Bufetov}, \cite{BC}, \cite{BCGS}, \cite{BBO},
\cite{Moll} relied on the \emph{exact form} of the Schur generating function or its
analogue; in particular, it was necessary to assume its factorization into a product
of $1$--variable functions.

From the technical point of view, even after all these observations are made, the asymptotic
analysis still needs many efforts and is much more complicated than that of \cite{BG} where the Law
of Large Numbers was addressed through the same technique.

\medskip

Let us end this section with a speculation. We believe that it should be possible to \emph{reverse}
the theorems of the present article: the knowledge of the Law of Large Numbers and Central Limit
Theorem should give (perhaps, subject to technical conditions) exhaustive information about
asymptotics of the Schur generating functions for all but finitely many values of coordinates $x_i$
equal to $1$. We plan to develop this direction in a separate publication\footnote{This was subsequently proven to be true, see \cite{Bufetov_Gorin_GFF2}.}.

\subsection{Organization of the article}

The rest of the text is organized as follows. In Section \ref{Section_main_results} we formulate
our main results linking the Central Limit Theorem for global fluctuations to the asymptotic of
Schur generating functions. Numerous applications of these results are presented in Section
\ref{Section_Applications}. Section \ref{sec:4} gives a generalization of
\eqref{eq_moment_computation} which underlies all our developments. The remaining sections present
a step-by-step proof for the statements of Sections \ref{Section_main_results} and
\ref{Section_Applications}.

\subsection{Acknowledgements}

We would like to thank Alisa Knizel for help with preparing the picture of a domino tiling. We
thank Alexei Borodin for useful comments. We are grateful to an anonymous referee for suggestions on improving the text.
V.G.\ was partially supported by the NSF grant
DMS-1407562, by the NEC Corporation Fund for Research in Computers and Communications and by the Sloan Research Fellowship. Both authors were partially supported by the NSF grant DMS-1664619.

\subsection{Notation}
\label{sec:notations}

Here we collect some notations that we use throughout this paper. Note that some of these notations are slightly unconventional.

By $\vec{x}$ we denote the variables $(x_1, \dots, x_N)$.

We denote by $(1^N)$ the sequence $\underbrace{(1,1,\dots,1)}_{N}$.

By $\pa_i$ we denote the partial derivative $\frac{\pa}{\pa x_i}$. We use $\pa_z$ instead of
$\frac{\pa}{\pa z}$. For a function of one variable $f(x)$ we sometimes denote the derivative by the
conventional notation $f'(x)$. By $\pa_i^0 f$ we mean the function $f$ itself.

For a differential operator $\mathcal D$ by $\mathcal D [ F(x) ] G(x)$ we mean that the differential operator is applied to $F(x)$ only.
Let $S_r$ be the group of all permutations of $r$ elements; then
$$
Sym_{x_1, \dots, x_{r}} f(x_1, \dots, x_r) := \frac{1}{r!} \sum_{\sigma \in S_r} f( x_{\sigma(1)}, x_{\sigma(2)}, \dots, x_{\sigma(r)}),
$$
denotes the symmetrization of a function.

Let $V_N (\vec{x}) := \prod_{1 \le i<j \le N} (x_i-x_j)$ be the Vandermond determinant in variables $x_1, \dotsm x_N$.

Sometimes we omit the arguments of functions in formulas. For example, we can use the symbol $V_N$ instead of $V_N (\vec{x})$.

We use notations $[N]:= \{1,2,\dots, N\}$, $[2;N] := \{2,3,\dots, N\}$.

$\sum_{\{a_1, \dots, a_{r} \} \subset [N]}$ denotes the summation over all subsets of $[N]$ consisting of $r$ elements.

All contours of integration in this paper are counter-clockwise.

\section{Main results}

\label{Section_main_results}

\subsection{Preliminaries and Law of Large Numbers}
\label{sec:LLN-formulation}

An $N$-tuple of non-increasing integers $\la = (\la_1\ge \la_2 \ge \dots \ge \la_N)$
is called a \textit{signature} of length $N$. We denote by $\GT_N$ the set of all
signatures of length $N$. The \textit{Schur function} $s_{\la}$, $\la\in\GT_N$, is a
symmetric Laurent polynomial defined by
\begin{equation*}
 s_\la (x_1, \dots ,x_N)=\frac{\det \left[x_i^{ \lambda_j + N -j}\right] }{ \det \left[x_i^{N-j}\right]}= \frac{\det\left[x_i^{\lambda_j+N-j}\right]}{\prod \limits_{i<j}
 (x_i-x_j)}.
\end{equation*}

Let $\rh$ be a probability measure on the set $\GT_N$.
A \emph{Schur generating function}
$S_{\rh} (x_1,\dots,x_N)$ is a symmetric Laurent power series in $x_1,\dots,x_N$ given by
$$
 S_{\rh} (x_1,\dots,x_N)=\sum_{\lambda \in \GT_N} \rh (\lambda)
 \frac{s_\lambda(x_1,\dots,x_N)}{s_\lambda(1^N)}.
$$
In what follows we always assume that the measure $\rh$ is such that this (in
principle, formal) sum is uniformly convergent in an open neighborhood of $(1^N)$.
Note that the uniform convergence of such a series in a neighborhood of $(1^N)$
implies the uniform convergence in an open neighborhood of the $N$-dimensional torus
$\{(x_1, \dots, x_N): |x_i|=1, i=1, \dots, N \}$. Indeed, it follows from the
estimate $|s_{\la} (x_1, \dots, x_n)| \le s_{\la} (|x_1|, \dots, |x_n|)$ (which is
an immediate corollary of the combinatorial formula for Schur functions as a positive sum of monomials, see
\cite[Chapter I, Section 5, (5.12)]{M}).

The goal of this paper is to show how to extract information about $\rh$ with the help of $S_\rh (x_1,\dots,x_N)$.

\begin{definition}
\label{def:LLN-appr_schur} A sequence of symmetric functions $\{ \F_N (\vec{x})
\}_{N \ge 1}$ is called \textbf{LLN-appropriate} if there exists a collection of
reals $\{ \c_k \}_{k \ge 1}$ such that

\begin{itemize}

\item For any $N$ the function $\log F_N (\vec{x})$ is holomorphic in an open complex neighborhood of $(1^N)$.

\item
For any index $i$ and any $k \in \N$ we have
$$
\lim_{N \to \infty} \left. \frac{ \pa_i^k \log F_N (\vec{x})}{N} \right|_{\vec{x}= (1^N)} = \c_k.
$$

\item
For any $ s \in \N$ and any indices $i_1, \dots, i_s$ such that there are at least
two distinct indices among them we have
$$
\lim_{N \to \infty} \left. \frac{ \pa_{i_1} \dots \pa_{i_s} \log F_N (\vec{x})}{N} \right|_{\vec{x}=1^N} = 0.
$$

\item
The power series
$$
\sum_{k=1}^{\infty} \frac{\c_k}{(k-1)!} (x-1)^{k-1}
$$
converges in a neighborhood of the unity.

\end{itemize}
\end{definition}

\begin{definition}
\label{def:LLN-appr_mes} A sequence $\rho = \{ \rho_N \}_{N \ge 1}$, where $\rho_N$
is a probability measure on $\GT_N$, is called \textbf{LLN-appropriate} if the
sequence $\{ S_{\rho_N} \}_{N \ge 1}$ of its Schur generating functions is
\textit{LLN-appropriate}. For such a sequence we define a function $F_{\rho} (x)$
via
$$
F_{\rho} (x) := \sum_{k=1}^{\infty} \frac{\c_k}{(k-1)!} (x-1)^{k-1},
$$
where $\{ \c_i \}_{i \ge 1}$ are the coefficients from Definition \ref{def:LLN-appr_schur}.
\end{definition}

\begin{genEx}
\label{exam:lln}
Assume that the Schur generating functions of a sequence of probability measures $\rho = \{ \rho_N \}_{N \ge 1}$, where $\rho_N$ is a probability measure on $\GT_N$, satisfies the condition
$$
\lim_{N \to \infty} \frac{ \pa_1 \log S_{\rho_N} (x_1, \dots, x_k, 1^{N-k})}{N} = U (x_1), \qquad \mbox{for any $\ k \ge 1$},
$$
where $U (x)$ is a holomorphic function, and the convergence is uniform in a complex
neighborhood of $(1^k)$. Then $\rho_N$ is a LLN-appropriate sequence with $F_{\rho}
(x) = U (x)$.
\end{genEx}
Indeed, for a uniform limit of holomorphic functions the order of taking derivatives and limit can be interchanged, which shows that the example above is correct. In applications studied in this paper all LLN-appropriate measures will come from the construction of Example \ref{exam:lln}. However, we prefer to prove general theorems in a slightly more general setting of Definition \ref{def:LLN-appr_mes}.

For a signature $\la \in \GT_N$ consider the measure on $\R$
\begin{equation}
\label{eq_signature_measure} m[ \la] := \frac{1}{N} \sum_{i=1}^N \delta \left(
\frac{ \la_i + N -i}{N} \right).
\end{equation}
The pushforward of a measure $\rh$ on $\GT_N$ with respect to the map $\la \to
m[\la]$ defines a random probability measure on $\R$ which we denote by $m[\rh]$.

The following theorem is essentially \cite[Theorem 5.1]{BG}. In Section \ref{sec:4-2} we comment on the slight difference between this formulation and the one given in \cite{BG}.

\begin{theorem}
\label{theorem:LLN-general}
 Suppose that a sequence of probability measures $\rho = \{ \rho(N) \}_{N \ge 1}$, where $\rho(N)$ is a probability measure on $\GT_N$, is LLN-appropriate, and $k \in \N$.
Then the random measures $m[\rho(N)]$ converge as
$N\to\infty$ in probability, in the sense of moments to a \emph{deterministic} measure $\mes$ on
$\mathbb R$, such that its $k$th moment equals
\begin{equation}
\label{eq:LLN-theorem-statem}
\int_{\R} x^k d \mes(x) = \frac{1}{2 \pi \ii (k+1)} \oint_{|z|= \ep} \frac{dz}{1+z} \left( \frac{1}{z} +1 + (1+z) F_{\rho} (1+z) \right)^{k+1},
\end{equation}
where $\ep \ll 1$.
\end{theorem}

\subsection{Main result: CLT for one level}
\label{sec:2-2}

\begin{definition}
\label{def:appr-func}
We say that a sequence of symmetric functions $\{ F_N (x_1, \dots, x_N) \}_{N \ge 1}$ is \textbf{appropriate} (or \textbf{CLT-appropriate}) if there exist two collections of reals $\{ \c_k \}_{k \ge 1}$, $\{ \d_{k,l} \}_{k,l \ge 1}$, such that
\begin{itemize}

\item For any $N$ the function $\log F_N (\vec{x})$ is holomorphic in an open complex neighborhood of $(1^N)$.

\item
For any index $i$ and any $k \in \N$ we have
$$
\lim_{N \to \infty} \left. \frac{ \pa_i^k \log F (\vec{x}) }{N} \right|_{\vec{x}=1} = \c_k,
$$

\item
For any distinct indices $i,j$ and any $k,l \in \N$ we have
$$
\lim_{N \to \infty} \left. \pa_i^k \pa_j^l \log F_N (\vec{x}) \right|_{\vec{x}=1} = \d_{k,l},
$$

\item
For any $s \in \N$ and any indices $i_1, \dots, i_s$ such that there are at least three distinct numbers among them we have
$$
\lim_{N \to \infty} \left. \partial_{i_1} \partial_{i_2} \dots \partial_{i_s} \log F_N (\vec{x}) \right|_{\vec{x}=1} = 0,
$$

\item
The power series
$$
\sum_{k=1}^{\infty} \frac{\c_k}{(k-1)!} (x-1)^{k-1}, \qquad \sum_{k=1; l=1}^{\infty} \frac{\d_{k,l}}{(k-1)! (l-1)!} (x-1)^{k-1} (y-1)^{l-1},
$$
converge in an open neighborhood of $x=1$ and $(x,y)=(1,1)$, respectively.

\end{itemize}
\end{definition}


\begin{definition}
\label{def:main} We say that a sequence of measures $\rho = \{ \rho_N \}_{N \ge 1}$
is \textbf{appropriate} (or \textbf{CLT-appropriate}) if the sequence of its Schur
generating functions $\{ S_{\rho_N} (x_1, \dots, x_N) \}_{N \ge 1}$ is appropriate.
For such a sequence we define functions
$$
F_{\rho} (x) = \sum_{k=1}^{\infty} \frac{\c_k}{(k-1)!} (x-1)^{k-1}, \qquad G_{\rho}
(x,y) = \sum_{k=1; l=1}^{\infty} \frac{\d_{k,l}}{(k-1)! (l-1)!} (x-1)^{k-1}
(y-1)^{l-1},
$$
$$
Q_{\rho} (x,y) = G_{\rho} (1+x,1+y) + \frac{1}{(x-y)^2}.
$$
\end{definition}

\begin{genEx}
\label{exam:clt}
Assume that the Schur generating function of a sequence of probability measures $\rho = \{ \rho_N \}_{N \ge 1}$ on $\GT_N$ satisfies the conditions
$$
\lim_{N \to \infty} \frac{ \pa_1 \log S_{\rho_N} (x_1, \dots, x_k, 1^{N-k})}{N} = U_1 (x_1), \qquad \mbox{for any $\ k \ge 1$},
$$
$$
\lim_{N \to \infty} \pa_1 \pa_2 \log S_{\rho} (x_1, \dots, x_k, 1^{N-k}) = U_2 (x_1, x_2), \qquad \mbox{for any $\ k \ge 1$},
$$
where $U_1(x), U_2(x,y)$ are holomorphic functions, and the convergence is uniform in a complex neighborhood of unity. Then $\rho$ is a (CLT-)appropriate sequence of measures with $F_{\rho} (x) = U_1(x)$, $G_{\rho} (x,y) = U_2 (x,y)$.
\end{genEx}

Indeed, for a uniform limit of holomorphic functions the order of taking derivatives and limit can be interchanged, which shows that the example above is correct. In applications studied in this paper all CLT-appropriate measures will come from the construction of Example \ref{exam:clt}. However, we prefer to prove theorems for a slightly more general setting of Definition \ref{def:main}.

Let $\rho_N$ be a probability measure on $\GT_N$. Set
$$
p_k^{(N)} := \sum_{i=1}^N \left( \la_i +N-i \right)^k, \qquad k=1,2, \dots, \quad \mbox{$\lambda=(\la_1, \dots, \la_N)$ is $\rho_N$-distributed.}
$$

The following theorem is the main result of this paper.

\begin{theorem}
\label{theorem:main-one-level}
Let $\rho = \{ \rho_N \}_{N \ge 1}$ be an appropriate sequence of measures on signatures with limiting functions $F_{\rho}(x)$ and $Q_{\rho} (x,y)$ (see Definition \ref{def:main}).

Then the collection
$$
\{ N^{-k} ( p_{k}^{(N)} - \E p_{k}^{(N)} ) \}_{k \in \N}
$$
converges, as $N \to \infty$, in the sense of moments, to the Gaussian vector with zero mean and covariance
\begin{multline*}
\lim_{N \to \infty} \frac{\cov(p_{k_1}^{(N)}, p_{k_2}^{(N)})}{N^{k_1+k_2}} = \frac{1}{(2 \pi \ii)^2} \oint_{|z|=\ep} \oint_{|w|=2 \ep} \left( \frac{1}{z} +1 + (1+z) F_{\rho} (1+z) \right)^{k_1} \\ \times \left( \frac{1}{w} +1 + (1+w) F_{\rho} (1+w) \right)^{k_2} Q_{\rho} (z, w) dz dw,
\end{multline*}
where the $z$- and $w$-contours of integration are counter-clockwise and $\ep \ll 1$.
\end{theorem}

This theorem serves as a model example of our approach. However, for applications it
is often required to study the joint distributions of several random particle
systems. Our approach can be applied to (some of) these cases as well: We deal with
them in the next sections.

\subsection{General setting for several levels}
\label{Section_multilevel_general}

Let us introduce a general construction of Markov chains which are analyzable by our methods.

For a positive integer $m$ and $\eps>0$ let $\Lambda^{m}_{\eps}$ be the space of
analytic symmetric functions in the region
$$
\{ (z_1, \dots, z_m) \in \C^m : 1+\eps > |z_1|>1-\eps, 1+\eps > |z_2|>1-\eps, \dots,
1+\eps > |z_m|>1-\eps \}.
$$
We consider $\Lambda^{m}_{\eps}$ as a topological space with topology of uniform convergence in this region.

Consider $\Lambda^{m}:= \cup_{\eps>0} \Lambda^m_{\eps}$ endowed with the topology of the inductive limit. Note that for $\mathfrak{f} \in \Lambda^{m}$ the function $\mathfrak{f} (x_1, x_2, \dots, x_m) \prod_{1 \le i<j \le m} (x_i-x_j)$ is an (antisymmetric) analytic function. Therefore, it can be written as an absolutely convergent sum of monomials $x_1^{l_1} \dots x_m^{l_m}$, where $l_i \in \mathbb{Z}$, $i=1,2, \dots, m$. Dividing both sides of such a sum by $\prod_{1 \le i<j \le m} (x_i-x_j)$, we obtain that each element of $\Lambda^{m}$ can be written in a unique way as an absolutely convergent sum
$$
\sum_{\la \in \GT_m} \mathfrak{c}_{\la} s_{\la} (x_1, \dots, x_m),  \qquad \mathfrak{c}_{\la} \in \C,
$$
in some neighborhood of the $m$-dimensional torus.

We consider a map $\mathfrak{p}_{m, n} : \Lambda^m \to \Lambda^n$ with the following properties:

1) $\mathfrak{p}_{m, n}$ is a linear continuous map.

2) For every $\la \in \GT_m$ we have
$$
\mathfrak{p}_{m, n} \left( \frac{s_{\la} (x_1, \dots, x_m)}{s_{\la} (1^m)} \right) = \sum_{\mu \in \GT_n} \mathfrak{c}_{\la,\mu}^{\mathfrak{p}_{m, n}} \frac{s_{\mu} (x_1, \dots, x_n)}{s_{\mu} (1^n)}, \qquad \mathfrak{c}_{\la,\mu}^{\mathfrak{p}_{m, n}} \in \R_{\ge 0}.
$$
This property says that the coefficients $\mathfrak{c}_{\la,\mu}^{\mathfrak{p_{m, n}}}$ must be nonnegative reals (rather than arbitrary complex numbers). Note that the sum in the right-hand side is absolutely convergent due to the definition of $\mathfrak{p_{m, n}}$.

3) For any $\mathfrak{f} \in \Lambda^m$ we have
$$
\mathfrak{f} (1^m) = \mathfrak{p}_{m, n} ( \mathfrak{f} ) (1^n).
$$
In words, this property asserts that our map should preserve the value at unity.

It follows from conditions 2) and 3) that
$$
\sum_{\mu \in \GT_n} \mathfrak{c}_{\la,\mu}^{\mathfrak{p}_{m, n}} =1.
$$
Since these coefficients are nonnegative reals one can consider them as transitional probabilities of a Markov chain. In more details, let $n_1, \dots, n_s$ be positive integers, and let $\mathfrak{p}_{n_2, n_1}$, ..., $\mathfrak{p}_{n_{s}, n_{s-1}}$ be maps satisfying conditions above. Let $\rho$ be a probability measure on $\GT_{n_s}$. Define the probability measure on the set
$$
\GT_{n_1} \times \GT_{n_2} \times \dots \times \GT_{n_s}
$$
via
\begin{equation}
\label{eq:gener-form-measure-markov}
\mathrm{Prob} (\la^{(1)}, \la^{(2)}, \dots, \la^{(s)}) = \rho (\la^{(s)} ) \prod_{i=2}^{k} \mathfrak{c}_{\la^{(i)},\la^{({i-1})}}^{\mathfrak{p_{n_i, n_{i-1}}}}.
\end{equation}
In Section \ref{sec:4} we prove a formula for the expectation of \textit{joint} moments of signatures $\{ \la^{(i)} \}_{i=1, \dots, s}$ distributed according to this measure.

\subsection{Main result: CLT for several levels}
\label{Section_multilevel_specified} In this section we state three multi-level
generalisations of Theorem \ref{theorem:main-one-level}. They are mainly shaped to
the applications studied in the present paper. With the use of the construction from
Section \ref{Section_multilevel_general} it is possible to produce many other
similar multi-level generalisations of Theorem \ref{theorem:main-one-level}; this
should be regulated by applications that one has in mind.

We consider the following particular examples of the map $\mathfrak{p_{m, n}}$ from
Section \ref{Section_multilevel_general}. In the first case, it is given by $f (x_1,
\dots,x_m) \to f (x_1,\dots, x_n, 1^{m-n})$, for $m>n$. In the second case, it is
given by $ s_{\la} (x_1, \dots,x_m) \to g (x_1, \dots, x_m) s_{\la} (x_1,
\dots,x_m)$, for $m=n$ and a function $g (x_1, \dots, x_m)$ which is a Schur
generating function of a probability measure on $\GT_m$. Finally, in the third case
we combine the two previous ones.

In an attempt to make the exposition more explicit, we repeat the construction of
Section \ref{Section_multilevel_general} in all three cases below.

\textbf{Example 1)} For $\la \in \GT_{k_1}, \mu \in \GT_{k_2}$, with $k_1 \ge k_2$, let us introduce the coefficients $\pr_{k_1 \to k_2} (\la \to \mu)$ via
\begin{equation}
\label{eq:Schur-branching-pr-coef}
\frac{s_{\lambda} (x_1, \dots, x_{k_2}, 1^{k_1-k_2})}{s_{\la} (1^{k_1})} = \sum_{\mu \in \GT_{k_2}} \mathrm{pr}_{k_1 \to k_2} (\la \to \mu) \frac{s_{\mu} (x_1, \dots, x_{k_2})}{s_{\mu} (1^{k_2})}.
\end{equation}
The branching rule for Schur functions asserts that the coefficients $\pr_{k_1 \to
k_2} (\la \to \mu)$ are non-negative for all $\la$, $\mu$ (see \cite[Chapter
1.5]{M}). Plugging in $x_1=\dots=x_k=1$, we see that $\sum_{\mu \in \GT_{k_2}}
\mathrm{pr}_{k_1 \to k_2} (\la \to \mu)=1$.

Let $0<a_1 \le a_2 \le \dots \le a_n=1$ be fixed positive reals, and let $\rho_{N}$ be a probability measure on $\GT_{N}$.

Let us introduce the probability measure on the set $\GT_{[a_1 N]} \times \GT_{[a_2 N]} \times \dots \times \GT_{[a_n N]}$ via
\begin{equation}
\label{eq:def-projections-general}
\mathrm{Prob} ( \la^{(1)}, \la^{(2)}, \dots, \la^{(n)} ) := \rho_{N} ( \la^{(n)} ) \prod_{i=1}^{n-1} \pr_{[a_{i+1} N] \to [a_i N]} ( \la^{(i+1)} \to \la^{(i)}), \qquad \la^{(i)} \in \GT_{[a_i N]},
\end{equation}
(the fact that all these weights are summed up to $1$ can be straightforwardly checked by induction over $n$.)

We are interested in the \textit{joint} distributions of random signatures of this random array.
For $t =1,2,\dots,n$, let $p_k^{[a_t N]}$ be a (shifted) power sum of coordinates of signatures defined by the formula
$$
p_k^{[a_t N]} := \sum_{i=1}^{[a_t N]} \left( \la_i^{(t)} + [a_t N] -i \right)^k, \qquad k=1,2, \dots,
$$
where $\la^{(t)}$ is distributed according to the measure \eqref{eq:def-projections-general}.

\begin{theorem}
\label{theorem:main-projections}
Assume that $\rho = \{ \rho_{N} \}$ is an appropriate sequence of probability measures on $\GT_{N}$, $N=1,2,\dots$, in the sense of Definition \ref{def:main} and corresponding to functions $F_{\rho}$ and $Q_{\rho}$. Let us consider the probability measure on the sets of signatures defined by \eqref{eq:def-projections-general}. In the notations above,
the collection of random variables
$$
\left\{ N^{-k} \left( p_{k}^{[a_t N]} - \E p_{k}^{[a_t N]} \right) \right\}_{t=1, \dots,n; k \ge 1}
$$
converges, as $N \to \infty$, in the sense of moments, to the Gaussian vector with zero mean and covariance:
\begin{multline*}
\lim_{N \to \infty} \frac{ \cov \left( p_{k_1}^{[a_{t_1} N]}, p_{k_2}^{[a_{t_2} N]} \right)}{N^{k_1+k_2}} = \frac{a_{t_1}^{k_1} a_{t_2}^{k_2} }{(2 \pi \ii)^2} \oint_{|z|=\ep} \oint_{|w| = 2 \ep} \left( \frac{1}{z} +1 + \frac{(1+z) F_{\rho}(1+z)}{a_{t_1}} \right)^{k_1} \\ \times \left( \frac{1}{w} +1 + \frac{(1+w) F_{\rho} (1+w)}{a_{t_2}} \right)^{k_2} Q_{\rho} (z, w) dz dw,
\end{multline*}
where $1 \le t_1 \le t_2 \le n$ and $\ep \ll 1$.
\end{theorem}

\textbf{Example 2)} Let us start with the following classical fact. For $\la, \mu \in \GT_N$ there is a decomposition of the product of two Schur functions into a linear combination of Schur functions:
\begin{equation}
\label{eq_Littlewood_Richardson} s_{\la} (x_1, \dots, x_N) s_{\mu} (x_1, \dots, x_N)
= \sum_{\eta \in \GT_N} c^{\eta}_{\la \mu} s_{\eta} (x_1, \dots, x_N).
\end{equation}
The coefficients $c_{\la \mu}^{\eta}$ are well-known under the name of Littlewood-Richardson coefficients. It is known that
for arbitrary $\la, \mu, \eta$ they are nonnegative (see, e.g., \cite[Chapter 1.9]{M}).


Let $\rho = \{ \rho_N \}$, $\{ \rh^{(1)}_N \}$, $\{ \rh^{(2)}_N \}$, ..., $\{ \rh^{(n-1)}_N \}$ be sequences of appropriate measures, where $\rho_N$, $ \rh^{(1)}_N$, ..., $\rh^{(n-1)}_N $ are probability measures on $\GT_N$.
Let $g_1^{(N)} (\vec{x}), \dots, g_{n-1}^{(N)} (\vec{x})$ be the Schur generating functions of $\rh^{(1)}_N$, $\rh^{(2)}_N$, ..., $\rh^{(n-1)}_N$, respectively.

Define the coefficients $\st_{(g_r)}^{(N)} (\la \to \mu)$, for $\la \in \GT_N$, $\mu \in \GT_N$, $1 \le r \le (n-1)$, via
\begin{equation}
\label{eq:expansion-g}
g_r^{(N)} (\vec{x}) \frac{s_{\la} (\vec{x})}{s_{\lambda} (1^N)} = \sum_{\mu \in \GT_N} \st_{(g_r)}^{(N)} (\la \to \mu) \frac{s_{\mu} (\vec{x})}{s_{\mu} (1^N)}.
\end{equation}
Note that the series in the right-hand side is absolutely convergent and the coefficients $\st_{(g_r)}^{(N)}$ are nonnegative. Using \eqref{eq_Littlewood_Richardson}, one can write an explicit formula for them:
$$
\st_{(g_r)}^{(N)} (\la \to \mu) = \sum_{\eta \in \GT_N} \frac{s_{\mu} (1^N)}{s_{\la} (1^N) s_{\eta} (1^N)} c^{\mu}_{\la \eta} \rh^{(r)}_N (\eta).
$$

Let us define a probability measure on the set
\begin{equation}
\label{eq:many-GTn}
\underbrace{\GT_N \times \GT_N \times \dots \times \GT_N}_{\mbox{$n$ factors}}.
\end{equation}
We define the probability of the configuration
$$
(\la^{(1)}, \la^{(2)}, \dots, \la^{(n)}) \in \GT_N \times \GT_N \times \dots \times \GT_N
$$
via
\begin{equation}
\label{eq:gener-def-multiplications}
\mathrm{Prob} (\la^{(1)}, \la^{(2)}, \dots, \la^{(n)}) := \rho_N (\la^{(n)}) \prod_{i=1}^{n-1} \st_{g_i}^{(N)} (\la^{(i+1)} \to \la^{(i)}).
\end{equation}

Let $p_{k;s}^{(N)}$ be the $k$th shifted power sum of $\la^{(s)}$:
$$
p_{k;s}^{(N)} := \sum_{i=1}^{N} \left( \la_i^{(s)} + N -i \right)^k, \qquad k=1,2, \dots, \qquad \mbox{$(\la^{(1)}, \dots, \la^{(n)})$ is \eqref{eq:gener-def-multiplications}-distributed.}
$$

Let
\begin{equation}
\label{eq:h-func-def-mult}
H_n^{(N)} (\vec{x}) := S_{\rho_N} (\vec{x}), \ \ H_{n-1}^{(N)} (\vec{x}) := H_n^{(N)} (\vec{x}) g_{n-1}^{(N)} (\vec{x}), \ ... \ , \ H_1^{(N)} (\vec{x}) := H_{2}^{(N)} (\vec{x}) g_{1}^{(N)} (\vec{x}).
\end{equation}
It can be directly shown by induction that the functions $H_s^{(N)} (\vec{x})$, are Schur generating functions of $\la^{(s)}$, $s=1, \dots, n$. Moreover, they are appropriate (in the sense of Definition \ref{def:appr-func}) because $g_i^{(N)}$ and $S_{\rho_N}$ are appropriate sequences of functions.

Let us denote the corresponding to $\{ H_s^{(N)} \}_{N \ge 1}$ limit functions from Definition \ref{def:main} by $F_{\rho;(s)} (x)$, $G_{\rho;(s)} (x,y)$, and $Q_{\rho, (s)} (x,y)$.

\begin{theorem}
\label{theorem:main-multiplication}
Assume that $\rho = \{ \rho_N \}$, $\{ \rh^{(1)}_N \}$, $\{ \rh^{(2)}_N \}$, ..., $\{ \rh^{(n-1)}_N \}$ are appropriate sequences of probability measures, and let $g_1(\vec{x})$, $\dots$, $g_{n-1} (\vec{x})$, $F_{\rho,(s)}(\vec{x})$, $Q_{\rho, (s)} (\vec{x})$, $p_{k;s}$, $s=1, \dots, n$, be as above.
Then the collection of random variables
$$
\left\{ N^{-k} \left( p_{k;s}^{(N)} - \E p_{k;s}^{(N)} \right) \right\}_{k \ge 1; s=1, \dots, n}
$$
converges, in the sense of moments, to the Gaussian vector with zero mean and covariance:
\begin{multline} \label{eq_covariance_multiplication}
\lim_{N \to \infty} \frac{\cov \left( p_{k_1;s_1}^{(N)}, p_{k_2;s_2}^{(N)} \right)}{N^{k_1+k_2}} = \frac{1}{(2 \pi \ii)^2} \oint_{|z|=\ep} \oint_{|w| = 2 \ep} \left( \frac{1}{z} +1 + (1+z) F_{\rho, (s_1)} (1+z) \right)^{k_1} \\ \times \left( \frac{1}{w} +1 + (1+w) F_{\rho, (s_2)} (1+w) \right)^{k_2} Q_{\rho, (s_2)} (z, w) dz dw,
\end{multline}
where $1 \le s_1 \le s_2 \le n$ and $\ep \ll 1$.
\end{theorem}

\textbf{Example 3)} Now let us turn to a case which unites the two previous ones. Let $n$ be a positive integer and let $0 < a_1 \le \dots \le a_n$ be reals. Let $\{ \rh^{(1)}_N \}$, $\{ \rh^{(2)}_N \}$, ..., $\{ \rh^{(n-1)}_N \}$, $\{ \rh^{(n)}_N \}$ be appropriate sequences of probability measures such that $\rh^{(i)}_N$ is a probability measure on $\GT_{[a_{i} N]}$.
Let $g_1^{(N)} (x_1, \dots, x_{[a_1 N]})$, $\dots$, $g_{n-1}^{(N)} (x_1, \dots, x_{[a_{n-1} N]})$, $g_{n}^{(N)} (x_1, \dots, x_{[a_n N]})$ be Schur generating functions of $\rh^{(1)}_N$, $ \rh^{(2)}_N $, ..., $ \rh^{(n-1)}_N $, $ \rh^{(n)}_N $, respectively.

Define the coefficients $\st_{(g_r);r+1 \to r}^{(N)} (\la \to \mu)$, for $\la \in \GT_{[a_{r+1} N]}$, $\mu \in \GT_{[a_{r} N]}$, $1 \le r \le (n-1)$, via
\begin{multline}
\label{eq:expansion-g2}
g_{r}^{(N)} (x_1, \dots, x_{[a_{r} N]}) \frac{s_{\la} (x_1, \dots, x_{[a_{r} N]}, 1^{[a_{r+1} N]-[a_{r} N]})}{s_{\lambda} (1^{[a_{r+1} N]})} \\ = \sum_{\mu \in \GT_N} \st_{(g_r);r+1 \to r}^{(N)} (\la \to \mu) \frac{s_{\mu} (x_1, \dots, x_{[a_r N]})}{s_{\mu} (1^{[a_r N]})}.
\end{multline}
Note that the series in the right-hand side is absolutely convergent and the coefficients $\st_{(g_r)}^{(N)}$ are nonnegative. Using \eqref{eq:Schur-branching-pr-coef} and \eqref{eq_Littlewood_Richardson}, one can write an explicit formula for them:
\begin{multline*}
\st_{(g_r);r+1 \to r}^{(N)} (\la \to \mu) = \sum_{\eta_1 \in \GT_{[a_r N]}} \sum_{\eta_2 \in \GT_{[a_r N]}} \frac{s_{\mu} (1^{[a_r N]})}{s_{\eta_1} (1^{[a_{r} N]}) s_{\eta_2} (1^{[a_r N]})} \\ \times c^{\mu}_{\eta_1 \eta_2} \mathrm{pr}_{[a_{r+1} N] \to [a_r N]} (\la \to \eta_1) \rh^{(r)}_N (\eta_2).
\end{multline*}

In this definition we combine two operations on appropriate Schur generating functions which we use
in the two previous cases: The substitution of $1$'s into some variables and multiplication by a
function.

Let us define the probability measure on the set
$$
\GT_{[a_1 N]} \times \GT_{[a_{2} N]} \times \dots \times \GT_{[a_{n-1} N]} \times \GT_{[a_n N]}.
$$

That is, we want to define the probability of a collection of signatures
$$
(\la^{(1)}, \la^{(2)}, \dots, \la^{(n-1)}, \la^{(n)}),
$$
where $\la^{(i)}$ is a signature of length $[a_i N]$, $i=1,2, \dots, n$. Let us do this in the following way.
We define this probability via
\begin{multline}
\label{eq:mes-time-space}
\mathrm{Prob} (\la^{(1)}, \la^{(2)}, \dots, \la^{(n-1)}, \la^{(n)}) := \rho_{[a_n N]} (\la^{(n)}) \prod_{i=1}^{n-1} \st^{(N)}_{(g_i);r+1 \to r} ( \la^{(i+1)} \to \la^{(i)}).
\end{multline}

Let $p_{k;t}^{[a_t N]}$ be the $k$th shifted power sum of $\la^{(t)}$:
$$
p_{k;t}^{[a_t N]} := \sum_{i=1}^{[a_t N]} \left( \la_i^{(t)} + [a_t N] -i \right)^k, \qquad k=1,2, \dots, \qquad \mbox{$(\la^{(1)}, \dots, \la^{(n)})$ is \eqref{eq:mes-time-space}-distributed.}
$$

Let
\begin{align*}
H_n^{(N)} (x_1, \dots, x_{[a_n N]}) := g_n^{(N)} (x_1, \dots, x_{[a_n N]}), \\
H_{n-1}^{(N)} (x_1, \dots, x_{[a_{n-1} N]}) := H_{n}^{(N)} (x_1, \dots, x_{[a_{n-1} N]}, 1^{[a_n N]-[a_{n-1} N]}) g_{n-1}^{(N)} (x_1, \dots, x_{[a_{n-1} N]}), \\
\dots \dots \dots \dots \dots \\
H_{1}^{(N)} (x_1, \dots, x_{[a_1 N]}) := H_{2}^{(N)} (x_1, \dots, x_{[a_2 N]}, 1^{[a_{2} N]-[a_1 N]}) g_{1}^{(N)} (x_1, \dots, x_{[a_{n-1} N]}).
\end{align*}
It can be directly shown by induction that the functions $H_t^{(N)} (x_1, \dots, x_{[a_t N]})$, are Schur generating functions of $\la^{(t)}$, $t=1, \dots, n$. Moreover, they are appropriate (in the sense of Definition \ref{def:appr-func}) because $g_i^{(N)}$ and $S_{\rho_N}$ are appropriate sequences of functions. Let us denote the corresponding limit functions by $F_{\rho;(t)} (x)$, $G_{\rho;(t)} (x,y)$, $Q_{\rho, (t)} (x,y)$, for $t=1,2, \dots, n$.

\begin{theorem}
\label{th:general-for-domino}
In the notations above, the collection of random functions
$$
\left\{ N^{-k} \left( p_{k;t}^{[a_t N]} - \E p_{k;t}^{[a_t N]} \right) \right\}_{t=1, \dots, n; k \in \N}
$$
is asymtoticaly Gaussian with the limit covariance
\begin{multline*}
\lim_{N \to \infty} \frac{ \cov \left( p_{k_1;t_1}^{[a_{t_1} N]}, p_{k_2;t_2}^{[a_{t_2} N]} \right) }{N^{k_1+k_2} } = \frac{a_{t_1}^{k_1} a_{t_2}^{k_2} }{(2 \pi \ii)^2} \oint_{|z|=\ep} \oint_{|w| = 2 \ep} \left( \frac{1}{z} +1 + (1+z) F_{\rho, (t_1)} (1+z) \right)^{k_1} \\ \times \left( \frac{1}{w} +1 + (1+w) F_{\rho, (t_2)} (1+w) \right)^{k_2} Q_{\rho, (t_2)} (z, w) dz dw,
\end{multline*}
where $1 \le t_1 \le t_2 \le s$ and $\ep \ll 1$.
\end{theorem}


\section{Applications}

\label{Section_Applications}

In this section we state several applications of general theorems from Sections
\ref{sec:2-2} and \ref{Section_multilevel_specified}. The theorems are split into
two parts: Sections \ref{Section_Tensor_statement}, \ref{Section_extreme_characters}
are devoted to problems of asymptotic representation theory, while Sections
\ref{sec:restrictions}, \ref{Section_Aztec_statement}, \ref{Section_noncol} deal
with 2d lattice models of statistical mechanics.

\subsection{Preliminary definitions}
\label{sec:statem-requir-formulae}

In a considerable part of our theorems an input is given by a sequence
$\lambda(N)\in \GT_N$, $N=1,2,\dots$ of signatures. Depending on the context, they
encode irreducible representations, boundary conditions in statistical mechanics
models or initial conditions of Markov chains.

In our asymptotic results we are going to make the following technical assumption on
the behavior of $\lambda(N)$ as $N$ becomes large.
\begin{definition}
\label{definition_regularity}
 A sequence of signatures $\lambda(N) \in \GT_N$ is called \emph{regular}, if there exists a
 piecewise--continuous function $f(t)$ and a constant $C$ such that
\begin{equation}
\label{eq_reg1}
 \lim_{N\to\infty} \frac{1}{N}\sum_{j=1\dots,N} \left|\frac{\lambda_j(N)}{N}-f(j/N)\right|=0
\end{equation}
and
\begin{equation}
\label{eq_reg2} \left|\frac{\lambda_j(N)}{N}-f(j/N)\right|<C,\quad \quad
j=1,\dots,N,\quad N=1,2,\dots.
\end{equation}
\end{definition}
\begin{remark}
Informally, the condition \eqref{eq_reg1} means that scaled by $N$ coordinates of
$\lambda(N)$ approach a limit profile $f$. The restriction that $f(t)$ is
piecewise--continuous is reasonable, since $f(t)$ is a limit of monotonous functions
and, thus, is monotonous (therefore, we only exclude the case of countably many
points of discontinuity for $f$). This restriction originates in the asymptotic
results of \cite{GP} and we believe that it, in fact, can be weakened for most
applications, cf.\  \cite{Novak}, \cite{Novak2}.
\end{remark}

It is clear that if the sequence $\la(N)$ is regular, then the sequence $m[\la(N)]$
(defined by \eqref{eq_signature_measure}) weakly converges to a probabilistic
measure on $\R$ with compact support. The complete information about such measure
can be encoded in several generating functions that we now define.

For a probability measure $\mes$ on $\R$ with compact support let us define the
Cauchy-Stieltjes transform $C_{\mes} (z)$ by

\begin{equation}
\label{eq:def-sm-function}
C_{\mes} (z) := \int_{\R} \frac{d \mes (x)}{z - x} = z^{-1} + z^{-2} \int_{\R} x d \mes (x) + z^{-3} \int_{\R} x^2 d \mes (x) + \dots.
\end{equation}
This is a power series in $z^{-1}$ which converges in a neighborhood of infinity.

Define $C_\mes^{(-1)}(z)$ to be the inverse series to $C_\mes(z)$, i.e.\ such that
 $$
  C_\mes^{(-1)}\bigl( C_\mes(z)\bigr)= C_\mes\bigl(C_\mes^{(-1)}(z)\bigr)=z,
 $$
(As a power series $C_\mes^{(-1)}(z)$ has a form $\frac{1}{z} + a_0 + a_1 z+ a_1 z^2
+ a_2 z^3+ \dots$). Further, set
\begin{equation}
\label{eq_Voiculescu_R}
R_\mes(z)= C_\mes^{(-1)}(z) -\frac{1}{z}.
\end{equation}
 The function $R_\mes(z)$ is
well-known in the free probability theory under the name of \emph{Voiculescu $R$--transform}, cf.\
\cite{VDN}, \cite{NS}.

Integrating $R_\mes(z)$, set
\begin{equation*}
H_\mes(z) := \int_0^{\ln(z)}R_\mes(t)dt+\ln\left(\frac{\ln(z)}{z-1}\right),
\end{equation*}
which should be understood as a holomorphic function in a neighborhood of $z=1$.

The derivative of $H_{\mes} (z)$ has a simpler form:
\begin{equation}
\label{eq:H-S-connection}
H'_{\mes} (z) = \frac{C_\mes^{(-1)}(\log(z))}{z} - \frac{1}{z-1}.
\end{equation}
The function $H'_{\mes} (z)$ plays an important role in the context of the
\textit{quantized free convolution}, see \cite{BG}.

\subsection{Asymptotic decompositions of representations of $U(N)$}
\label{Section_Tensor_statement}

Here we briefly recall some facts about representations of the unitary group (see
e.g. \cite{FH}, \cite{Weyl-book}, \cite{Zh}) and state a central limit theorem for
decompositions of their tensor products and restrictions.

Let $U(N)$ be the group of all unitary $N \times N$ matrices. It is a classical fact
that the irreducible representations of $U(N)$ are parameterized by signatures of
length $N$. Let us denote by $\pi^{\la}$ the irreducible representation of $U(N)$
corresponding to the signature $\la$ ($\la$ is the \textit{highest weight} of this
representation), and let $\dim(\la)$ denote the dimension of this representation.


\smallskip

Consider a \emph{reducible} finite-dimensional representation $T_N$ of  $U(N)$ and
let
$$
T_N = \bigoplus_{\la \in \GT_N} c_{\la} \pi^{\la}
$$
be a decomposition of $T_N$ into irreducibles.

One of the basic ideas of asymptotic representation theory is to associate with
$T_N$ a \textit{probability measure} on the set of labels of irreducible
representations. In the case of the unitary group this results into the definition
of the probability measure $\rho_{T_N}$:
\begin{equation}
\label{eq:repr-meas-sign-def} \rho_{T_N} (\la) := \frac{c_{\la} \dim
(\pi^{\la})}{\dim (T_N)}, \qquad \la \in \GT_N.
\end{equation}
We reduce the study of the asymptotic behavior of such probability measures to their
moments defined through
\begin{equation}
\label{eq_moments_of_decomposition} p_k^{T_N} := \sum_{i=1}^N (\la_i+N-i)^k, \qquad
k=1,2, \dots, \qquad \mbox{$(\la_1,\dots, \la_N)$ is $\rho_{T(N)}$-distributed}.
\end{equation}

One basic operation which creates reducible representations is tensor product. The
decomposition of the (Kronecker) tensor product $\pi^\lambda\otimes \pi^\mu$ into
irreducibles can be written with the use of classical \textit{Littlewood-Richardson
coefficients} $c^{\eta}_{\la \mu}$:
$$
\pi^{\la} \otimes \pi^{\mu} = \bigoplus_{\eta \in \GT_N} c^{\eta}_{\la \mu}
\pi^{\eta}, \qquad \la, \mu \in \GT_N,
$$
with an equivalent definition being \eqref{eq_Littlewood_Richardson}. The Law of
Large Numbers for tensor products was proven in \cite{BG}, and here is the Central
Limit Theorem.





For two probability measures $\mes^1$ and $\mes^2$ with compact support set
\begin{multline}
\label{eq:def-Q-function}
Q_{\mes^1, \mes^2}^{\otimes} (x,y) := \pa_x \pa_y \left( \log \left( 1 - x y \frac{(1+x) H'_{\mes^1}(1+x) - (1+y) H'_{\mes^1}(1+y)}{x-y} \right) \right. \\ \left. + \log \left( 1 - x y \frac{(1+x) H'_{\mes^2}(1+x) - (1+y) H'_{\mes^2}(1+y)}{x-y} \right) \right) + \frac{1}{(x-y)^2}.
\end{multline}

\begin{theorem}[Central Limit Theorem for tensor products]
\label{theorem:tensor}
 Suppose that $\lambda^1(N),\lambda^2 (N) \in \GT_N$, $N=1,2,\dots$, are regular sequences of signatures
 such that
 $$
  \lim_{N\to\infty} m[\lambda^i(N)]=\mes^i, \quad i=1, 2, \qquad \mbox{weak convergence.}
 $$
 Let
$T_N=\pi^{\lambda^1(N)}\otimes \pi^{\lambda^2(N)}$. Then, as $N\to\infty$, the
random vector of moments \eqref{eq_moments_of_decomposition}
$$
\left\{ N^{-k} \left( p_k^{T_N} - \E p_k^{T_N} \right) \right\}_{k \ge 1}
$$
converges, in the sense of moments, to the Gaussian vector with zero mean and covariance
\begin{multline}
\label{eq:th-tensor}
 \lim_{N \to \infty} \frac{\cov \left( p_{k_1}^{T_N}, p_{k_2}^{T_N} \right)}{N^{k_1+k_2}} \\ = \frac{1}{(2 \pi \ii)^2} \oint_{|z|=\ep} \oint_{|w|=2 \ep} \left( \frac{1}{z} +1 + (1+z) \left( H'_{\mes^1} (1+z)+H'_{\mes^2} (1+z) \right) \right)^{k_1} \\ \times \left( \frac{1}{w} + 1 + (1+w) \left( H'_{\mes^1} (1+w)+ H'_{\mes^2} (1+w) \right) \right)^{k_2} Q_{\mes^1, \mes^2}^{\otimes} (z, w) dz dw,
 \end{multline}
where $\ep \ll 1$, function $H'_{\mes}$ was defined in Section
\ref{sec:statem-requir-formulae}, and $Q_{\mes^1, \mes^2}^{\otimes}$ is defined in
\eqref{eq:def-Q-function}.
\end{theorem}
\begin{remark}
In this setting the operation of the tensor product of representations can be seen
as a quantization of the summation of independent random matrices. The degeneration
from representations to matrices is known as a \textit{semiclassical limit}, see
e.g. \cite[Section 1.3]{BG} and references therein for details. Under this limit
transition Theorem \ref{theorem:tensor} turns into the result for the spectra of the
sum of the Haar-distributed random Hermitian matrices with a fixed spectrum.  In
Section \ref{sec:9-4} we show that in this limit the covariance \eqref{eq:th-tensor}
turns into the covariance for the random matrix problem, which can be found in
\cite{PS}.
\end{remark}

\begin{remark}
In a similar way one can prove a central limit theorem for decomposition of
$\pi^{\la^{1}} \otimes \pi^{\la^{2}} \otimes \dots \otimes \pi^{\la^{s}}$ for
arbitrary positive integer $s$.
\end{remark}

\begin{remark}
There is an approach to decomposition of tensor products via
\textit{Perelomov-Popov} measures, see \cite{BG} for details. In this setting, one
obtains a direct relation of these measures and \textit{free probability}. It would
be interesting to relate Theorem \ref{theorem:tensor} and the concept of
\textit{second-order freeness} developed in \cite{MS}, \cite{MSS}.
\end{remark}

\begin{proof}[Proof of Theorem \ref{theorem:tensor}] Given that the character of $\pi^{\lambda}$ is precisely the Schur function
$s_\lambda$, and that taking tensor products corresponds to multiplying the
characters, Theorem \ref{theorem:tensor} is an immediate corollary of Theorem
\ref{theorem:main-one-level} and Proposition \ref{prop:Schur-are-approp}.
\end{proof}

We believe that Theorem \ref{theorem:tensor} is new. Yet, there are simpler tensor
products whose asymptotic decomposition were intensively studied before in the
context of the Schur--Weyl duality, cf.\ \cite{Biane},\cite{Meliot}. For that
consider a representation $W_{N,n}$ of $U(N)$ in vector space $(\C^N)^{\otimes n}$
via $ g (v_1 \otimes v_2 \otimes \dots \otimes v_n) = g(v_1) \otimes g(v_2) \otimes
\dots \otimes g(v_n)$, $g \in U(N)$. The decomposition of $W_{N,n}$ into
irreducibles is governed by the \emph{Schur--Weyl measure}, while its $N\to\infty$
limit (when $n$ is kept fixed) is the celebrated \emph{Plancherel measure} of the
symmetric group $S(n)$.

\begin{theorem} [Central Limit Theorem for Schur--Weyl measures] \label{Theorem_Schur_Weyl}
 Assume that $n=\lfloor c N^2\rfloor$ for $c>0$ and let $T_N=W_{N,n}$,
 $N=1,2,\dots$. Then, as $N\to\infty$, the
random vector of moments \eqref{eq_moments_of_decomposition}
$$
\left\{ N^{-k} \left( p_k^{T_N} - \E p_k^{T_N} \right) \right\}_{k \ge 1}
$$
converges, in the sense of moments, to the Gaussian vector with zero mean and
covariance
\begin{multline}
\label{eq:th-Shur-Weyl}
 \lim_{N \to \infty} \frac{\cov \left( p_{k_1}^{T_N}, p_{k_2}^{T_N} \right)}{N^{k_1+k_2}} = \frac{1}{(2 \pi \ii)^2} \oint_{|z|=\ep} \oint_{|w|=2 \ep} \left( \frac{1}{z} +1 + c + cz \right)^{k_1} \\ \times \left( \frac{1}{w} + 1 + c + cw \right)^{k_2} \left( -c + \frac{1}{(z-w)^2} \right) dz dw,
 \end{multline}
where $\ep \ll 1$.
\end{theorem}
\begin{proof}
 It is easy to see that the
Schur generating function of $T_N$ is given by the normalized character of this
representation:
$$
S_{\rho_{N,n}} (x_1, \dots, x_N) = \frac{(x_1+x_2+ \dots + x_N)^n}{N^n}.
$$
 We have
$$
\lim_{N \to \infty} \frac{\pa_1 \log S_{\rho_{N,n}} (x_1, 1^{N-1}) }{N} = \lim_{N
\to \infty} \frac{c N^2 \pa_1 \left[ \log \left( \frac{x_1}{N} + \frac{N-1}{N}
\right) \right]}{N} = c,
$$
$$
\lim_{N \to \infty} \pa_1 \pa_2 \log S_{\rho_{N,n}} (x_1, x_2, 1^{N-2}) = \lim_{N
\to \infty} \pa_1 \pa_2 \left[ c N^2 \log \left( \frac{x_1}{N} + \frac{x_2}{N} +
\frac{N-2}{N} \right) \right] = -c.
$$
It remains to use Theorem \ref{theorem:main-one-level}. \end{proof} An earlier proof
of Theorem \ref{Theorem_Schur_Weyl} is given in \cite{Meliot}, while its $c\to 0$
version is the Kerov's Central Limit Theorem for the Plancherel measure, see
\cite{Kerov}, \cite{Ivanov_Olsh}.

\medskip

Another natural operation on representations of $U(N)$ is \emph{restriction} onto
the subgroup $U(M)\subset U(N)$, where $U(M)$ is identified with the subgroup of
$U(N)$ fixing the last $N-M$ coordinate vectors.

\begin{theorem}[Central Limit Theorem for restrictions]
\label{theorem_restrictions}
 Suppose that $\lambda(N) \in \GT_N$, $N=1,2,\dots$, is a regular sequence of signatures
 such that
 $$
  \lim_{N\to\infty} m[\lambda(N)]=\mes, \quad i=1, 2, \qquad \mbox{weak convergence.}
 $$
 Take $0<a<1$ and let
$T_N$ be a representation of $U(\lfloor \alpha N \rfloor)$ given by
$T_N=\pi^{\lambda(N)}\large|_{U(\lfloor \alpha N \rfloor)}$. Then, as $N\to\infty$,
the random vector of moments \eqref{eq_moments_of_decomposition}
$$
\left\{ N^{-k} \left( p_k^{T_N} - \E p_k^{T_N} \right) \right\}_{k \ge 1}
$$
converges, in the sense of moments, to the Gaussian vector with zero mean and
covariance
\begin{multline}
\lim_{N \to \infty} \frac{\cov \left( p_{k_1}^{T_N}, p_{k_2}^{T_N}
\right)}{N^{k_1+k_2}}  \\ = \frac{a^{k_1+k_2}}{(2 \pi \ii)^2} \oint_{|z|=\ep}
\oint_{|w|=2 \ep} \left( \frac{1}{z} +1 + \frac{(1+z) \mathbf H'_{\mes} (1+z)}{a}
\right)^{k_1}
\left( \frac{1}{w} +1 + \frac{(1+w) \mathbf H'_{\mes} (1+w)}{a} \right)^{k_2} \\
\times \left( \pa_z \pa_w \left[ \log \left( 1 - z w \frac{ (1+z) \mathbf H'_{\mes}
(1+z) - (1+w) \mathbf H'_{\mes} (1+w)}{z-w} \right) \right] + \frac{1}{(z-w)^2}
\right) dz dw,
\end{multline}

where $\ep \ll 1$ and function $H'_{\mes}$ was defined in Section
\ref{sec:statem-requir-formulae}.
\end{theorem}
Theorem \ref{theorem_restrictions} is a particular case of Theorem
\ref{th:ctl-proj-GFF-1}, where we also present a more elegant formula for the
limiting covariance, expressing it in terms of a section of the \emph{Gaussian Free
Field}, which we define next.

\subsection{Preliminaries: $2d$ Gaussian Free Field}
A {\it Gaussian family} is a collection of Gaussian random variables $\{ \xi_a \}_{a \in \Upsilon}$
indexed by an arbitrary set $\Upsilon$. We assume that all our random variables are centered, i.e.
\begin{equation*}
\mathbf E \xi_a = 0, \qquad \mbox{ for all } a \in \Upsilon.
\end{equation*}
Any Gaussian family gives rise to a \emph{covariance kernel}
$\Cov : \Upsilon \times \Upsilon \to \mathbb R$ defined by
\begin{equation*}
\Cov (a_1, a_2) = \mathbf E ( \xi_{a_1} \xi_{a_2} ).
\end{equation*}

Assume that a function $\tilde C : \Upsilon \times \Upsilon \to \mathbb R$
is such that for any $n\ge 1$ and $a_1, \dots, a_n \in \Upsilon$,
$[\tilde C (a_i,a_j)]_{i,j=1}^{n}$ is a symmetric and positive-definite matrix.
Then (see e.g. \cite{Car}) there exists a centered Gaussian family with the covariance
kernel $\tilde C$.

Let $\mathbb H := \{ z \in \mathbb C : \mathfrak I (z) >0 \}$ be the upper half-plane, and
let $C_0^\infty$ be the space of smooth real--valued compactly supported test functions on
$\mathbb H$.
Let us set
\begin{equation*}
\tilde G(z,w):= -\frac{1}{2 \pi} \ln \left| \frac{z-w}{z - \bar w} \right|, \qquad z,w \in \mathbb H,
\end{equation*}
and define a covariance kernel $C : C_0^\infty \times C_0^\infty \to \mathbb R$ via
\begin{equation*}
C (f_1, f_2) := \int_{\mathbb H} \int_{\mathbb H} f_1 (z) f_2 (w) \tilde G(z,w)
dz d \bar z dw d \bar w.
\end{equation*}

The \emph{Gaussian Free Field} (GFF) $\mathfrak G$ on $\mathbb{H}$ with zero boundary conditions can be
defined as a Gaussian family $\{ \xi_f \}_{ f \in C_0^\infty}$ with covariance kernel $C$.
The field $\mathfrak G$ cannot be defined as a random function on $\mathbb H$,
but one can make sense of the integrals $\int f(z) \mathfrak G(z) dz$ over finite contours
in $\mathbb{H}$ with continuous functions $f(z)$, see \cite{She}, \cite[Section 4]{D}, \cite[Section 2]{HMP} for more details.

In our results GFF will play a role of the universal limit object for
two-dimensional fluctuations of probabilistic models under consideration. In this
sense, GFF plays a similar role to Brownian motion and Gaussian distribution.

\subsection{Extreme characters of $U(\infty)$}
\label{Section_extreme_characters}

In this section we switch from $U(N)$ to its infinite--dimensional version. Consider
the tower of embedded unitary groups
\begin{equation*}
U(1) \subset U(2) \subset \dots \subset U(N) \subset U(N+1) \subset \dots, \qquad
U(N) = \{ u_{ij} \}_{i,j=1}^N,
\end{equation*}
where $U(N)$ is embedded into $U(N+1)$ as the subgroup fixing the last coordinate
vector. \emph{The infinite--dimensional unitary group} is the inductive limit of
these groups:
$$
U(\infty) := \bigcup_{N=1}^{\infty} U(N).
$$

Define a \emph{character} of the group $U(\infty)$ as a \emph{continuous} function
$\chi : U(\infty) \to \mathbb C$ that satisfies the following conditions:

\begin{itemize}
\item  $\chi(e)=1$, where $e$ is the identity element of $U(\infty)$ (normalization);

\item $\chi(g h g^{-1}) = \chi (h)$, where $g,h$ are any elements of $U(\infty)$
(centrality);

\item $[\chi( g_i g_j^{-1})]_{i,j=1}^n$ is an Hermitian and positive-definite matrix
for any $n\ge 1$ and $g_1, \dots, g_n \in U(\infty)$ (positive-definiteness);
\end{itemize}

The space of characters of $U(\infty)$ is obviously convex. The extreme points of
this space are called \emph{extreme} characters; they replace characters of
irreducible representations in this infinite-dimensional setting. The classification
of the extreme characters of $U(\infty)$ is known as the Edrei--Voiculescu theorem
(see \cite{Vo}, \cite{Edr}, \cite{VK}, \cite{OkoOls98}, \cite{BorOls}). It turns out
that the extreme characters can be parameterized by the set $\Omega=(\alpha^+,
\alpha^-, \beta^+,\beta^-, \gamma^+, \gamma^-)$, where
\begin{gather*}
\alpha^{\pm} = \alpha_1^{\pm} \ge \alpha_2^{\pm} \ge \dots \ge 0, \\
\beta^{\pm} = \beta_1^{\pm} \ge \beta_2^{\pm} \ge \dots \ge 0, \\
\gamma^{\pm} \ge 0, \ \ \ \sum_{i=1}^{\infty} (\alpha_i^{\pm}+\beta_i^{\pm}) \le
\infty, \ \ \ \beta_1^+ + \beta_1^- \le 1.
\end{gather*}

Each $\omega \in \Omega$ gives rise to a function $\Phi^{\omega} : \{ u \in \mathbb
C : |u|=1 \} \to \mathbb C$ via
\begin{equation}
\label{eq:Voic-formula} \Phi^{\omega} (u) := \exp( \gamma^+ (u-1) + \gamma^{-}
(u^{-1}-1) ) \prod_{i=1}^{\infty} \frac{(1+\beta_i^+ (u-1))}{(1-\alpha_i^+ (u-1))}
\frac{(1+\beta_i^- (u^{-1}-1))}{(1- \alpha_i^- (u^{-1}-1))}.
\end{equation}
Then the extreme character of $U(\infty)$ corresponding to $\omega \in \Omega$ is
$\chi^{\omega}$ given by
\begin{equation*}
\chi^{\omega} (U) := \prod_{u \in Spectrum(U)} \Phi^{\omega} (u), \qquad U \in
U(\infty),
\end{equation*}
(this product is essentially finite, because only finitely many of $u$'s are distinct from $1$).

\smallskip

Each character gives rise to a probabilistic object known as the \emph{central
measure on the Gelfand--Tsetlin} graph. Let us present the necessary definitions.

The \emph{Gelfand-Tsetlin graph} $\mathbb {GT}$ is defined by specifying its set of
vertices as $\bigcup_{N=0}^{\infty} \mathbb{GT}_N $ and putting an edge between any
two signatures $\lambda\in\GT_N$ and $\mu\in\GT_{N-1}$ such that they
\emph{interlace} $\mu \prec \lambda$, which means
$$
 \lambda_1\ge\mu_1\ge\lambda_2\ge\dots\ge \mu_{N-1}\ge\lambda_N.
$$
We agree that $\GT_0$ consists of a single empty signature $\varnothing$ joined by an
edge with each vertex of $\GT_1$. A \emph{path} between signatures $\kappa \in
\mathbb {GT}_K$ and $\upsilon \in \mathbb {GT}_N$, $K<N$, is a sequence
\begin{equation*}
\kappa = \lambda^{(K)} \prec \lambda^{(K+1)} \prec \dots \prec \lambda^{(N)} =
\upsilon, \qquad  \lambda^{(i)} \in \mathbb{GT}_i,\quad K\le i\le N.
\end{equation*}

An \emph{infinite path} is a sequence
\begin{equation*}
\varnothing \prec \lambda^{(1)} \prec \lambda^{(2)} \prec \dots \prec \lambda^{(k)}
\prec \lambda^{(k+1)} \prec \dots.
\end{equation*}
We denote by $\mathcal P_N$ the set of all paths starting in $\varnothing$ and of
length $N$. We denote by $\mathcal P$ the set of all infinite paths.

\smallskip

For any character $\chi$ of $U(\infty)$ one can associate a probability measure on paths $\mathcal
P$. Indeed, for any fixed $N$ let us define a probability measure $M_N^{\chi}$ on $\GT_N$ via the
linear decomposition
\begin{equation*}
{\chi |}_{U(N)} = \sum_{\lambda \in \mathbb{GT}_N} M_N^{\chi} (\lambda)
\frac{s_{\la} (u_1, \dots, u_N)}{s_{\la} (1^N)}.
\end{equation*}
Next, define a weight of a subset of $\mathcal P$ consisting of all paths
with prescribed members up to $\mathbb{GT}_N$ by
\begin{equation}
\label{mera-puti} P^{\chi} ( \lambda^{(1)}, \lambda^{(2)}, \dots, \lambda^{(N)} ) =
\frac{ M_N^{\chi} (\lambda^{(N)})}{ s_{\la} (1^N) }.
\end{equation}
Note that this weight depends on $\lambda^{(N)}$ only. It can be easily deduced from the branching
rules for characters of $U(N)$ that this definition is consistent and correctly defines a
probability measure $\mu_{\chi}$ on $\mathcal P$.

We will analyze the asymptotics of probability measures corresponding to certain
sequences of extreme characters
$$
\omega (N) = \{ \{ \al_{i}^+ (N) \}_{i \ge 1}, \{ \al_i^- (N) \}_{i \ge 1}, \{
\beta_i^+ (N) \}_{i \ge 1}, \{ \beta_i^- (N) \}_{i \ge 1}, \gamma^+(N), \gamma^-(N)
\}.
$$
In more detail, we will assume that a sequence $\omega (N)$ satisfies the following
condition.

\textbf{Condition.} We will consider sequences $\omega(N)$ such that, as $N \to
\infty$, we have
\begin{multline}
\label{eq:ext-basic-cond} \frac{1}{N} \sum_{i \ge 1} \delta (\al_i^+ (N)) \to
\mathcal A^+, \ \ \frac{1}{N} \sum_{i \ge 1} \delta (\beta_i^+ (N)) \to \mathcal
B^+, \ \ \lim_{N \to \infty} \frac{\gamma^+(N)}{N} = \Gamma^+, \\ \frac{1}{N}
\sum_{i \ge 1} \delta (\al_i^- (N)) \to \mathcal A^-, \ \ \frac{1}{N} \sum_{i \ge 1}
\delta (\beta_i^- (N)) \to \mathcal B^-, \ \ \lim_{N \to \infty} \frac{\gamma^-
(N)}{N} = \Gamma^-,
\end{multline}
where $\mathcal A^+$, $\mathcal A^-$, $\mathcal B^+$, $\mathcal B^-$ are arbitrary
finite (not necessarily probability) measures on $\R$ with compact support,
$\Gamma^+, \Gamma^-$ are two positive real numbers, and we consider the convergence
of finite measures in the weak sense. We will denote by $\mathbf{J}$ the sextuple
$(\mathcal A^+$, $\mathcal A^-$, $\mathcal B^+$, $\mathcal B^-, \Gamma^+, \Gamma^-)$
which consists of 4 finite measures and 2 real numbers.

\smallskip

A direct computation shows that if a sequence $\omega(N)$ satisfies the condition
\eqref{eq:ext-basic-cond}, then we have the following convergence of the Voiculescu functions
\eqref{eq:Voic-formula}
\begin{equation}
\label{eq:ext-main-cond} \lim_{N \to \infty} \frac{\pa_z \log \Phi^{\omega (N)}
(1+z)}{N} = \F(1+z), \qquad \mbox{uniformly in $|z|<\ep, \ \ep >0$},
\end{equation}
where $\F = \F_{\mathbf J}$ is determined by $\mathbf J$ with the use of the formula
\eqref{eq:limit-Voiculescu}; we do not need the explicit formula for it at this
moment.

The description of CLT for extreme characters involves the  following functions.
\begin{proposition}
\label{prop:comp-struct-ext} Let $\F(z) = \F_{\mathbf J} (z)$ be the function which
is obtained in the limit \eqref{eq:ext-main-cond}. For any $y \in \R$ and $\eta>0$
the equation
$$
\frac{1}{z} +1 + \frac{(1+z) \F (1+z)}{\eta} = \frac{y}{\eta}
$$
has at most one root $ z \in \HH$. Let $\mathbf D_{\F} \in \R^2$ be the set of pairs
$(y,\eta)$ such that this root exists. Then the map $\mathbf D_{\F} \to \HH$ from
such a pair to such a root is a diffeomorphism.
\end{proposition}
We prove this proposition in Section \ref{sec:ext-GFF-proof}.

Let $z \to (y_{\F} (z), \eta_{\F} (z))$ be an inverse of the map given by Proposition \ref{prop:comp-struct-ext}.
 Proposition \ref{prop:comp-struct-ext} introduces coordinates in which the
fluctuations of extreme characters become a Gaussian Free Field.

In order to make this statement precise, let us introduce the \textit{height function} $H_N:
\mathbb R \times \mathbb R_{\ge 1} \times \mathcal P \to \mathbb Z_{\ge 0}$ given by the formula
\begin{equation} \label{eq_Height_function}
H_N (y,\eta, \{ \lambda^{(j)} \}_{j \ge 1} ) := \left|\left\{ 1 \le i \le \lfloor N
\eta\rfloor : \lambda_i^{(N\eta)} + \lfloor N \eta\rfloor - i \ge N y
\right\}\right|,
\end{equation}
where $\lambda_i^{(N\eta)}$ are the coordinates of the signature from $\GT_{\lfloor
N \eta \rfloor}$ in the path which belongs to $\mathcal P$.

Let us equip $\mathcal P$ with a probability measure $\mu_{\chi^\omega}$, where
$\omega = \omega(N)$ satisfies the condition \eqref{eq:ext-main-cond}. Then $H_N
(y,\eta):= H_N(y,\eta,\cdot)$ becomes a random function which describes a certain
random stepped surface.

Let us carry $H_N(y,\eta)$ over to $\mathbb H$ through
$$
H_N(z) := H_N ( y_{\F}(z), \eta_{\F} (z)), \qquad z \in \mathbb H.
$$
One might worry that some information is lost in this transformation, as the image
of the map $z\to( y_{\F}(z), \eta_{\F} (z))$ is smaller than $\mathbb R \times
\mathbb R_{\ge 0}$, yet the configuration is actually \emph{frozen} outside this
image and there are no fluctuations to study, cf.\ Figures \ref{Fig_Lozenge_1},
\ref{Fig_Domino_1}, where random tilings are frozen outside inscribed circles.

For $\eta>0$ and $k=1,2,\dots$ define a moment of the random height function as
$$
M_{\eta,k}^{\omega(N)} = \int_{-\infty}^{+\infty} y^k \left( H_N ( y, \eta) -
\mathbf E H_N ( y, \eta) \right) dy.
$$
Also define the corresponding moment of GFF via
$$
\mathcal M_{\eta,k}^{\F} = \int_{z \in \mathbb H\mid  \eta_{\F} (z)=\eta}
y_{\F}(z)^k \mathfrak G(z) \frac{d y_{\F}(z)}{dz} dz.
$$

\begin{theorem}[Central Limit Theorem for extreme characters]
\label{theorem:extr-char-gff-1}

Assume that the sequence of extreme characters $\omega(N)$ satisfies condition
\eqref{eq:ext-basic-cond}. Let $H_N (z)$ be a random height function on $\mathbb H$
corresponding to $\omega(N)$ as above. Then
$$
\sqrt{\pi}\left( H_N(z) - \mathbf E H_N (z) \right)  \xrightarrow[N \to \infty]{}
\mathfrak G (z).
$$
In more details, as $N \to \infty$, the collection of random variables $\{
\sqrt{\pi} M_{A,k}^{\omega(N)} \}_{A>0; k \in \mathbb Z_{\ge 0}}$ converges, in the
sense of moments, to $\{ \mathcal M_{A,k}^{\F} \}_{A>0; k \in \mathbb Z_{\ge 0}}$.
\end{theorem}
\begin{remark}
 For explicit expressions for the covariance of $\{ \mathcal M_{A,k}^{\F} \}_{A>0; k
\in \mathbb Z_{\ge 0}}$ see \eqref{eq:extr-gauss-proof}.
\end{remark}
\begin{remark}
The condition \eqref{eq:ext-basic-cond} for the growth of extreme characters was
introduced and studied in \cite{BBO}, where the law of large numbers for this
probabilistic model was proven. Among other connections, the condition
\eqref{eq:ext-basic-cond} is related to the hydrodynamical limit of random surfaces
related to probabilistic particle systems with local interaction, see Section 3.3 of
\cite{BBO} for more details.
\end{remark}
The proof of Theorem \ref{theorem:extr-char-gff-1} is given in Section
\ref{sec:ext-GFF-proof}. We believe that this statement is new for general extreme
characters. For the very special case when the only non-zero parameter in
\eqref{eq:ext-basic-cond} is $\Gamma^+$ it was previously proven in \cite{BF},
\cite{Borodin_Bufetov}.

\smallskip

The paths of $\mathcal P_N$ and $\mathcal P$ can be identified with \emph{lozenge
tilings}, which leads us to statistical mechanics applications.

\subsection{Lozenge tilings}
\label{sec:restrictions}

Consider a (right) halfplane on the regular triangular lattice. We would like to
\emph{tile} this halfplane with lozenges (rhombuses) of three types: horizontal
{\scalebox{0.16}{\includegraphics{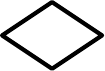}}}, and two others
{\scalebox{0.16}{\includegraphics{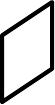}}},
{\scalebox{0.16}{\includegraphics{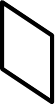}}}. Let $\widehat{\mathcal P}$
denote the set of complete tilings of the half--plane subject to two boundary
conditions: the lozenges become
{\scalebox{0.16}{\includegraphics{lozenge_v_up.pdf}}} as one goes far up and
{\scalebox{0.16}{\includegraphics{lozenge_v_down.pdf}}} as one goes far down, see
Figure \ref{Fig_Lozenge_extreme}.

\begin{figure}[t]
\begin{center}
 {\scalebox{0.7}{\includegraphics{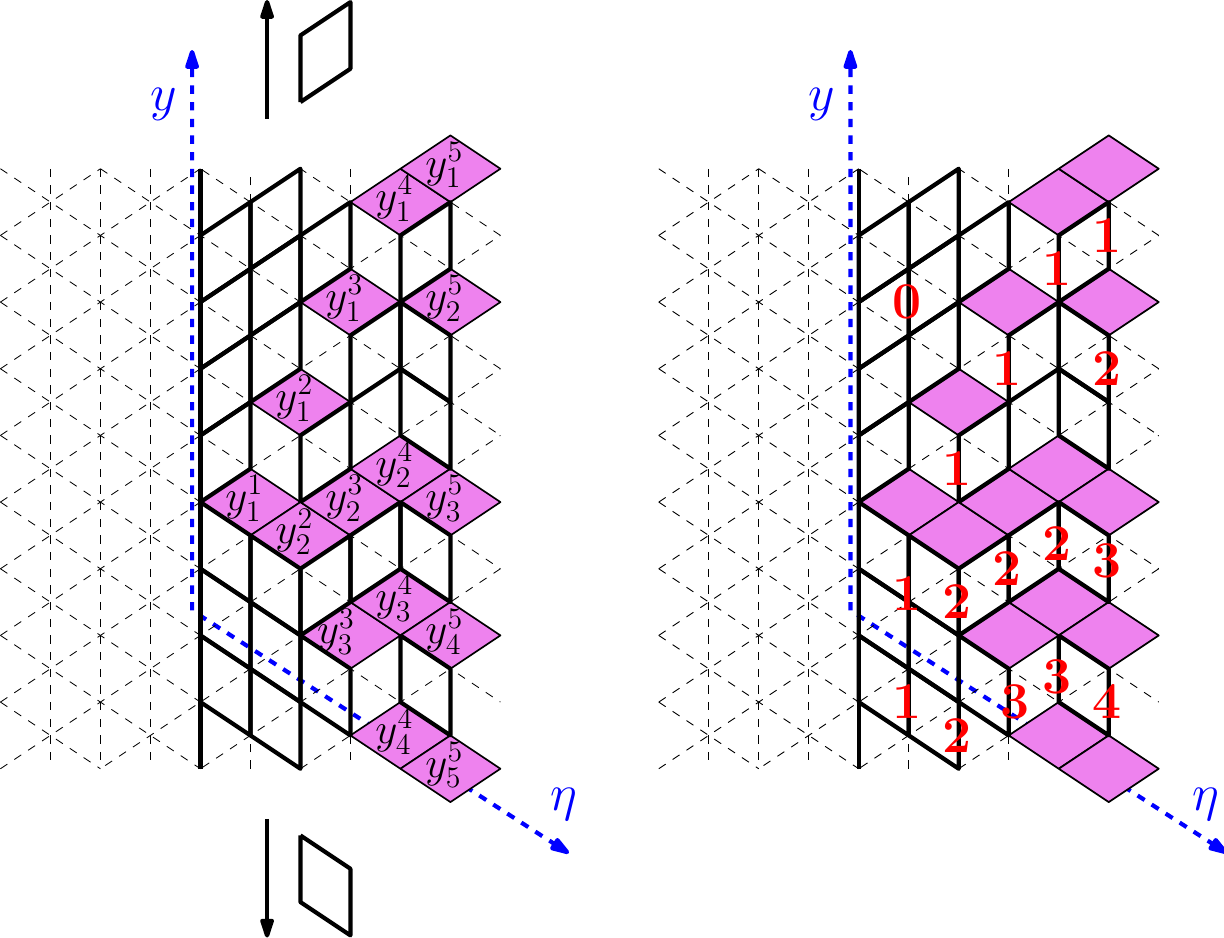}}}
\end{center}
 \caption{Lozenge tilings of halfplane corresponding to paths of $\mathcal P$.
 Left panel: Horizontal lozenges encode coordinates of signatures in path. Right
 panel: Some values of the height function.
 \label{Fig_Lozenge_extreme}}
\end{figure}

There is a natural bijection between $\widehat{\mathcal P}$ and the set $\mathcal P$
of paths in the Gelfand--Tsetlin graph. For that observe that due to combinatorial
constraints, there are precisely $N$ horizontal lozenges with horizontal coordinate
$N$ in a tiling of $\widehat{\mathcal P}$. Let $y_1^N>y_2^N>\dots>y_N^N$ denote the
coordinates of this lozenges, where the coordinate system is shown in Figure
\ref{Fig_Lozenge_extreme}. Then define $\lambda^{(N)}\in\GT_N$ through
\begin{equation}
\label{eq_shifts}
 y_i^N=\lambda^{(N)}_i+N-i,\quad 1\le i \le N.
\end{equation}
A direct check shows that then $\lambda^{(1)}\prec\lambda^{(2)}\prec\dots$ and
moreover \eqref{eq_shifts} is a one-to-one correspondence between $\widehat{\mathcal
P}$ and $\mathcal P$.

In terms of lozenge tilings, the height function $H_N(y,\eta,\cdot)$ has a very
transparent meaning: for a given $(y,\eta)$ it counts the number of horizontal
lozenges {\scalebox{0.16}{\includegraphics{lozenge_hor.pdf}}} \emph{above}
$(Ny,N\eta)$, cf.\ Figure \ref{Fig_Lozenge_extreme}\footnote{Many articles use
another definition, counting the number of lozenges of types
{\scalebox{0.16}{\includegraphics{lozenge_v_up.pdf}}},
{\scalebox{0.16}{\includegraphics{lozenge_v_up.pdf}}} below the point $(Ny,N\eta)$.
Two definitions of the height function are related by an affine transform, and so
the CLT for them is the same.}. In this way Theorem \ref{theorem:extr-char-gff-1}
can be restated as a Central Limit Theorem for certain probability measures on
lozenge tilings.

\smallskip
There is also a \emph{different} family of probability measures on lozenge tilings, which we can
analyze. The definition of these measures is purely combinatorial. Instead of tiling a half--plane,
let us take a \emph{strip} of width $N$, allowing $N$ horizontal lozenges to stick out of its
right--boundary, see Figure \ref{Fig_Lozenge_extreme} and left panel of Figure
\ref{Fig_Lozenge_1}. Note that if we fix the lozenges along the right--boundary, then the tiling is
deterministic outside a finite \emph{trapezoid}: above the trapezoid we observe only
{\scalebox{0.16}{\includegraphics{lozenge_v_up.pdf}}} lozenges, and below there are only
{\scalebox{0.16}{\includegraphics{lozenge_v_down.pdf}}} lozenges (such a trapezoid is also shown in
the left panel of Figure \ref{Fig_Lozenge_1}).

\begin{figure}[t]
\begin{center}
 {\scalebox{0.7}{\includegraphics{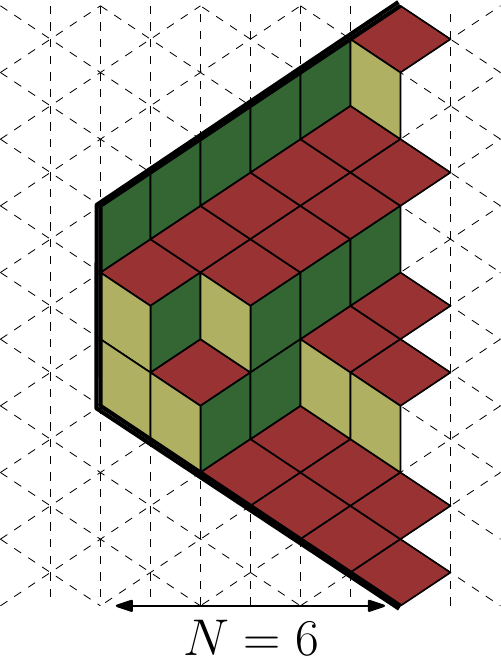}}}
 \quad \quad
  {\scalebox{0.19}[0.31]{\includegraphics{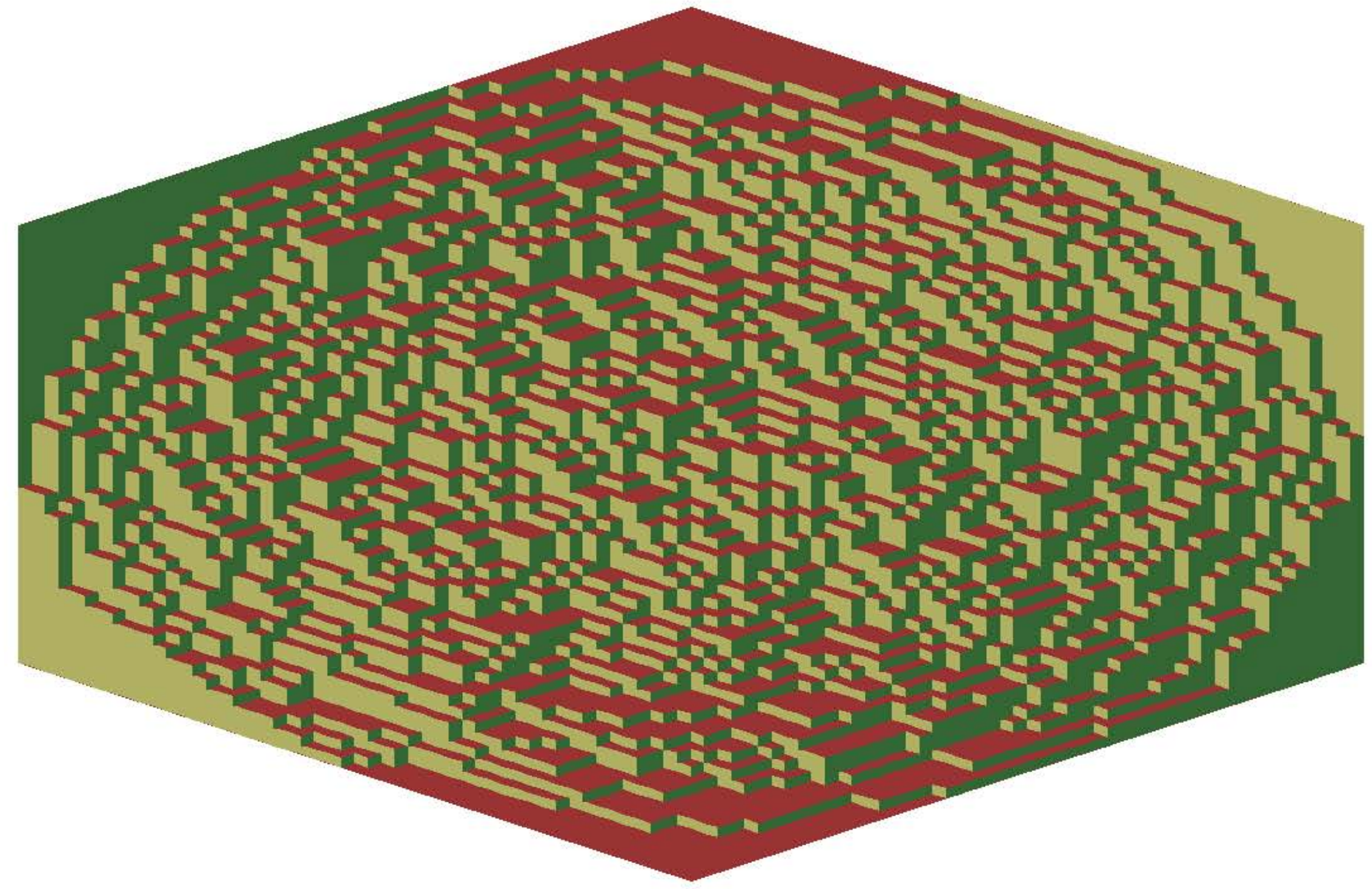}}}
\end{center}
 \caption{Left panel: Lozenge tiling of a trapezoid domain of width $N=6$. Right
 panel: Tilings of a hexagon can be identified with tilings of a specific trapezoid
 domain. Here a sample of uniformly random tiling of $50\times 50\times 50$ hexagon
 is shown.
 \label{Fig_Lozenge_1}}
\end{figure}

Repeating the bijection $\widehat {\mathcal P} \leftrightarrow \mathcal P$ we arrive
at a correspondence between paths from $\mathcal P_N$ and lozenge tilings of
trapezoids.

Let us fix $\la(N) \in \GT_N$ and consider the set $\mathcal P_N (\la(N)) \subset
\mathcal P_N$ of all paths between $\varnothing$ and $\la(N)$. This is a finite set.
Let us equip this set with a uniform probability measure. We are interested in the
asymptotic behavior of random paths distributed according to this measure. In terms
of lozenge tilings, we consider a uniformly random tiling of a trapezoid of width
$N$ with prescribed (deterministic) positions of horizontal lozenges along the right
boundary.

Repeating \eqref{eq_Height_function} we now define the (random) height function
$H^{\la(N)}(y,\eta)$ of such path. As before, in terms of a lozenge tiling, it
counts the number of horizontal lozenges
{\scalebox{0.16}{\includegraphics{lozenge_hor.pdf}}} above a point $Ny,N\eta$. Note
that we now have $0\le \eta \le 1$, as the tiling is not defined outside this range.

As in Section \ref{Section_extreme_characters}, the CLT for $H^{\la(N)}(y,\eta)$
involves a certain map to the upper half--plane $\mathbb H$. Let us introduce it.

For a probability measure $\mes$ with compact support on $\R$, we define a map $z
\to (y_{\mes} (z), \eta_{\mes} (z))$, $\HH \to \R \times \R$ via
$$
y_{\mes} (z) = z+  \frac{(z- \bar z) (\exp ( C_{\mes} ( \bar z) )-1) \exp( C_{\mes}
(z))}{\exp(C_{\mes} (z)) - \exp( C_{\mes} (\bar z))},
$$
$$
\eta_{\mes} (z) = 1 + \frac{(z- \bar z) (\exp ( C_{\mes} ( \bar z) )-1) ( \exp(
C_{\mes} (z))-1)}{\exp(C_{\mes} (z)) - \exp( C_{\mes} (\bar z))}.
$$
Note that the expressions on the right-hand side of the equations above are
invariant with respect to complex conjugations, so $y_{\mes} (z)$ and $\eta_{\mes}
(z)$ are indeed real for any $z$. Let $D_{\mes} \subset \mathbb R^2$ be the image of
this map. Also set
$$
\FF_{\mes; \eta} (z) := z + \frac{1-\eta}{\exp(- C_{\mes} (z)) -1}.
$$

\begin{proposition}
\label{prop:comp-struct-tilings} a) Assume that $\mes$ is a probability measure with
compact support and density $\le 1$ with respect to the Lebesgue measure. Then the
map $z \to (y_{\mes} (z), \eta_{\mes} (z))$ is a diffeomorphism between $\HH$ and
$D_{\mes}\subset \mathbb R \times [0,1]$.

b) This diffeomorphism can be defined in another way. For fixed $(y,\eta)\in\mathbb
R\times[0,1]$ consider the equation $\FF_{\mes; \eta} (z) =y$. Then this equation
has either 0 or 1 root in $\HH$. Moreover, there is a root in $\HH$ if and only if
$(y, \eta) \in D_{\mes}$, and if we put into correspondence to the pair $(y, \eta)
\in D_{\mes}$ the root from $\HH$ we obtain the inverse of the map $z \to (y_{\mes}
(z), \eta_{\mes} (z))$.
\end{proposition}
\begin{proof}
This is Theorem 2.1 of \cite{DM}. Note that there is a slight difference in
notations: $\chi = y_{\mes} +1-\eta_{\mes}$ and $\eta_{\mes} = \eta$, where $(\chi,
\eta)$ is a notation from \cite{DM}.
\end{proof}
As in Section \ref{Section_extreme_characters} we carry the height function
$H^{\la(N)} (y,\eta)$ over to $\mathbb H$ through
$$
H_N (z) := H^{\la(N)} ( y_{\mes}(z), \eta_{\mes} (z)), \qquad z \in \HH.
$$
As before, we do not lose any information here, as the tiling is \emph{frozen}
outside $D_\mes$ and there are no fluctuations, cf.\ right panel of Figure
\ref{Fig_Lozenge_1}, where the lozenge tiling is frozen outside the circle inscribed
into the hexagon.

Define a moment of the random height function as
$$
M_{\eta,k}^{\la(N)} = \int_{-\infty}^{+\infty} y^k \left( H^{\la(N)} ( y,\eta) - \E
H^{\la(N)} ( y, \eta) \right) dy, \qquad 0 < \eta \le 1, \ k \in \N.
$$
Also define the corresponding moment of GFF via
$$
\mathcal M_{\eta,k}^{\mes} = \int_{z \in \mathbb H; \eta = \eta_{\mes} (z)} y_{\mes}
(z)^k \mathfrak G(z) \frac{d y_{\mes}(z)}{dz} dz, \qquad 0 < \eta \le 1, \ k \in \N.
$$
\begin{theorem}[Central Limit Theorem for lozenge tilings]
\label{th:ctl-proj-GFF-1} Suppose that $\la(N)\in \GT_N$, $N=1,2,\dots$, is a
regular sequence of signatures such that
\begin{equation}
\label{eq:signat-to-meas-proj}
  \lim_{N\to\infty} m[\la(N)]=\mes, \qquad \mbox{weak convergence,}
\end{equation}
and let $H_N(z)$ be the height function for the uniformly random element of
$\mathcal P_N(\lambda(N))$. Then
 $$
 \sqrt{\pi}\left(H_N (z) - \mathbf E H_N (z) \right) \xrightarrow[N \to \infty]{} \mathfrak G (z), \qquad z \in \mathbb H,
 $$
in the sense that, as $N \to \infty$, the collection of random variables
$\{\sqrt{\pi} M_{\eta,k}^{\lambda(N)} \}_{\eta>0; k \in \mathbb Z_{\ge 0}}$
converges, in the sense of moments, to $\{ \mathcal M_{\eta,k}^{\mes} \}_{\eta>0; k
\in \mathbb Z_{\ge 0}}$.
\end{theorem}

\begin{remark}
 For explicit expressions for the covariance of $\{ \mathcal M_{\eta,k}^{\mes} \}_{\eta>0; k
\in \mathbb Z_{\ge 0}}$, see Lemma \ref{lem:two-int-proj}.
\end{remark}

The proof of Theorem \ref{th:ctl-proj-GFF-1} is given in Section
\ref{Section_tilings_GFF}, and we believe that in this generality it is new.

The convergence to the Gaussian Free Field for certain lozenge tiling models was
first obtained by Kenyon \cite{Ken}. Theorem \ref{th:ctl-proj-GFF-1} is closely
related to the result obtained by Petrov \cite{Pet2}. There are two differences:
First, in \cite{Pet2} the convergence is obtained only for measures $\mes$ which
consist of finitely many segments with density 1. In Theorem \ref{th:ctl-proj-GFF-1}
an arbitrary measure $\mes$ with compact support is allowed. The second difference
is that, though the limit object is the same, the convergence is proved for
different sets of observables.

\subsection{Domino tilings}
\label{Section_Aztec_statement}

\begin{figure}[t]
\begin{center}
 {\scalebox{0.95}{\includegraphics{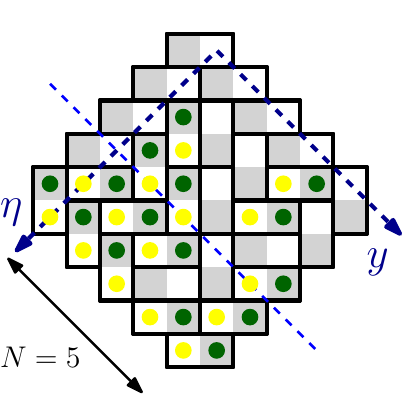}}}
  {\scalebox{0.15}{\includegraphics{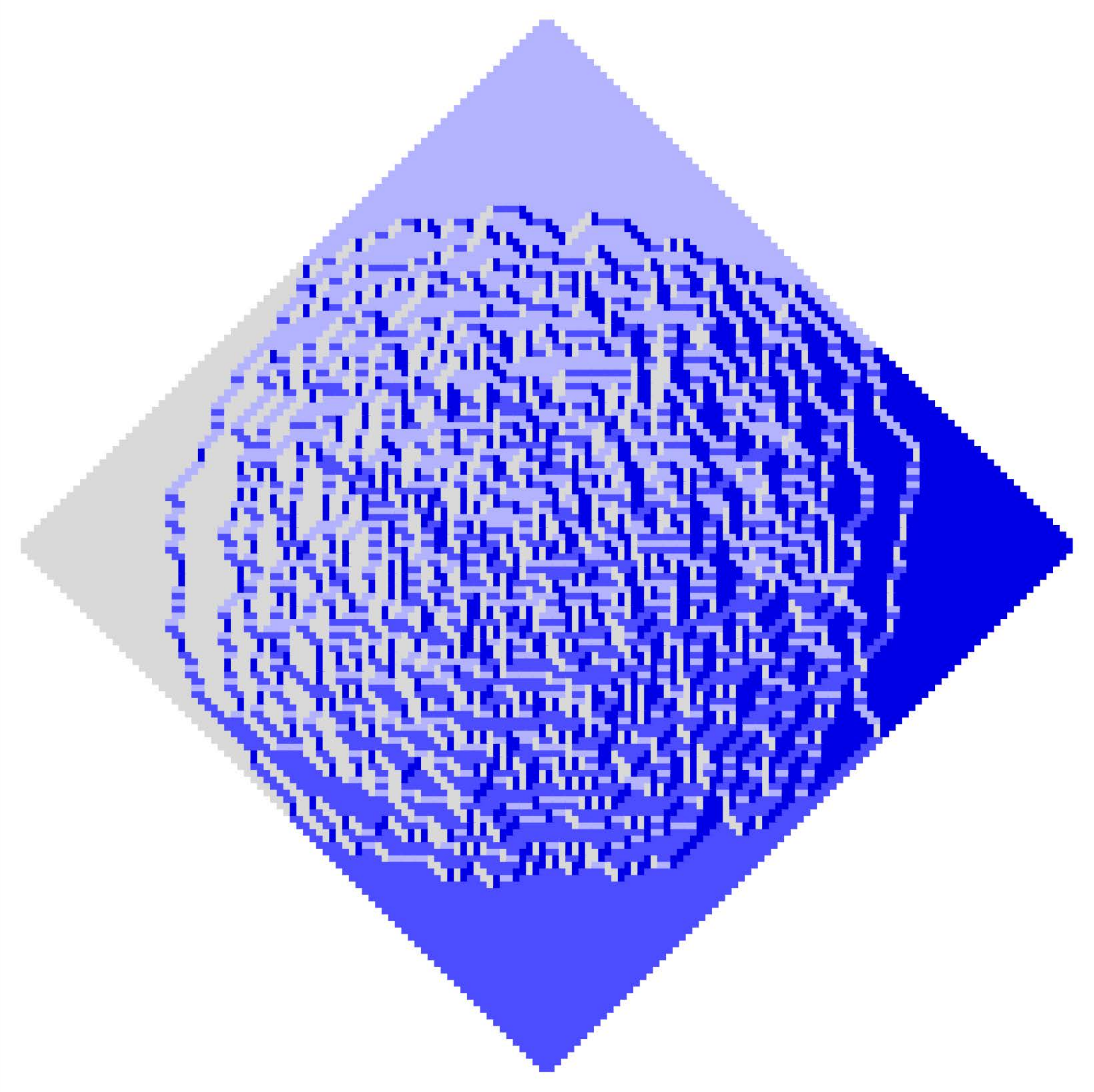}}}
\end{center}
 \caption{Left panel: Domino tiling of Aztec diamond of size $5$ and corresponding
 particle system. Right panel: Uniformly random domino tiling of Aztec diamond of
 size $80$.
 \label{Fig_Domino_1}}
\end{figure}

In this section we switch from the triangular grid to the square grid and replace
lozenges by \emph{dominos}. Consider an Aztec diamond of size $N$, which is the side
$N$ ``sawtooth'' rhombus drawn on the square grid, as shown in Figure
\ref{Fig_Domino_1}. Following \cite{EKLP}, we consider tilings of this rhombus with
vertical and horizontal $2\times 1$ dominos. For a positive real $\mathfrak{q}$ it is known that
$$
 \sum_{\Omega \text { is a domino tiling of size } N \text{ Aztec diamond}} \mathfrak{q}^{\frac{1}{2}(\text{number
 of horizontal dominos in }\Omega)}=(1+\mathfrak{q})^{N(N+1)/2}.
$$
Let us pick a random tiling of size $N$ Aztec diamond according to the probability
measure $\mathfrak{q}^{\frac{1}{2}(\text{number of horizontal dominos})} \cdot {(1+\mathfrak{q})^{-N(N+1)/2}}$. A
sample from this measure for $\mathfrak{q}=1$ is shown in the right panel of Figure
\ref{Fig_Domino_1}.

Similarly to Section \ref{sec:restrictions}, we can identify domino tilings with
sequences of signatures, although the construction is more delicate this time.
Coloring the grid in the checkerboard order, we can distinguish four types of
dominos: two vertical ones
{\scalebox{0.2}{\includegraphics{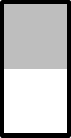}}},
{\scalebox{0.2}{\includegraphics{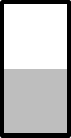}}} and two horizontal
ones {\scalebox{0.2}{\includegraphics{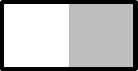}}},
{\scalebox{0.2}{\includegraphics{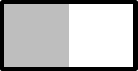}}}. We further choose
one of the horizontal types and one of the vertical types; for the sake of being
definite let us choose {\scalebox{0.2}{\includegraphics{domino_vert_gray_top.pdf}}}
and {\scalebox{0.2}{\includegraphics{domino_hor_gray_right.pdf}}}. We stick to these
two types and put \emph{green} particles on the gray squares (of the checkerboard
coloring) and \emph{yellow} particles on the white squares, as shown in the left
panel of Figure \ref{Fig_Domino_1}.

Reading the yellow particle configuration from up--right to down--left, we observe
$N$ slices with $1$, $2$, \dots, $N$ particles, respectively; a $3$--particle slice
is shown in Figure \ref{Fig_Domino_1}. The particles of the $t$--particle slice have
coordinates $y^1_t>y^2_t>\dots>y^t_t$, which we identify with a signature
$\lambda^{(t)}\in\GT_t$ through
$$
 y^t_i=\lambda^{(t)}_i+t-i,\quad 1\le i \le t \le N.
$$
We can now define the height function $H_N(y,\eta)$ of uniformly random domino
tiling of the size $N$ Aztec diamond through the very same formula
\eqref{eq_Height_function} as before. In terms of tilings, the height function
counts the yellow particles in the down--right direction on the given diagonal (of fixed $\eta$ and
growing $y$) from the point $(Ny,N\eta)$.


As before we would like to carry the height function to the upper half--plane. For that we need the following proposition.
\begin{proposition}
\label{prop:aztec-struct} For any $y \in \R$ and $\eta \in (0;1]$ the equation
\begin{equation}
\label{eq_Aztec_equation}
z^2 (\mathfrak{q} - y \mathfrak{q})+ z ( \eta \mathfrak{q} + \eta + \mathfrak{q} -y (1+\mathfrak{q})) + \eta (1+\mathfrak{q})=0,
\end{equation}
has 0 or 1 root in $\HH$. It has a root in $\HH$ if and only if the pair $(y,\eta)$
lies in the ellipse inscribed in the Aztec diamond
$$
\mathbf D_{\mathbf{A}} = \{ (y,\eta): \left( \frac{(y-\eta)^2}{\mathfrak{q}}+(y+\eta-1)^2 \right) (1+\mathfrak{q}) \le 1 \}.
$$
The map $\mathbf D_{\mathbf{A}} \to \HH$ given by this root is a diffeomorphism. We
denote by $z \to (y_{\mathbf A} (z), \eta_{\mathbf{A}}(z))$ the inverse of this map.
\end{proposition}
This proposition coincides with Lemma 5.1 from \cite{CJY}.

Let us carry $H_N(y,\eta)$ over to $\mathbb H$ --- define
$$
H^{\mathbf{A};N} (z) := H_N (y_{\mathbf{A}}(z), \eta_{\mathbf{A}} (z)), \qquad z \in
\mathbb H.
$$
For  $0<\eta \le 1$ and $k=1,2,\dots$, define a moment of the random height function
as
$$
M_{\eta,k}^{\mathbf{A},N} = \int_{-\infty}^{+\infty} y^k \left( H_N ( y, \eta) -
\mathbf E H_N (y, \eta) \right) dy.
$$
Also define the corresponding moment of GFF via
$$
\mathcal M_{\eta,k}^{\mathbf{A}} = \int_{z \in \mathbb H\mid  \eta_{\mathbf{A}}
(z)=\eta} y_{\mathbf{A}}(z)^k \mathfrak G(z) \frac{d y_{\mathbf{A}}(z)}{dz} dz.
$$
\begin{theorem}[Central Limit Theorem for the domino tilings of the Aztec diamond]
\label{theorem:extr-char-aztec}

Let $H^{\mathbf{A};N} (z)$ be a random function corresponding to the uniformly random domino tiling of the Aztec diamond in the way described above. Then
$$
\sqrt{\pi}\left(H^{\mathbf{A};N} (z) - \mathbf E H^{\mathbf{A};N} (z) \right)
\xrightarrow[N \to \infty]{} \mathfrak G (z).
$$
In more details, as $N \to \infty$ the collection of random variables $\{\sqrt{\pi}
M_{\eta,k}^{\mathbf{A},N} \}_{0<\eta<1; k \in \mathbb Z_{\ge 0}}$ converges, in the
sense of finitely-dimensional distributions, to $\{ \mathcal M_{\eta,k}^{\mathbf{A}}
\}_{0<\eta<1; k \in \mathbb Z_{\ge 0}}$.
\end{theorem}
\begin{remark}
 For the explicit expression for the covariance of $\{ \mathcal
M_{\eta,k}^{\mathbf{A}} \}$ see \eqref{eq:aztec-moment-formula}.
\end{remark}

Theorem \ref{theorem:extr-char-aztec} was first announced in \cite{CJY} without
technical details. Our proof is given in Section \ref{Section_Aztec_proof}.
Moreover, Theorem \ref{theorem:extr-char-aztec} can be extended to random
domino tilings of more general domains, as shown in \cite{BK}.


\subsection{Noncolliding random walks}
\label{Section_noncol}
 We proceed to our final application. Here the general framework is to study $N$
 independent identical random walks on $\mathbb Z$ conditioned to have no collisions
 with each other. This model is quite general, as one can start from different
 random walks, and also the initial configuration for the conditional process might
 vary.
\smallskip

Here we stick to three simplest random walks (but it is natural to expect that the
results generalize far beyond that). Let $R$ be one of the following:
\begin{itemize}
\item The continuous time Poisson random walk $R=R_\gamma$ of intensity $\gamma>0$.
\item The discrete time Bernoulli random walk $R=R_\beta$, where at each moment the
particle can either jump to the right by one with probability $0<\beta<1$ or stay
put with probability $1-\beta$.
\item The discrete time geometric random walk $R=R_\alpha$, where for $\alpha>0$ at each
moment the particle jumps to the right $i$ steps with probability $(1-\alpha)
\alpha^i$, $i=0,1,2,\dots$.
\end{itemize}

We now define for each $N=1,2,\dots$ the $N$--dimensional noncolliding process $X^{N;R}$. We fix an
arbitrary initial condition $X^{N;R}_1(0)>\dots>X^{N;R}_N(0)$, take $N$ independent identically
$R$--distributed random walks started from points $X^{N;R}_1(0)$,\dots $X^{N;R}_N(0)$ and define
$X^{N;R}(t)$ as the conditional process given that the trajectories of these random walks do not
intersect (at all times $t\ge 0$), cf.\ Figure \ref{Fig_Poisson_non}. Note that the condition has
probability zero, and so one needs to make sense of it. One way here is to start with considering
distinct ordered speeds (which correspond to the parameters $\gamma$, $\beta$ or $\alpha$), and
then make them all equal through a limit transition. We refer to \cite{OCon}, \cite{KOR} for the
details of the construction. The result is that $X^{N;R}$ is a Markov process, which fits into the
formalism of Section \ref{Section_multilevel_general}, more specifically, the maps $\mathfrak
p_{N,N}$ are given by the multiplication, as in Example 2 of Section
\ref{Section_multilevel_specified}.

\begin{figure}[t]
\begin{center}
 {\scalebox{1.0}{\includegraphics{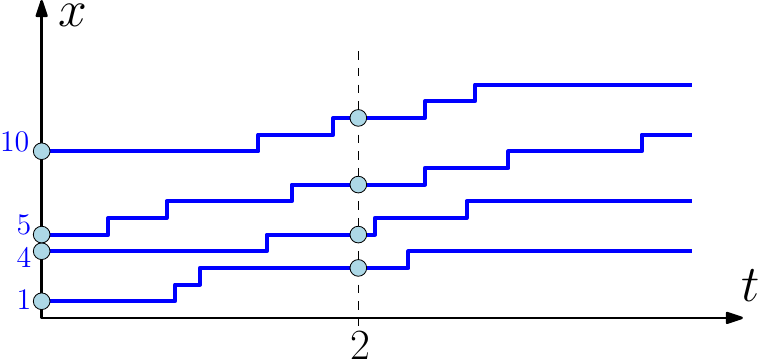}}}
\end{center}
 \caption{Four noncolliding Poisson random walks started at configuration
 $(1,4,5,10)$ and the positions of walkers at time $t=2$.
 \label{Fig_Poisson_non}}
\end{figure}

Let us identify the points of $X^{N;R}$ with a signature $\lambda$ through
\begin{equation}
\label{eq_particles_to_signatures}
 X^{N;R}_i=\lambda_i+N-i.
\end{equation}
In this notation, if $R=R_\gamma$ and $\lambda^{(1)}$,\dots $\lambda^{(n)}$ describe $X^{N;R}(t)$
at times $t_1>t_2>\dots>t_n=0$, then  (in the notations of Example 2 in Section
\ref{Section_multilevel_specified}),
\begin{equation}
\label{eq_transition_gamma}
 g_k=\exp\left(- N \gamma(t_{k}-t_{k+1})+ \gamma(t_{k}-t_{k+1}) \sum_{i=1}^{N} x_i
 \right),\quad k=1,\dots,n-1.
\end{equation}
If $R=R_\beta$, then (this time $t_k$ should be integers)
\begin{equation}
\label{eq_transition_beta}
 g_k=\prod_{i=1}^N \bigl(1+\beta(x_i-1)\bigr)^{t_{k}-t_{k+1}},\quad k=1,\dots,n-1.
\end{equation}
If $R=R_\alpha$, then (again $t_k$ are integers)
\begin{equation}
\label{eq_transition_alpha}
 g_k=\prod_{i=1}^N \left(\frac{1}{1-\alpha(x_i-1)}\right)^{t_{k}-t_{k+1}},\quad k=1,\dots,n-1.
\end{equation}
We are in a position to consider the large $N$-limit of these models. For that assume
that $X^{N;R}(0)$ is given through \eqref{eq_particles_to_signatures} by a signature
$\lambda(N)$, and as $N\to\infty$ these signatures are regular in the sense of
Definition \ref{definition_regularity}. Let us choose some $k$ times
$\tau_1>\tau_2>\dots>\tau_k>0$ and consider $X^{N;R}(t)$ at
$t=N\tau_1,N\tau_2,\dots,N\tau_k$. Then using Theorems
\ref{Theorem_character_asymptotics}, \ref{Theorem_character_asymtotics_2} for the
asymptotic of Schur generating function for $\lambda(N)$, and explicit formulas
 \eqref{eq_transition_gamma}, \eqref{eq_transition_beta}, \eqref{eq_transition_alpha}
we can use Theorem \ref{theorem:main-multiplication} and obtain the Central Limit
Theorem for the global fluctuations of $X^{N;R}(N\tau_1), \dots, X^{N;R}(N\tau_N)$.
The fluctuations are asymptotically Gaussian with covariance given by the double
contour integral \eqref{eq_covariance_multiplication}. It is plausible that the
covariance structure can be again described in terms of the Gaussian Free Field, as
in Sections \ref{Section_extreme_characters}, \ref{sec:restrictions},
\ref{Section_Aztec_statement}, but we do not address this question in the present
paper.

As far as we know, the CLT for global fluctuations was not addressed
before in this generality. The situation is different for a special case of densely
packed initial condition $\lambda(N)=(0^N)$. Then for $R=R_\gamma$ the CLT
(and identification with the Gaussian Free Field) was previously addressed in
\cite{BF} by the technique of determinantal point processes and in
\cite{Borodin_Bufetov}, \cite{Kuan_noncom} by computations in the universal enveloping algebra of $U(N)$.
Further, for all three cases $R=R_\gamma,R_\beta,R_\alpha$ (and still
$\lambda(N)=(0^N)$) the CLT for global fluctuations was established in
\cite{Duits_nonc} by employing recurrence relations for orthogonal polynomials.

\section{Formula for moments}
\label{sec:4}

Our method of proof is based on the fact that given the knowledge of Schur generating function of a probability measure, one can compute its moments. In order to do this, one can apply a certain family of differential operators such that the Schur functions are eigenfunctions of these operators. In more details, it is a straightforward computation that for a probability measure $\rho_N$ on $\GT_N$ with the Schur generating function $S_N (\vec{x})$ we have
$$
\left. \frac{1}{V_N (\vec{x})} \sum_{i=1}^N (x_i \pa_i)^k V_N (\vec{x}) S_N (\vec{x}) \right|_{\vec{x}=1} = \sum_{\la \in \GT_N} \rho_N (\la) \sum_{i=1}^N (\la_i +N -i)^k = \E \sum_{i=1}^N (\la_i + N-i)^k.
$$
More generally, we have
\begin{equation}
\label{eq:method-of-proof}
\left. \frac{1}{V_N (\vec{x})} \sum_{i=1}^N (x_i \pa_i)^k \sum_{j=1}^N (x_j \pa_j)^l V_N (\vec{x}) S_N (\vec{x}) \right|_{\vec{x}=1} = \E \left( \sum_{i=1}^N (\la_i + N-i)^k \sum_{i=1}^N (\la_i + N-i)^l \right),
\end{equation}
and the similar formulas hold for the joint moments of several power sums of coordinates.

Let us address now a general case of Markov chains introduced in Section
\ref{Section_multilevel_general}. In Proposition
\ref{prop:gen-formula-moments-markov} below we prove a general formula for moments
in this setting. Similar formulas for Macdonald processes can be found in
\cite[Section 4]{BCGS}.

We will need the following technical lemmas.

\begin{lemma}
\label{lem:shur-shod-2}
Assume that the sum
\begin{equation}
\label{eq:monomy-sym}
\sum_{l_1, \dots, l_N \in \mathbb{Z}} c_{l_1, \dots, l_N} x_1^{l_1} x_2^{l_2} \dots x_N^{l_N}, \qquad c_{l_1, \dots, l_N} \in \R,
\end{equation}
absolutely converges in an open neighborhood of the $N$-dimensional torus $\{ (x_1, \dots, x_N): |x_1|=1, \dots, |x_N|=1 \}$.
Let $\{ m_{l_1, \dots, l_N} \}_{l_1, \dots, l_N \in \mathbb{Z}}$ be a sequence of reals such that $|m_{l_1, \dots, l_N}| \le C \left( |l_1|^k+ \dots + |l_N|^k \right)$, where $C$ is a positive real.
Then the sum
$$
\sum_{l_1, \dots, l_N \in \mathbb{Z}} m_{l_1, \dots, l_N} c_{l_1, \dots, l_N} x_1^{l_1} x_2^{l_2} \dots x_N^{l_N}, \qquad
$$
absolutely converges in an open neighborhood of the $N$-dimensional torus.
\end{lemma}
\begin{proof}
Let $\eps>0$ be a real number such that the series \eqref{eq:monomy-sym} absolutely converges in $\{ (x_1, \dots, x_N): 1-\eps \le |x_i| \le 1+\eps, i=1,\dots, N \}$. Consider a series
\begin{equation}
\label{eq:monom-sym-2}
\sum_{l_1, \dots, l_N \in \mathbb{Z}} \sum_{signs} |c_{l_1, \dots, l_N}| |1 \pm \eps|^{l_1} |1 \pm \eps|^{l_2} \dots |1 \pm \eps|^{l_N},
\end{equation}
where the sum over signs contains $2^N$ terms corresponding to different choices of signs in $\pm$ inside the arguments. Note that this series is convergent. Assume that $l_1, \dots, l_s$ are positive, and $l_{s+1}, \dots, l_N$ are negative. Then
\begin{equation}
\label{eq:monom-sym-3}
(1+\eps)^{l_1} \dots (1+\eps)^{l_s} (1-\eps)^{l_{s+1}} \dots (1-\eps)^{l_N} \ge |x_1|^{l_1} |x_2|^{l_2} \dots |x_N|^{l_N}.
\end{equation}
Since for any $l_i$'s there is a term of the form \eqref{eq:monom-sym-3} in the summation \eqref{eq:monom-sym-2}, we obtain
that there exists $D>1$ such that the series
$$
\sum_{l_1, \dots, l_N \in \mathbb{Z}} |c_{l_1, \dots, l_N}| D^{|l_1|+\dots+|l_N|}
$$
is convergent. This implies the statement of the lemma.
\end{proof}

\begin{lemma}
\label{lem:shur-shod-3}
Assume that the series
$$
\sum_{\la \in \GT_N} c_{\la} s_{\la} (x_1, \dots, x_N), \qquad c_{\la} \in \R_{\ge 0},
$$
absolutely converges in an open neighborhood of the $N$-dimensional torus $\{ (x_1, \dots, x_N): |x_1|=1, \dots, |x_N|=1 \}$.
Let $\{ m_{\la} \}_{\la \in \GT_N}$ be a sequence of reals such that $|m_{\la}| \le C \left( |\la_1|^k+ \dots + |\la_N|^k \right)$, where $\la = (\la_1, \dots, \la_N)$ and $C$ is a positive real.
Then the sum
$$
\sum_{\la \in \GT_N} m_{\la} c_{\la} s_{\la} (x_1, \dots, x_N), \qquad c_{\la} \in \R_{\ge 0}
$$
absolutely converges in an open neighborhood of the $N$-dimensional torus.
\end{lemma}
\begin{proof}
Set
$$
\mathfrak{f} (x_1, \dots, x_N) := \sum_{\la \in \GT_N} c_{\la} s_{\la} (x_1, \dots, x_N), \qquad c_{\la} \in \R_{\ge 0}.
$$
Then $\mathfrak{f} (x_1, \dots, x_N)$ is an analytic symmetric function in a neighborhood of the $N$-dimensional torus. Therefore, $\mathfrak{f} (x_1, \dots, x_N) \prod_{1\le i < j \le N} (x_i-x_j)$ is an analytic antisymmetric function and can be written as an absolutely convergent sum of monomials:
$$
\mathfrak{f} (x_1, \dots, x_N) \prod_{1\le i < j \le N} (x_i-x_j) = \sum_{l_1, \dots, l_N \in \mathbb{Z}} c_{l_1, \dots, l_N} x_1^{l_1} \dots x_N^{l_N}.
$$
Due to antisymmetry we can consider only terms with $l_1 >l_2 > \dots >l_N$.
Then Lemma \ref{lem:shur-shod-2} shows that the sum
$$
\sum_{l_1 > \dots > l_N \in \mathbb{Z}} m_{(l_1-N+1, l_2-N+2, \dots, l_N)} c_{l_1, \dots, l_N} x_1^{l_1} \dots x_N^{l_N}
$$
is absolutely convergent in some neighborhood of the $N$-dimensional torus. Multiplying this series by the inverse of the Vandermond determinant, we obtain that the desired series
\begin{equation}
\label{eq:vand-polyn}
 \sum_{\la \in \GT_N} m_{\la} c_{\la} s_{\la} (x_1, \dots, x_N),
\end{equation}
absolutely converges in the region $\prod_{i<j} (x_i-x_j) > \delta$ for any $\delta>0$. Since the series \eqref{eq:vand-polyn} consists of analytic functions, the use of the Cauchy integral formula gives the absolute convergence in a neighborhood of the torus.
\end{proof}

For a positive integers $m,n$ set
$$
\mathcal D_m^{(n)} := \prod_{1\le i<j \le n} \frac{1}{ x_i-x_j} \left( \sum_{i=1}^n \left (x_i
\partial_i\right)^m \right) \prod_{1\le i<j\le n} (x_i-x_j).
$$

\begin{proposition}
\label{prop:gen-formula-moments-markov} In notations of Section
\ref{Section_multilevel_general} let $m_1, \dots, m_k$ be positive integers, let
$n_1, \dots, n_k$, $\mathfrak{p}_{n_2, n_1}$, ..., $\mathfrak{p}_{n_{k}, n_{k-1}}$
be as in Section \ref{Section_multilevel_general}, and let $\rho$ be a probability
measure on $\GT_{n_k}$ with the Schur generating function $S_{\rho} \in
\Lambda^{n_k}$. Assume that $(\la^{(1)}, \dots, \la^{(k)})$ is distributed according
to \eqref{eq:gener-form-measure-markov}. Then
\begin{multline}
\label{eq:prop-gen-form-moments}
\left. \mathcal D_{m_1}^{(n_1)} \mathfrak{p}_{n_2, n_1} \mathcal D_{m_2}^{(n_2)} \mathfrak{p}_{n_3, n_2} \dots \mathfrak{p}_{n_k, n_{k-1}} \mathcal D_{m_k}^{(n_k)} S_{\rho} (x_1, \dots, x_{n_k}) \right|_{x=1} \\ = \mathbf E \left( \sum_{i_1=1}^{n_1} (\la_i^{(1)} + n_1-i_1)^{m_1} \sum_{i_2=1}^{n_2} (\la_i^{(2)} + n_2-i_2)^{m_2} \dots \sum_{i_k=1}^{n_k} (\la_i^{(k)} + n_k-i_k)^{m_k} \right),
\end{multline}
where in the left-hand side we set to 1 all variables after applying all differential operators.
\end{proposition}
\begin{proof}
We will prove this proposition for $k=2$; the proof for general $k$ is analogous. We have
$$
\mathcal D_{m_2}^{(n_2)} S_{\rho} (x_1, \dots, x_{n_2}) = \sum_{\la \in \GT_{n_2}} \mathrm{Prob} (\la^{(2)} = \la) \frac{s_{\la} (x_1, \dots, x_{n_2})}{s_{\la} (1^{n_2})} \sum_{i_2=1}^{n_2} \left( \la_{i_2} +n_2-i_2 \right)^{m_2}.
$$

Lemma \ref{lem:shur-shod-3} shows that this sum is absolutely convergent in an open neighborhood of the $n_2$-dimensional torus, and, therefore, belongs to $\Lambda^{n_2}$.

Thus, one can apply $\mathfrak{p}_{n_2, n_1}$ and obtain
\begin{multline*}
\mathcal D_{m_1}^{(n_1)} \mathfrak{p}_{n_2, n_1} \mathcal D_{m_2}^{(n_2)} S_{\rho} (x_1, \dots, x_{n_2}) \\ = \mathcal D_{m_1}^{(n_1)} \sum_{\la \in \GT_{n_2}} \mathrm{Prob} (\la^{(2)} = \la) \left( \sum_{\mu \in \GT_{n_1}} \mathfrak{c}_{\la,\mu}^{\mathfrak{p}_{n_2, n_1}} \frac{s_{\mu} (x_1, \dots, x_{n_1})}{s_{\mu} (1^{n_1})} \right) \sum_{i_2=1}^{n_2} \left( \la_{i_2} +n_2-i_2 \right)^{m_2}
\\ = \sum_{\la \in \GT_{n_2}} \sum_{\mu \in \GT_{n_1}} \mathrm{Prob} (\la^{(2)} = \la) \mathfrak{c}_{\la,\mu}^{\mathfrak{p}_{n_2, n_1}} \frac{s_{\mu} (x_1, \dots, x_{n_1})}{s_{\mu} (1^{n_1})} \sum_{i_1=1}^{n_1} (\la_i^{(1)} + n_1-i_1)^{m_1} \\ \times \sum_{i_2=1}^{n_2} \left( \la_{i_2}^{(2)} +n_2-i_2 \right)^{m_2}.
\end{multline*}
Plugging $(x_1,\dots, x_{n_1})=(1^{n_1})$ and using \eqref{eq:gener-form-measure-markov} we obtain the statement of the proposition.
\end{proof}

Therefore, our goal is to compute the asymptotics of the expressions in the left-hand side of \eqref{eq:method-of-proof} and \eqref{eq:prop-gen-form-moments}. This computation is the content of Sections \ref{sec:5}, \ref{sec:6}, \ref{sec:7}.

\section{Technical lemmas}
\label{sec:5}

This section contains the technical ingredients for the proofs of our main theorems.

\subsection{Preliminary definitions and lemmas}

For any $N \ge 1$ let $F_N (\vec{x})$ be a function of $N$ variables $\vec{x}$. For an integer $D$ we will say that a sequence of analytic complex functions $\{F_N (\vec{x})\}_{N=1}^{\infty}$ \textit{has an $N$-degree at most $D$} if for any $s \ge 0$ (not depending on $N$) and any indices $i_1, \dots, i_s$ we have
\begin{equation}
\label{eq:degree-def}
\lim_{N \to \infty} \frac{1}{N^D} \left. \pa_{i_1} \dots \pa_{i_s} F_N (\vec{x}) \right|_{\vec{x}=1} = c_{i_1, \dots, c_{i_s}},
\end{equation}
for some constants $c_{i_1, \dots, i_s}$. In particular, the limit
$$
\lim_{N \to \infty} \frac{1}{N^D} \left. F_N (\vec{x}) \right|_{\vec{x}=1}
$$
should exist (this corresponds to $s=0$).

Similarly, we will say that a sequence of analytic complex functions $\{F_N (\vec{x})\}_{N=1}^{\infty}$ \textit{has $N$-degree less than $D$} if for any $s \ge 0$ (not depending on $N$) and any indices $i_1, \dots, i_s$ we have
$$
\lim_{N \to \infty} \frac{1}{N^D} \left. \pa_{i_1} \dots \pa_{i_s} F_N (\vec{x}) \right|_{\vec{x}=1} = 0.
$$

Our main source of such functions is the following lemma.

\begin{lemma}
Assume that for $D \in \N$, a sequence of functions $\{F_N (\vec{x})\}_{N=1}^{\infty}$ satisfies the following condition: For any $k \in \N$ there exists $\eps = \eps (k) >0$ such that
$$
\lim_{N \to \infty} \frac{1}{N^D} F_N (x_1, \dots, x_k, 1^{N-k}) = \mathbb G (x_1, \dots, x_k),
$$
where $\mathbb G (x_1, \dots, x_k)$ is an analytic function in the neighborhood of $(x_1, \dots, x_k) = (1^k)$, and the convergence is uniform in the region $|x_i - 1| < \eps, \ i=1,2, \dots, k$. Then $\{F_N (\vec{x})\}_{N=1}^{\infty}$ has a $N$-degree at most $D$. If the function $\mathbb G (x_1, \dots, x_k)$ equals 0, then $\{F_N (\vec{x})\}_{N=1}^{\infty}$ has a $N$-degree less than $D$.
\end{lemma}
\begin{proof}
Let $i_1, \dots, i_s$ be indices from \eqref{eq:degree-def}. For computing the expression $\pa_{i_1} \dots \pa_{i_s} F_N (\vec{x})$ we can set to 1 all variables $x_i$ such that $i> \max(i_1, \dots, i_s)$ prior to the differentiation. After this, let us recall that the uniform convergence of complex analytic functions implies
$$
\lim_{N \to \infty} \frac{1}{N^D} \pa_{i_1} \dots \pa_{i_s} F_N (x_1, \dots, x_k, 1^{N-k}) = \pa_{i_1} \dots \pa_{i_s} \mathbb G (x_1, \dots, x_k).
$$
\end{proof}

Let $F_N^{(1)} (\vec{x})$ have $N$-degree at most $D_1$, and let $F_N^{(2)} (\vec{x})$ have $N$-degree at most $D_2$. Then it is easy to see that for any index $i$ the function $\pa_i F_N^{(1)} (\vec{x})$ has $N$-degree at most $D_1$, $F_N^{(1)} (\vec{x}) + F_N^{(2)} (\vec{x})$ has $N$-degree at most $\max(D_1, D_2)$, and $F_N^{(1)} (\vec{x}) F_N^{(2)} (\vec{x})$ has $N$-degree at most $D_1+D_2$.

\begin{lemma}
\label{lem:funct-raznosti}
Assume that for each $N=1,2,\dots$, $F_N (\vec{x})$ is a symmetric analytic function in an open neighborhood of $(1^N)$. Then for any indices $a_1, \dots, a_{q+1}$ the function
\begin{equation}
\label{eq:funct-raznosti}
Sym_{a_1, a_2, \dots, a_{q+1}}  \left( \frac{ F_N (\vec{x})}{(x_{a_1}-x_{a_2}) \dots (x_{a_1}-x_{a_{q+1}})} \right)
\end{equation}
is analytic in a (possibly smaller) open neighborhood of $1^N$. If $\{ F_N (\vec{x}) \}$ has $N$-degree at most $D$ (less than $D$), then the sequence \eqref{eq:funct-raznosti} has $N$-degree at most $D$ (less than $D$).
\end{lemma}

\begin{proof}
\cite[Lemma 5.4]{BG} implies the first claim.

We need to prove that for any indices $i_1, \dots, i_s$ the limit
\begin{multline}
\label{eq:diff-raznosti}
\lim_{x_1,\dots, x_N=1} \pa_{i_1} \dots \pa_{i_s} Sym_{a_1, a_2, \dots, a_{q+1}}  \left( \frac{ F_N (\vec{x})}{(x_{a_1}-x_{a_2}) \dots (x_{a_1}-x_{a_{q+1}})} \right)
\end{multline}
has the same $N$-degree as the function $F_N$. Note that we can immediately specialize to $1$ all variables except for $x_{i_1}, \dots, x_{i_s}$, $x_{a_1}, \dots, x_{a_{q+1}}$. After this, we deal with a statement about the functions with finite (not depending on $N$) number of variables.

Any coefficient of Taylor expansion of \eqref{eq:funct-raznosti} can be written as a finite (not depending on $N$) combination of the Taylor coefficients of the function $F_N$. Indeed, it was shown in the proof of Lemma 5.4 of \cite{BG} (see formula (5.4)) that the Taylor coefficient of $F_N$ of the term of $N$-degree $M$ from can contribute to the Taylor coefficients of \eqref{eq:funct-raznosti} of $N$-degree $d$ with $M-(q+s+1)^2 \le d \le M+ (q+s+1)^2$, where $q+s$ appears because this is the number of variables which were not immediately set to 1.

Therefore, the $N$-degree of \eqref{eq:diff-raznosti} is at most that of $F_N$.
\end{proof}

Let $\mathbf F(x)$ be a complex analytic function of one variable at the neighborhood of the unity. Let us introduce the notation for the coefficients in its Taylor expansion
$$
\mathbf F(x) =: \mathbf a_0 + \mathbf a_1 (x-1) + \mathbf a_2 (x-1)^2 + \dots + \mathbf a_n (x-1)^n + \dots.
$$

\begin{lemma}
\label{lem:basic-1}
For a function $\mathbf F(x)$ and positive integer $r$ we have
$$
\left. Sym_{x_1, \dots, x_{r+1}} \left( \frac{\mathbf F(x_1)}{(x_1-x_2) \dots (x_1-x_{r+1})} \right) \right|_{\vec{x}=1} = \left. \frac{1}{(r+1)!} \pa_x^r F(x) \right|_{\vec{x}=1} = \frac{\mathbf a_r}{r+1}.
$$
\end{lemma}
\begin{proof}
This is Lemma 5.5 in \cite{BG}.
\end{proof}

\begin{lemma}
\label{lem:basic-2}
We have
\begin{multline}
\label{eq:lemma-basic-2}
\lim_{x_2, \dots, x_{r+1} \to 1} \left( \frac{\mathbf F(x_1)}{(x_1-x_2) (x_1-x_3) \dots (x_1-x_{r+1})} + \frac{\mathbf F(x_2)}{(x_2-x_1) (x_2-x_3) \dots (x_2-x_{r+1})} + \dots \right. \\ \left. + \frac{\mathbf F(x_{r+1})}{(x_{r+1}-x_1) (x_{r+1}-x_2) \dots (x_{r+1}-x_{r})} \right)  = \frac{\mathbf F(x_1) - \mathbf a_0 - \mathbf a_1 (x_1-1) - \dots - \mathbf a_{r-1} (x_1-1)^{r-1}}{(x_1-1)^r}.
\end{multline}
Note that we do not set the value of the variable $x_1$ in the left-hand side.
\end{lemma}
\begin{proof}
In the left-hand side of \eqref{eq:lemma-basic-2} the first term has a limit as $x_2, \dots, x_{r+1} \to 1$, and the sum of other terms has a limit by Lemma \ref{lem:basic-1} applied to the function $\mathbf F(x) / (x-x_1)$. We obtain that the left-hand side of \eqref{eq:lemma-basic-2} equals
\begin{multline*}
\frac{\mathbf F(x_1)}{(x_1-1)^r} + \frac{1}{(r-1)!} \pa_x^{r-1} \left. \left[ \frac{\mathbf F (x)}{x-x_1} \right] \right|_{x=1} = \frac{\mathbf F(x_1)}{(x_1-1)^r} +  \sum_{k=0}^{r-1} \binom{r-1}{k} \\ \left. \times \frac{ \pa_x^k \left[ \mathbf F(x) \right] (-1)^{r-1-k} (r-1-k)!} {(x-x_1)^{r-1-k}} \right|_{x=1} = \frac{\mathbf F(x_1) - \mathbf a_0 - \mathbf a_1 (x_1-1) - \dots - \mathbf a_{r-1} (x_1-1)^{r-1}}{(x_1-1)^r}.
\end{multline*}
\end{proof}

\subsection{Expectation-contributing terms}
\label{sec:5-2}

Let us introduce notations which we will use in the rest of Section \ref{sec:5}. Let $\rho = \{ \rho_N \}$ be an appropriate sequence of measures on $\GT_N$ with the Schur generating function $S_N = S_{\rho_N} (x_1, \dots, x_N)$, and limiting functions $F_{\rho}(x)$, $G_{\rho} (x,y)$, and $Q_{\rho} (x,y)$ (see Definition \ref{def:main}). In this section we will analyze expressions which eventually contribute to the leading order of the expectation of the moments of the measure $\rho_N$.

For an integer $l>0$ let us introduce the notation
\begin{equation}
\label{eq:def-Fl}
\FF_{(l)} (\vec{x}) := \frac{1}{S_N (\vec{x}) V_N (\vec{x}) } \sum_{i=1}^N \left( x_i \pa_i \right)^{l} V_N (\vec{x}) S_N (\vec{x}).
\end{equation}

\begin{lemma}
\label{lem:sum-Schur-Vand-afterDiff}
The following statements hold:

a) The functions $\FF_{(l)} (\vec{x})$ have $N$-degree at most $l+1$.

b) For any index $i$ the functions $\pa_i \FF_{(l)} (\vec{x})$ have $N$-degree at most $l$.

c) For any indices $i \ne j$ the functions $\pa_i \pa_j \FF_{(l)} (\vec{x})$ have $N$-degree less than $l$.

\end{lemma}
\begin{proof}
Since $S_N (1^N)=1$, the function $\log S_N$ is well-defined in a neighborhood of $(1^N)$
and we can rewrite \eqref{eq:def-Fl} in the following form:
\begin{equation*}
\FF_{(l)} (x_1, \dots, x_N) := \frac{1}{S_N V_N} \sum_{i=1}^N \left( x_i \pa_i \right)^{l} V_N \exp( \log S_N).
\end{equation*}
We will write the result of the application of the differential operator $\pa_i$ to $\exp( \log S_N)$ in the form
\begin{equation}
\label{eq:log-differenciruem}
\pa_i S_N = \pa_i \exp( \log S_N) = \pa_i [ \log S_N] \exp( \log S_N).
\end{equation}
After the application of all differential operators in \eqref{eq:def-Fl} in this fashion we can cancel $S_N$ in the numerator and the denominator and write $\FF_{(l)} (\vec{x})$ as a large sum of factors of the form
\begin{equation}
\label{eq:term-example}
\frac{c_0 x_i^{l-s_0} (\pa_i^{s_1} [ \log S_N ])^{d_1} \dots (\pa_i^{s_t} [ \log S_N ])^{d_t}}{(x_i-x_{a_1}) \dots (x_i-x_{a_r})},
\end{equation}
where $i, a_1, \dots, a_r$ are distinct indices, $\{ s_j \}$ and $\{ d_j \}$ are nonnegative integers such that $s_1 < s_2 <\dots < s_t$ and
\begin{equation}
\label{eq:comb-assump-L55}
r+s_0 + s_1 d_1 + \dots + s_t d_t = l,
\end{equation}
and $c_0$ depends on $r, \{ s_j \}, \{ d_j \}$, but does not depend on $N$ or $a_1, \dots, a_r$.
Since the operator $\sum_{i=1}^N (x_i \pa_i)^l$ is symmetric, all terms obtained from \eqref{eq:term-example} by permuting variables $x_i, x_{a_1}, \dots, x_{a_r}$ are present in our sums. Therefore, $\FF_{(l)} (\vec{x})$ can be represented as a sum
\begin{multline}
\label{eq:F-l-sum-all-terms}
\FF_{(l)} (\vec{x}) = \sum_{r, \{ s_j \}, \{ d_j \} } (r+1)! \\ \times \sum_{\{a_1, \dots, a_{r+1}\} \subset [N]} Sym_{a_1, \dots, a_{r+1}} \left( \frac{c_0 x_{a_1}^{l-s_0} (\pa_{a_1}^{s_1} [ \log S_N ])^{d_1} \dots (\pa_{a_1}^{s_t} [ \log S_N ])^{d_t}}{(x_{a_1}-x_{a_2}) \dots (x_{a_1}-x_{a_{r+1}})} \right),
\end{multline}
where the first sum is subject to \eqref{eq:comb-assump-L55}, and we omitted the dependence of $c_0$ on $r, \{ s_j \}, \{ d_j \}$.

Let us now prove three 3 statements of Lemma \ref{lem:sum-Schur-Vand-afterDiff}.

{\bf a}) First, let us consider the asymptotics of the expression
\begin{equation*}
\left. Sym_{a_1, \dots, a_{r+1}} \left( \frac{c_0 x_{a_1}^{l-s_0} (\pa_{a_1}^{s_1} [ \log S_N ])^{d_1} \dots (\pa_{a_1}^{s_t} [ \log S_N ])^{d_t}}{(x_{a_1}-x_{a_2}) \dots (x_{a_1}-x_{a_{r+1}})} \right) \right|_{\vec{x}=1}.
\end{equation*}
Note that each factor $\left( \pa_{a_1}^{s_1} [ \log S_N ] \right)^{d_1}$ has $N$-degree at most $d_1$, since $\rho_N$ is an appropriate sequence. Therefore, Lemma \ref{lem:funct-raznosti} and equality \eqref{eq:comb-assump-L55} imply that this function has $N$-degree $l-r$ at most. The expression
$$
\left. \sum_{\{a_1, \dots, a_{r+1}\} \subset [N] } Sym_{a_1, \dots, a_{r+1}} \left( \frac{c_0 x_{a_1}^{l-s_0} (\pa_{a_1}^{s_1} [ \log S_N ])^{d_1} \dots (\pa_{a_1}^{s_t} [ \log S_N ])^{d_t}}{(x_{a_1}-x_{a_2}) \dots (x_{a_1}-x_{a_{r+1}})} \right) \right|_{\vec{x}=1}
$$
contains $O(N^{r+1})$ terms of this form; therefore, it has $N$-degree at most $N^{l+1}$.

{\bf b}) We are interested in the asymptotics of the expression
\begin{equation}
\label{eq:term-symmetr}
\pa_i  \sum_{\{a_1, \dots, a_{r+1}\} \subset [N]} \left. Sym_{a_1, \dots, a_{r+1}} \left( \frac{c_0 x_{a_1}^{l-s_0} (\pa_{a_1}^{s_1} [ \log S_N ])^{d_1} \dots (\pa_{a_1}^{s_t} [ \log S_N ])^{d_t}}{(x_{a_1}-x_{a_2}) \dots (x_{a_1}-x_{a_{r+1}})} \right) \right|_{\vec{x}=1}.
\end{equation}
Let us consider two cases.

b1) First, consider a term
\begin{equation}
\label{eq:Fl-one-deriv}
\pa_i \left. Sym_{a_1, \dots, a_{r+1}} \left( \frac{c_0 x_{a_1}^{l-s_0} (\pa_{a_1}^{s_1} [ \log S_N ])^{d_1} \dots (\pa_{a_1}^{s_t} [ \log S_N ])^{d_t}}{(x_{a_1}-x_{a_2}) \dots (x_{a_1}-x_{a_{r+1}})} \right) \right|_{\vec{x}=1}
\end{equation}
with $i \notin \{a_1, \dots, a_{r+1} \}$. Note that there are $O( N^{r+1})$ such terms. We need to apply the operator $\pa_i$ to one of the factors $\pa_{a_1}^{s_q} \log S_N$, because only these factors depend on $x_i$ in this case. Note that
\begin{equation*}
\pa_i (\pa_{a_1}^{s_q} \log S_N)^{d_q} = d_q (\pa_{a_1}^{s_q} \left[ \log S_N \right])^{d_q -1} \pa_i \left[ \pa_{a_1}^{s_q} \log S_N \right]
\end{equation*}
has $N$-degree at most $d_q-1$.
Assume that the operator $\pa_i$ is applied to $(\pa_{a_1}^{s_1} [ \log S_N ])^{d_1}$ (other cases can be considered analogously). Then the expression can be written as a sum of terms of the form
\begin{equation*}
\left. Sym_{a_1, \dots, a_{r+1}} \left( \frac{c_0 x_{a_1}^{l-s_0} \pa_i \left[ (\pa_{a_1}^{s_1} [ \log S_N ])^{d_1} \right] \dots (\pa_{a_1}^{s_t} [ \log S_N ])^{d_t}}{(x_{a_1}-x_{a_2}) \dots (x_{a_1}-x_{a_{r+1}})} \right) \right|_{\vec{x}=1}.
\end{equation*}
Lemma \ref{lem:funct-raznosti} asserts that this sum has $N$-degree at most $(d_1-1) +d_2 + d_3 + \dots + d_t$. Recall that $r+s_0 + s_1 d_1 + \dots + s_t d_t = l$. We see that the maximum of $N$-degree is achieved at $t=1$, $s_0=0$, $s_1=1$, $d_1=l-r$. It follows that the expression \eqref{eq:Fl-one-deriv} has $N$-degree at most $l-r-1$. Taking into account that there are $O(N^{r+1})$ terms of such a form, we obtain that the sum has $N$-degree at most $l$.

b2) Now let us consider the term of the form \eqref{eq:Fl-one-deriv} with $i \in \{a_1, \dots, a_{r+1} \}$. Since $i$ is fixed, there are $O(N^r)$ terms of such form. Note that since the function
$$
\left. Sym_{a_1, \dots, a_{r+1}} \left( \frac{c_0 x_{a_1}^{l-s_0} (\pa_{a_1}^{s_1} [ \log S_N ])^{d_1} \dots (\pa_{a_1}^{s_t} [ \log S_N ])^{d_t}}{(x_{a_1}-x_{a_2}) \dots (x_{a_1}-x_{a_{r+1}})} \right) \right|_{\vec{x}=1}
$$
has $N$-degree at most $l-r$, then its derivative also has degree at most $l-r$. Therefore, the sum of all terms of such a form has $N$-degree at most $l$, which concludes the proof of the claim b).

{\bf c}) We are interested in the asymptotics of the expression
\begin{equation}
\label{eq:term-symmetr}
\pa_i \pa_j \sum_{\{a_1, \dots, a_{r+1}\} \subset [N]} \left. Sym_{a_1, \dots, a_{r+1}} \left( \frac{c_0 x_{a_1}^{l-s_0} (\pa_{a_1}^{s_1} [ \log S_N ])^{d_1} \dots (\pa_{a_1}^{s_t} [ \log S_N ])^{d_t}}{(x_{a_1}-x_{a_2}) \dots (x_{a_1}-x_{a_{r+1}})} \right) \right|_{\vec{x}=1}.
\end{equation}
Again, let us consider several cases related to whether indices $i$ and $j$ are from $\{a_1, \dots, a_{r+1}\}$ or not.

c1) If both indices $i$ and $j$ are outside of $\{a_1, \dots, a_{r+1}\}$, and both differentitations $\pa_i$ and $\pa_j$ are applied to same $\log S_N$. Since $\pa_{a} \pa_i \pa_j \log S_N$ has $N$-degree less than 0, the same considerations as in the case b1) imply the statement of proposition.

c2) If both indices are outside of $\{a_1, \dots, a_{r+1}\}$, and these differentiations are applied to different $\pa_{a_1} [\log S_N]$. It is easy to see that in this case all terms have $N$-degree at most $l-1$ which is even stronger than we need.

c3) If $i \in \{a_1, \dots, a_{r+1}\}$ and $j$ is outside of this set, then we lose one degree of $N$ in the summation over sets of indices and another degree when we differentiate $\log S_N$. Therefore, all these terms have $N$-degree at most $l-1$, what is stronger than we need.

c4) If $i,j \in \{a_1, \dots, a_{r+1}\}$, then we lose two degrees in the summation over sets of
indices. Again, all such terms give contribution $N$-degree ${l-1}$ at most. This concludes the
proof of the lemma.
\end{proof}

\begin{remark}
Note that we have
\begin{multline}
\label{eq:Fl-main-terms}
\pa_i \FF_{(l)} (\vec{x}) = \pa_i \left[ \sum_{r=0}^{l} \binom{l}{r} (r+1)! \right. \\ \left. \times \sum_{\{a_1, \dots, a_{r+1}\} \subset [N]} Sym_{a_1, \dots, a_{r+1}} \left( \frac{x_{a_1}^{l} (\pa_{a_1} [ \log S_N ])^{l-r}}{(x_{a_1}-x_{a_2}) \dots (x_{a_1}-x_{a_{r+1}})} \right) \right] + \hat T_{(l)} (\vec{x}),
\end{multline}
where the function $\hat T_{(l)} (\vec{x})$ has $N$-degree less than $l$. Indeed, the proof of Lemma \ref{lem:sum-Schur-Vand-afterDiff} shows that the highest $N$-degree is obtained in the case $s_1=1$, $d_1+r=l$, $s_2=s_3= \dots = 0$. A coefficient $\binom{l}{r}$ appears because we need to apply $l-r$ differentiations to $\exp ( \log (S_N))$ and $r$ differentiations to $V_N$.
\end{remark}

\subsection{Covariance-contributing terms}

For positive integers $l_1, l_2$ let us define one more function by
\begin{multline}
\label{eq:def-GGfunc}
\GG_{(l_1, l_2)} (\vec{x}) := l_1 \sum_{r=0}^{l_1-1} \binom{l_1-1}{r} \sum_{\{a_1, \dots, a_{r+1} \} \subset [N] } (r+1)! \\ \times Sym_{a_1, \dots, a_{r+1}} \frac{ x_{a_1}^{l_1} \pa_{a_1} \left[ \FF_{(l_2)} \right] \left( \pa_{a_1} \left[ \log S_N \right] \right)^{l_1-1-r}} {(x_{a_1} -x_{a_2} ) \dots (x_{a_1}-x_{a_{r+1}})}.
\end{multline}

The meaning of this function is given by the next lemma; essentially, this lemma describes the covariance in our probability models.

\begin{lemma}
\label{lem:two-diff-terms}
For any positive integers $l_1, l_2$ we have
\begin{equation}
\label{eq:two-diff-terms}
\frac{1}{V_N S_N} \sum_{i_1=1}^N (x_{i_1} \pa_{i_1})^{l_1} \sum_{i_2=1}^N (x_{i_2} \pa_{i_2})^{l_2} \left[ V_N S_N \right] = \FF_{(l_1)} (\vec{x}) \FF_{(l_2)} (\vec{x}) + \GG_{(l_1, l_2)} (\vec{x}) + \tilde T (\vec{x}),
\end{equation}
where $\GG_{(l_1, l_2)} (\vec{x})$ has $N$-degree at most $l_1+l_2$, and $\tilde T (\vec{x})$ has $N$-degree less than $l_1+l_2$.
\end{lemma}
\begin{proof}
The left-hand side of \eqref{eq:two-diff-terms} can be written as
$$
\frac{1}{V_N S_N} \sum_{i_1=1}^N (x_{i_1} \pa_{i_1})^{l_1} \left[ V_N S_N \FF_{(l_2)} (\vec{x}) \right].
$$
Applying differentiations $\pa_{i_1}$ with the use of \eqref{eq:log-differenciruem}, we can rewrite it as the sum of terms of the form
$$
Sym_{a_1, \dots, a_{r+1}} \frac{ c_0 x_{a_1}^{l_1-s_0} \pa_{a_1}^{s_1} \left[ \FF_{(l_2)} \right] \left( \pa_{a_1}^{s_2} \left[ \log S_N \right] \right)^{d_2} \dots \left( \pa_{a_1}^{s_t} \left[ \log S_N \right] \right)^{d_t}}{(x_{a_1} -x_{a_2} ) \dots (x_{a_1}-x_{a_{r+1}})},
$$
for nonnegative integers $r$, $s_0, s_1, \dots, s_t$, $d_2, \dots, d_t$, such that $s_2 < s_3 < \dots < s_t$ and
\begin{equation}
\label{eq:tech-est-3}
s_0+s_1 + s_2 d_2 + \dots + s_t d_t + r = l_1.
\end{equation}
From terms with $s_1=0$ we obtain $FF_{(l_1)} (\vec{x}) \FF_{(l_2)} (\vec{x})$. Let us deal with other terms.

Let us estimate $N$-degree of all terms with fixed collection of numbers $r$, $s_0, s_1, \dots, s_t$, $d_2, \dots, d_t$. Lemma \ref{lem:sum-Schur-Vand-afterDiff} asserts that $\pa_{a_1}^{s_1} \left[ \FF_{(l_2)} \right]$ has $N$-degree at most $l_2$ since $s_1 \ge 1$. Therefore, the total $N$-degree of these terms is at most $l_2 +d_2 + \dots + d_t + (r+1)$ (as usual, we apply Lemma \ref{lem:funct-raznosti} here). Given \eqref{eq:tech-est-3} and $s_1 \ge 1$, it is clear that this number is maximal for $s_0=0$, $s_1=1$, $s_2=1$, $d_2=l_1-1-r$; for this choice of parameters our sum of terms has $N$-degree at most $l_1 + l_2$, and for all other terms the expression $l_2 +d_2 + \dots + d_t + (r+1)$ is smaller and the total contribution of all other terms have $N$-degree less than $l_1+l_2$.

The terms with $s_0=0$, $s_1=1$, $s_2=1$, $d_2=l_1-1-r$ are exactly those which are present in the expression \eqref{eq:def-GGfunc}.
\end{proof}

\begin{lemma}
\label{lem:deg-Gfunc}
The function $\GG_{(l_1,l_2)} (\vec{x})$ has $N$-degree at most $l_1+l_2$. For any index $i$ the function $ \pa_i \GG_{(l_1,l_2)} (\vec{x})$ has $N$-degree less than $l_1+l_2$.
\end{lemma}
\begin{proof}
The first statement was proven in the previous lemma. We know that $\GG_{(l_1,l_2)} (\vec{x})$ is the sum of terms
$$
Sym_{a_1, \dots, a_{r+1}} \frac{ x_{a_1}^{l_1} \pa_{a_1} \left[ \FF_{(l_2)} \right] \left( \pa_{a_1} \left[ \log S_N \right] \right)^{l_1-1-s}} {(x_{a_1} -x_{a_2} ) \dots (x_{a_1}-x_{a_{r+1}})},
$$
over $r =0,1, \dots, l_1-1$, and all sets $\{a_1, \dots, a_{r+1} \} \subset \{1,\dots, N \}$. When we differentiate the sum of these terms by $\pa_i$, we need to consider two cases. First, the terms with $i$ inside $\{a_1, \dots, a_{r+1} \}$ has $N$-degree at most $l_1+l_2-1$, because the index $i$ is fixed and the total number of terms has smaller order in $N$. Second, if $i$ is outside of $\{a_1, \dots, a_{r+1} \}$, then $\pa_i$ should be applied to $\pa_{a_1} \left[ \FF_{(l_2)} \right]$ or $\pa_{a_1} \left[ \log S_N \right]$. By Lemma \ref{lem:sum-Schur-Vand-afterDiff} $\pa_i \pa_{a_1} \left[ \FF_{(l_2)} \right]$ has $N$-degree less than $l_2$, and our conditions on $\log S_N$ imply that $\pa_1 \pa_{a_1} \left[ \log S_N \right]$ has $N$-degree less than 1. Therefore, for these terms the $N$-degree also decreases due to this differentiation; we obtain that the total $N$-degree of the expression is less than $l_1+l_2$.
\end{proof}

\begin{remark}
\label{rem:terms-for-covariance}
The proof of Lemma \ref{lem:two-diff-terms} shows that
$$
\frac{1}{V_N (\vec{x}) S_N ( \vec{x})} l_1 \sum_{i=1}^N \left( x_i \pa_i \left[ \FF_{(l_2)} (\vec{x}) \right] \right) \left( x_i \pa_i \right)^{l_1-1} \left[ V_N (\vec{x}) S_N ( \vec{x}) \right] = \GG_{(l_1, l_2)} (\vec{x}) + \bar T_{(l_1+l_2)} (\vec{x}),
$$
where $\bar T_{(l_1+l_2)} (\vec{x})$ has $N$-degree less than $l_1+l_2$.
\end{remark}

\subsection{Product of several moments}
\label{sec:5-4}

For a positive integer $s$ and a subset $\{ j_1, \dots, j_p \} \in [s]$ we denote by $\mathcal P^s_{j_1, \dots, j_p}$ the set of \textit{all pairings} of the set $\{1,2, \dots, s \} \backslash \{j_1, \dots, j_p\}$. In particular, this set is empty if $\{1,2, \dots, s \} \backslash \{j_1, \dots, j_p\}$ has odd number of elements. We will also need the notation $\mathcal P^{2;s}_{j_1, \dots, j_p}$ which stands for the set of all pairings of $\{2, \dots, s \} \backslash \{j_1, \dots, j_p\}$. For a pairing $P$ we denote by $\prod_{(a,b) \in P}$ the product over all pairs $(a,b)$ from this pairing.

\begin{proposition}
\label{lem:gauss-many-Schur}
For any positive integer $s$ and any positive integers $l_1, \dots, l_s$ we have
\begin{multline}
\label{eq:gauss-many-Schur}
\frac{1}{V_N S_N} \sum_{i_1=1}^N (x_{i_1} \pa_{i_1})^{l_1} \sum_{i_2=1}^N (x_{i_2} \pa_{i_2})^{l_2} \dots \sum_{i_s=1}^N (x_{i_s} \pa_{i_s})^{l_s} \left[ V_N S_N \right] \\ = \sum_{p=0}^s \sum_{ \{ j_1, \dots, j_p \} \in [s] } \FF_{(l_{j_1})} (\vec{x}) \dots \FF_{(l_{j_p})} (\vec{x}) \left( \sum_{ P \in \mathcal P^s_{j_1, \dots, j_p} } \prod_{(a,b) \in P} \GG_{(l_a, l_b)} (\vec{x}) + \tilde T_{j_1, \dots, j_p}^{1;s} (\vec{x}) \right),
\end{multline}
where $\tilde T_{j_1, \dots, j_p}^{1;s} (\vec{x})$ has $N$-degree less than $\sum_{i=1}^s l_i - \sum_{i=1}^p l_{j_i}$.
\end{proposition}
\begin{proof}
We will prove this statement by induction over $s$. For $s=1$ the statement follows from definition \eqref{eq:def-Fl}. For $s=2$ it follows from Lemma \ref{lem:two-diff-terms}. Assume that we already proved it for $s-1$. Let us apply the operators $ \sum_{i_s=1}^N (x_{i_s} \pa_{i_s})^{l_s}$, $\dots$, $\sum_{i_2=1}^N (x_{i_2} \pa_{i_2})^{l_2}$, and use the induction assumption.
We need to analyze the expression
\begin{multline*}
\frac{1}{V_N S_N} \left( \sum_{i_1=1}^N (x_{i_1} \pa_{i_1})^{l_1} \right) \left[ V_N S_N \sum_{p=0}^{s-1} \sum_{ \{ j_1, \dots, j_p \} \in [2;s] } \FF_{(j_1)} (\vec{x}) \dots \FF_{(j_p)} (\vec{x})  \right. \\ \left. \times \left( \sum_{ P \in \mathcal P^{[2;s]}_{j_1, \dots, j_p} } \prod_{(a,b) \in P} \GG_{(k_a, k_b)} (\vec{x}) + \tilde T_{j_1, \dots, j_p}^{2;s} (\vec{x}) \right) \right],
\end{multline*}
for any choice of the set of indices $J_{old} := \{ j_1, \dots, j_p \} \subset [2;s]$. Note that an induction hypothesis asserts that $\tilde T_{j_1, \dots, j_p}^{2;s}$ has $N$-degree less than $\sum_{i=2}^s l_i - \sum_{i=1}^p l_{j_i}$.
Let us consider several cases to analyze all arising terms.

{ \bf 1}) All differentiations $\pa_{i_1}$ are applied to $V_N S_N$ or $x_{i_1}$ from $(x_{i_1} \pa_{i_1})^{l_1}$. By definition, these terms give rise to the function $\FF_{(l_1)}$. The terms obtained in this way have the required form with the set of indices $J_{new} := J_{old} \cup \{1 \}$.

{ \bf 2}) One differentiation $\pa_{i_1}$ is applied to the function $\FF_{j_w}$ for some $w$, and all other differentiations $\pa_{i_1}$ are applied to $V_N S_N$. Using Remark \ref{rem:terms-for-covariance}, we see that these terms have the required form with $J_{new} := J_{old} \backslash \{ w \}$ and the arising function $\GG_{(l_1, l_{j_w})} (\vec{x})$ in the product of the pairings.

{ \bf 3}) Consider all other terms. We will show that they do not contribute to the leading order.
We fix the set $\{ j_1, \dots, j_p \} \subset [2;s]$.
Let us define the function
$$
\tilde H_{j_1, \dots, j_p} (\vec{x}) :=  \left( \sum_{ P \in \mathcal P^{[2;s]}_{j_1, \dots, j_p} } \prod_{(a,b) \in P} \GG_{(l_a, l_b)} (\vec{x}) + \tilde T_{j_1, \dots, j_p}^{2;s} (\vec{x}) \right).
$$
From Lemma \ref{lem:deg-Gfunc} it follows that $\tilde H := \tilde H_{j_1, \dots, j_p} (\vec{x})$ has $N$-degree at most $\sum_{i=2}^s l_i - \sum_{i=1}^p l_{j_p}$, but for any index $a$ the function $\pa_a \tilde H$ has a $N$-degree less than $\sum_{i=2}^s l_i - \sum_{i=1}^p l_{j_p}$.

We analyze the expression
$$
\frac{1}{V_N S_N} \left( \sum_{i_1=1}^N (x_{i_1} \pa_{i_1})^{l_1} \right) V_N S_N \FF_{(j_1)} (\vec{x}) \dots \FF_{(j_p)} (\vec{x}) \tilde H (\vec{x}).
$$

As before, we can write the result of the application of our differential operator as a sum of terms of the form
$$
Sym_{a_1, \dots, a_{r+1}} \frac{ x_{a_1}^{k_1-s_0} \left( \pa_{a_1}^{s_1} \left[ \log S_N \right] \right)^{d_1} \dots \left( \pa_{a_1}^{s_t} \left[ \log S_N \right] \right)^{d_t} \pa_{a_1}^{f_1} \left[ \FF_{(k_{j_1})} \right] \dots \pa_{a_1}^{f_p} \left[ \FF_{(k_{j_p})} \right] \pa_{a_1}^{h_0} \left[ \tilde H (\vec{x}) \right] }{(x_{a_1} -x_{a_2} ) \dots (x_{a_1}-x_{a_{r+1}})}.
$$
Since $\left( \pa_{a_1}^{s_1} \left[ \log S_N \right] \right)^{d_1} \dots \left( \pa_{a_1}^{s_t} \left[ \log S_N \right] \right)^{d_t}$ has $N$-degree at most $d_1+d_2+\dots+d_t$, it is easy to see that the highest $N$-degree terms are present in the expression
\begin{equation}
\label{eq:term-form-gauss}
Sym_{a_1, \dots, a_{r+1}} \frac{ x_{a_1}^{l_1} \left( \pa_{a_1} \left[ \log S_N \right] \right)^{d_1} \pa_{a_1}^{f_1} \left[ \FF_{(k_{j_1})} \right] \dots \pa_{a_1}^{f_p} \left[ \FF_{(k_{j_p})} \right] \pa_{a_1}^{h_0} \left[ \tilde H (\vec{x}) \right] }{(x_{a_1} -x_{a_2} ) \dots (x_{a_1}-x_{a_{r+1}})},
\end{equation}
where
\begin{equation}
\label{eq:equal-ind-gauss}
d_1 + f_1 + \dots + f_p + h_0 + r = l_1.
\end{equation}
Let us estimate the $N$-degree of this expression for fixed $d_1,f_1,\dots,f_p,h_0,r$.

Let $B$ be the set of indices $i \in \{1, \dots, p\}$ such that $f_i = 0$. Then this term is the product of $\prod_{i \in B} \FF_{(l_i)}$ and a certain symmetric function. Our goal is to show that the $N$-degree of this symmetric function can be estimated as less than $\sum_{i=1}^s l_i - \sum_{i \in B} l_{i}$, with the exception of cases 1) and 2) considered above, which means that this symmetric function is a part of $\tilde T_B (\vec{x})$.

The function $\pa_{a_1} \left[ \log S_N \right]^{d_1}$ has $N$-degree at most $d_1$. The summation
over indices contributes the $N$-degree $r+1$. If $f_i \ne 0$, then $\pa_{a_1}^{f_i} \left[
\FF_{(l_{j_i})} \right]$ has $N$-degree at most $l_{j_i}$. This and \eqref{eq:equal-ind-gauss}
means that if two different $f_i$ are not equal to 0, then the result has $N$-degree at most
$\sum_{i=1}^s l_i - \sum_{i \in B} l_{i} - 1$. However, if $h_0$ is greater than 0, then we obtain
the total $N$-degree less than $\sum_{i=1}^s l_i - \sum_{i \in B} l_{i}$. Therefore, if the term
\eqref{eq:term-form-gauss} contributes to the degree $\sum_{i=1}^s l_i - \sum_{i \in B} l_{i}$,
then $h_0=0$ and only one of the indices $f_i$ can be equal to non zero. This leaves out only two
possibilities: if all $f_i$ are equal to 0, then we are in the case 1) considered above, and if one
of $f_i$ is not equal to 0, then we are in the case 2) considered above. This concludes the proof
of the proposition.
\end{proof}

\subsection{Gaussian behavior}
\label{sec:5-5}

For a positive integer $l$ let us set:
\begin{equation}
\label{eq:def-e-l-const}
E_l := \FF_{(l)} (1^N) = \left. \frac{1}{V_N S_N} \sum_{i=1}^N (x_{i} \pa_{i})^{l} V_N S_N \right|_{\vec{x}=1}.
\end{equation}
This is the expectation of the $l$th moment of the probability measure with the Schur generating function $S_N$.

\begin{lemma}
\label{lem:covar-general}
For any positive integer $s$ and any positive integers $l_1, \dots, l_s$ we have
\begin{multline}
\label{eq:covar-general}
\frac{1}{V_N S_N} \left( \sum_{i_1=1}^N (x_{i_1} \pa_{i_1})^{l_1} - E_{l_1} \right) \left( \sum_{i_2=1}^N (x_{i_2} \pa_{i_2})^{l_2} - E_{l_2} \right) \\ \left. \times \dots \left( \sum_{i_s=1}^N (x_{i_s} \pa_{i_s})^{l_s} - E_{l_s} \right) V_N S_N \right|_{\vec{x}=1} = \left. \sum_{ P \in \mathcal P^{s}_{\emptyset} } \prod_{(a,b) \in P} \GG_{(l_a, l_b)} (\vec{x}) \right|_{\vec{x}=1} + \left. \tilde T_{\emptyset} (\vec{x}) \right|_{\vec{x}=1},
\end{multline}
where $\tilde T_{\emptyset} (\vec{x})$ has $N$-degree less than $\sum_{i=1}^s l_i$.
\end{lemma}
\begin{proof}
We use \eqref{eq:gauss-many-Schur} to compute \eqref{eq:covar-general}, and our goal is to show that the appearance of $E_{l_i}$'s cancels out all terms from the right-hand side of \eqref{eq:gauss-many-Schur} with the non-empty set $J = \{ j_1, j_2, \dots, j_p \}$, and the right-hand side of \eqref{eq:covar-general} comes from the term with the empty set $J$.

We use the following notation: Let $\{a_1, \dots, a_w\}$ be a subset of $[s]$; we denote by $\{b_1, \dots, b_{s-w}\}$ the complimentary subset such that $\{a_1, \dots, a_w\} \cup \{b_1, \dots, b_{s-w} \} = [s]$. Analogously, for $\{ j_1, \dots, j_p \} \subset \{b_1, \dots, b_{s-w} \}$ we denote by $\{ k_1, \dots, k_{s-w-p} \}$ the complementary subset such that $ \{ j_1, \dots, j_p \} \cup \{ k_1, \dots, k_{s-w-p} \} = \{b_1, \dots, b_{s-w} \}$.

Proposition \ref{lem:gauss-many-Schur} yields
\begin{multline}
\label{eq:gauss-many-Schur-2}
\frac{1}{V_N S_N} \sum_{i_{b_1}=1}^N (x_{i_{b_1}} \pa_{i_{b_1}})^{l_{b_1}} \sum_{i_{b_2}=1}^N (x_{i_{b_2}} \pa_{i_{b_2}})^{l_{b_2}} \dots \sum_{i_{b_{s-w}}=1}^N (x_{i_{b_{s-w}}} \pa_{i_{b_{s-w}}})^{l_{b_{s-w}}} V_N S_N \\ = \sum_{p=0}^s \sum_{ \{ j_1, \dots, j_p \} \subset \{b_1, \dots, b_{s-w} \} } \FF_{(l_{j_1})} (\vec{x}) \dots \FF_{(l_{j_p})} (\vec{x}) \mathcal A_{k_1, \dots, k_{s-w-p} },
\end{multline}
where
$$
A_{k_1, \dots, k_{s-w-p} } :=  \sum_{ P \in \mathcal P^{b_1,\dots, b_{s-w} }_{j_1, \dots, j_p} } \prod_{(a,b) \in P} \GG_{(l_a, l_b)} (\vec{x}) + \tilde T_{j_1, \dots, j_p}^{b_1,\dots, b_{s-w} } (\vec{x});
$$
we use an additional superscript here (in comparison with Proposition \ref{lem:gauss-many-Schur}) because we apply Proposition \ref{lem:gauss-many-Schur} to a different set of indices.
Note that $A_{k_1, \dots, k_{s-w-p} }$ does not depend on the choice of $\{j_1, \dots, j_p \}$: It depends on $\{ k_1, \dots, k_{s-w-p}  \}$ only.


Opening the parenthesis in the left-hand side of \eqref{eq:covar-general}, we write it as
\begin{multline}
\label{eq:gauss-intermed}
\sum_{ \{a_1, \dots, a_w\} \in [s] } (-1)^w E_{l_{a_1}} E_{l_{a_2}} \dots E_{l_{a_w}} \frac{1}{V_N S_N} \\ \left. \times \sum_{i_{b_1}=1}^N (x_{i_{b_1}} \pa_{i_{b_1}})^{l_{b_1}} \sum_{i_{b_2}=1}^N (x_{b_{a_2}} \pa_{i_{b_2}})^{l_{b_2}} \dots \sum_{i_{b_{s-w}}=1}^N (x_{i_{b_{s-w}}} \pa_{i_{b_{s-w}}})^{l_{b_{s-w}}} V_N S_N \right|_{\vec{x}=1}
\end{multline}
Applying \eqref{eq:gauss-many-Schur-2} and substituting $\vec{x}=(1^N)$, we see that \eqref{eq:gauss-intermed} turns into the sum of terms of the form
\begin{equation}
\label{eq:gauss-one-term-cancel}
(-1)^w E_{m_1} E_{m_2} \dots E_{m_{w+p}} \mathcal A_{k_1, \dots, k_{s-w-p} } ( 1^N),
\end{equation}
where $\{ m_1, m_2, \dots, m_{w+p} \} = \{ a_1, a_2, \dots, a_w \} \cup \{ j_1, j_2, \dots, j_p \}$, and $\{ m_1, m_2, \dots, m_{w+p} \} \cup \{ k_1, \dots, k_{s-w-p} \} = [s]$. The summation goes over all possible choices of $\{ a_1, a_2, \dots, a_w \}$ and $\{ j_1, j_2, \dots, j_p \}$.

Let us fix the set $\{ M_1, \dots, M_W \} = \{ m_1, m_2, \dots, m_{w+p} \}$. Note that the same term \eqref{eq:gauss-one-term-cancel} can be obtained for all possible choices of $a$'s and $j$'s such that the union of these sets of indices is $\{ M_1, \dots, M_W \}$; the only difference is the sign $(-1)^w$. Collecting all terms of this form, we see that the total coefficient is
$$
\binom{W}{0}- \binom{W}{1} + \dots + (-1)^{w+p} \binom{W}{W},
$$
which is always 0 unless $W=0$. Therefore, the only term which survives all cancellations in \eqref{eq:gauss-intermed} is the term with $w=0$ and $p=0$ which in combination with Proposition \ref{lem:gauss-many-Schur} implies Lemma \ref{lem:covar-general}.
\end{proof}

\begin{proposition}
\label{prop:gener-gaussianity-final}
Let $\rho_N$ be an appropriate sequence of measures on $\GT_N$, $N=1,2,\dots$. Recall that $\FF_{(l)}$ is defined in \eqref{eq:def-Fl} , $\GG_{k,l}$ is defined in \eqref{eq:def-GGfunc}, $\prod_{(a,b) \in P}$ for $P \in \mathcal P^{s}_{\emptyset}$ is defined in the beginning of Section \ref{sec:5-4}, and $E_l$ is defined in \eqref{eq:def-e-l-const}. Then
for any positive integer $s$ and any positive integers $l_1, \dots, l_s$ we have
\begin{multline}
\lim_{N \to \infty} \frac{1}{N^{l_1+\dots + l_s}} \frac{1}{V_N S_N} \left( \sum_{i_1=1}^N (x_{i_1} \pa_{i_1})^{l_1} - E_{l_1} \right) \left( \sum_{i_2=1}^N (x_{i_2} \pa_{i_2})^{l_2} - E_{l_2} \right) \\ \left. \times \dots \left( \sum_{i_s=1}^N (x_{i_s} \pa_{i_s})^{l_s} - E_{l_s} \right) V_N S_N \right|_{\vec{x}=1} = \lim_{N \to \infty} \frac{1}{N^{l_1+\dots + l_s}} \left. \sum_{ P \in \mathcal P^{s}_{\emptyset} } \prod_{(a,b) \in P} \GG_{(l_a, l_b)} (\vec{x}) \right|_{\vec{x}=1}.
\end{multline}
\end{proposition}
\begin{proof}
Passing to the limit in the equation \eqref{eq:covar-general} and using the definition of the $N$-degree of a function, we obtain from Lemma \ref{lem:covar-general} the statement of the proposition.
\end{proof}

\section{Computation of covariance}
\label{sec:6}

In this section we will compute the covariance in the setting of Theorems \ref{theorem:main-one-level}, \ref{theorem:main-projections}, \ref{theorem:main-multiplication}, \ref{th:general-for-domino}.

\subsection{Covariance for extreme characters }
\label{sec:cov-extr}
In this section we compute the covariance in the setting of Theorem \ref{theorem:main-one-level} for a special class of Schur generating functions (see equation \eqref{eq:covar-cond-extr} below). All computations of this section will be extensively used in the proof of the general result as well.

Let $F(x)$ be a complex analytic function in a neighborhood of the unity, and let
\begin{equation}
\label{eq:Alr-def}
x^l F(x)^{l-r} = \mathbf a_{0}^{l,r} + \mathbf a_{1}^{l,r} (x-1) + \dots + \mathbf a_{n}^{l,r} (x-1)^n + \dots.
\end{equation}
be the Taylor expansion of $x^l F(x)^{l-r}$ at $x=1$.

\begin{lemma}
\label{lem:a-l-r-integrals}
Assume that $x \ne 0$ is a complex number. With the above notations, we have
\begin{equation}
\label{eq:lem-a-l-r-integrals}
\sum_{r=0}^{l} \sum_{i=0}^{r-1} \binom{l}{r} \mathbf a_i^{l,r} x^{i-r} = \frac{1}{2 \pi \ii} \oint_{|y|=\eps} \frac{1}{x-y} \left( 1+ \frac{1}{y} + (1+y) F(1+y) \right)^l dy,
\end{equation}
where $\ep \ll \min(1, |x|)$.
\end{lemma}
\begin{proof}
The Cauchy integral formula yields
$$
\mathbf a_i^{l,r} = \frac{1}{2 \pi \ii} \oint_{|y|=\ep} \frac{(1+y)^l F(1+y)^{l-r}}{y^{i+1}} dy.
$$
Substituting this into the left-hand side of \eqref{eq:lem-a-l-r-integrals} and using the equalities
$$
\sum_{i=0}^{r-1} \frac{x^i}{y^{i+1}} = \frac{1 - x^r y^{-r}}{y-x},
$$
and
$$
\sum_{r=0}^l \binom{l}{r} \frac{1 - x^r y^{-r}}{x^r F(1+y)^r} = \left(1+ \frac{1}{x F(1+y)} \right)^l - \left( 1 + \frac{1}{y F(1+y)} \right)^l,
$$
we arrive at the formula
\begin{multline*}
\sum_{r=0}^{l} \sum_{i=0}^{r-1} \binom{l}{r} \mathbf a_i^{l,r} x^{i-r} = \oint_{|y|=\ep} \frac{dy }{y-x} \\ \times \left( \left( (1+y) F(1+y) +\frac{1+y}{x} \right)^l - \left( 1+ \frac{1}{y} + (1+y) F(1+y) \right)^l \right).
\end{multline*}
Note that the first term in the right-hand side has no pole inside $|y| \le \ep$ and, therefore, is equal to 0, while the second term coincides with the right-hand side of \eqref{eq:lem-a-l-r-integrals}.
\end{proof}

Let $F_1(x)$, $F_2(x)$ be analytic complex functions in a neighborhood of the unity. Let $\mathbf a_{i,[2]}^{l,r}$ we denote the coefficients determined by \eqref{eq:Alr-def} with $F(x) = F_2(x)$. Let us define the functions
\begin{equation}
\label{eq:def-Blr}
B_{l,r} (x) := \frac{x^l F_2(x)^{l-r} - \mathbf a_{0,[2]}^{l,r} - \mathbf a_{1,[2]}^{l,r} \cdot (x-1) - \dots - \mathbf a_{r-1,[2]}^{l,r} \cdot (x-1)^{r-1}}{(x-1)^r}.
\end{equation}
$$
\FFF_1 (z):= \frac{1}{z} +1 + (1+z) F_1(1+z), \qquad \FFF_2 (z):= \frac{1}{z} +1 + (1+z) F_2(1+z).
$$

\begin{lemma}
\label{lem:extremal-equality}
With the above notations, we have
\begin{multline}
\label{eq:extremal-equality}
\left. k \sum_{q=0}^{k-1} \sum_{r=0}^{l} \binom{l}{r} \binom{k-1}{q} \frac{1}{(q+1)!} \pa^q_x \left( x^k  F_1(x)^{k-1-q} B'_{l,r} (x) \right) \right|_{x=1} \\ = \frac{1}{(2 \pi \ii)^2} \oint_{|z|=\ep} \oint_{|w|=2 \ep} \FFF_1 (z)^k \FFF_2 (w)^l \frac{1}{(z-w)^2} dz dw,
\end{multline}
where the contours of integration are counter-clockwise and $\ep \ll 1$.
\end{lemma}
\begin{proof}

By the Cauchy integral formula we have
$$
\left. \pa^q_x \left( x^k F_1(x)^{k-1-q} B'_{l,r} (x) \right) \right|_{x=1} = \frac{q!}{2 \pi \ii} \oint_{|z|=\ep} \frac{(1+z)^{k} F_1 (1+z)^{k-1-q} B'_{l,r} (1+z)}{ z^{q+1}} dz.
$$

Therefore, the left-hand side of \eqref{eq:extremal-equality} can be written as
\begin{equation}
\label{eq:transform-tr-tr}
\sum_{r=0}^{l} \binom{l}{r} \frac{ k }{2 \pi \ii} \oint_{|z|=\ep} (1+z)^{k} F_1 (1+z)^{k} B'_{l,r} (1+z) \sum_{q=0}^{k-1} \binom{k-1}{q} \frac{1}{(q+1) F_1 (1+z)^{q+1} z^{q+1} } dz.
\end{equation}

The binomial theorem gives
\begin{equation}
\label{eq:binom-theor-add}
\sum_{q=0}^{k-1} \binom{k-1}{q} \frac{\left( F_1 (1+z)^{-1} z^{-1} \right)^{q+1}}{q+1} = \frac{\left( 1+F_1 (1+z)^{-1} z^{-1} \right)^{k}}{ k} - \frac{1}{k}.
\end{equation}
Plugging this expression into \eqref{eq:transform-tr-tr} and observing that the term with $-1/k$ gives zero contribution (because $(1+z)^{k} F_1 (1+z)^{k} B'_{l,r} (1+z)$ does not have a pole at zero), we obtain that the left-hand side of \eqref{eq:extremal-equality} equals
$$
\frac{1}{2 \pi \ii} \oint_{|z|=\ep} \FFF_1 (z)^k \sum_{r=0}^l \binom{l}{r} B'_{l,r}(1+z) dz.
$$
The definition \eqref{eq:def-Blr} implies that
$$
\sum_{r=0}^{l} \binom{l}{r} B_{l,r} (1+z) = \sum_{r=0}^l \binom{l}{r} \frac{(1+z)^l F_2 (1+z)^l}{z^r F_2 (1+z)^r} - \sum_{r=0}^l \binom{l}{r} \sum_{i=0}^{r-1} \frac{\mathbf a_{i,[2]}^{l,r} z^i}{z^r}.
$$
The binomial theorem and Lemma \ref{lem:a-l-r-integrals} allows to rewrite this expression in the form
$$
\FFF_2 (z)^l - \frac{1}{2 \pi \ii} \oint_{|w|=\ep/2} \frac{1}{z-w} \FFF_2 (w)^l dw.
$$

Therefore, the left-hand side of \eqref{eq:extremal-equality} can be expressed as a sum of two terms
\begin{equation}
\label{eq:covar-for-two-func-after}
\frac{1}{2 \pi \ii} \oint_{|z|=\ep} \FFF_1 (z)^k \pa_z \left[ \FFF_2 (z)^l \right] dz -
\frac{1}{(2 \pi \ii)^2} \oint_{|z|=\ep} \oint_{|w|=\ep/2} \FFF_1 (z)^k \pa_z \left[ \frac{1}{z-w} \FFF_2 (w)^l \right].
\end{equation}
Note that the second term in \eqref{eq:covar-for-two-func-after} equals
\begin{equation}
\label{eq:covar-lemma-second-term}
\frac{1}{(2 \pi \ii)^2} \oint_{|z|=\ep} \oint_{|w|=\ep/2} \FFF_1 (z)^k \FFF_2 (w)^l \frac{dz dw}{(z-w)^2}.
\end{equation}
Let us move the contour $|w|=\ep/2$ to the contour $|w|= 2 \ep$ in \eqref{eq:covar-lemma-second-term}.
In the process we get the residue at $z=w$ which cancels with the first term from \eqref{eq:covar-for-two-func-after}.

Thus, the left-hand side of \eqref{eq:extremal-equality} equals
\begin{multline*}
- \frac{1}{(2 \pi \ii)^2} \oint_{|z|=\ep} \oint_{|w|=2 \ep} \FFF_1 (z)^k \pa_z \left[ \frac{1}{z-w} \FFF_2 (w)^l \right] \\ = \frac{1}{(2 \pi \ii)^2} \oint_{|z|=\ep} \oint_{|w|=2 \ep} \FFF_1 (z)^k \FFF_2 (w)^l \frac{1}{(z-w)^2} dz dw,
\end{multline*}
which concludes the proof.
\end{proof}

We now consider a special case of Theorem \ref{theorem:main-one-level}. Let $\rho = \{ \rho_N \}$ be a sequence of probability measures, where $\rho_N$ is a probability measure on $\GT_N$, and let $\{ \c_k \}_{k \ge 1}$ be reals such that the function
$$
F(x) := \sum_{k=1}^{\infty} \frac{\c_k}{(k-1)!} (x-1)^{k-1}
$$
is well defined in a neighborhood of unity. We assume that the Schur generating function $S_N (\vec{x}) := S_{\rho_N} (\vec{x})$ has the form
\begin{equation}
\label{eq:covar-cond-extr}
S_N (\vec{x}) = \exp \left( N \sum_{i=1}^N F_N (x_i) \right),
\end{equation}
where $\{ F_N (x) \}_{N \ge 1}$ is a sequence of holomorphic functions such that
$$
\lim_{N \to \infty} \pa_x^{k} F_N (x) = \c_k, \qquad \mbox{for any $k \in \N$}.
$$
Clearly, such a Schur generating function is appropriate in the sense of Section \ref{sec:2-2} with $F_{\rho} (x) = F(x)$, $G_{\rho} (x,y) = 0$, and $Q(x,y)=\frac{1}{(x-y)^2}$.

\begin{proposition}
\label{prop:cov-extr}
Under the assumptions and in the notations of Theorem \ref{theorem:main-one-level} and additional assumption \eqref{eq:covar-cond-extr}, we have
\begin{equation*}
\lim_{N \to \infty} \frac{\cov \left( p_k^{(N)}, p_l^{(N)} \right)}{N^{k+l}} = \oint_{|z|=\ep} \oint_{|w|=2 \ep} \FFF (z)^k \FFF (w)^l \frac{1}{(z-w)^2} dz dw,
\end{equation*}
where $\FFF (z) := \frac{1}{z} +1 + (1+z) F(1+z)$.
\end{proposition}

\begin{proof}

We denote by $\approx$ the equality of highest $N$-degree.

As explained in Section \ref{sec:4}, we have
\begin{multline}
\label{eq:expect1}
\mathbf E \left( p_k^{(N)} p_l^{(N)} \right) = \left.
\frac{1}{V_N(\vec{x})} \sum_{i=1}^N (x_i \pa_i)^k \sum_{j=1}^N (x_j \pa_j)^l V_N (\vec{x}) S_N (\vec{x}) \right|_{\vec{x}=1} \\ \left. = \frac{1}{V_N(\vec{x})} \sum_{i=1}^N (x_i \pa_i)^k \sum_{j=1}^N (x_j \pa_j)^l V_N (\vec{x}) \exp \left( N \sum_{i=1}^{N} F_N (x_i) \right) \right|_{\vec{x}=1}.
\end{multline}

Therefore, Lemma \ref{lem:two-diff-terms} implies that
\begin{equation}
\label{eq:covar-ext-lim-beg}
\lim_{N \to \infty} \frac{\cov \left( p_k^{(N)}, p_l^{(N)} \right)}{N^{k+l}} = \lim_{N \to \infty} \frac{\GG_{(k,l)} (1^N)}{N^{k+l}}.
\end{equation}
Let us compute the right-hand side of this formula. By definition \eqref{eq:def-GGfunc},
\begin{multline}
\label{eq:covar-exact-form1}
\GG_{(k,l)} (1^N) \\ = k \left. \sum_{q=0}^{k-1} \sum_{\{a_1, \dots, a_{q+1} \} \subset [N] } \binom{k-1}{q} (q+1)! \ Sym_{a_1, \dots, a_{q+1}} \frac{ x_{a_1}^{k} \pa_{a_1} \left[ \FF_{(l)} \right] \left( \pa_{a_1} \left[ \log S_N \right] \right)^{k-1-q}} {(x_{a_1} -x_{a_2} ) \dots (x_{a_1}-x_{a_{q+1}})} \right|_{\vec{x}=1} \\ \approx k \sum_{q=0}^{k-1} \sum_{\{a_1, \dots, a_{q+1} \} \subset [N] } \binom{k-1}{q} (q+1)! \ Sym_{a_1, \dots, a_{q+1}} \frac{ x_{a_1}^{k} \left( \pa_{a_1} \left[ \log S_N \right] \right)^{k-1-q}} {(x_{a_1} -x_{a_2} ) \dots (x_{a_1}-x_{a_{q+1}})} \\ \left. \times \pa_{a_1} \left[ \sum_{r=0}^{l} \sum_{\{b_1, \dots, b_{r+1} \} \subset [N] } \binom{l}{r} (r+1)! \ Sym_{b_1, \dots, b_{r+1}} \frac{x_{b_1}^l \left( \pa_{b_1} \left[ \log S_N \right] \right)^{l-r}}{(x_{b_1}- x_{b_2}) \dots (x_{b_1}- x_{b_{r+1}})} \right] \right|_{\vec{x}=1}.
\end{multline}
The right-hand side of the approximate equality in \eqref{eq:covar-exact-form1} contains only leading terms from $\pa_{a_1} \left[ \FF_{(l)} \right]$, see \eqref{eq:Fl-main-terms}; it is proven by following the same arguments as in Section \ref{sec:5-2}.

Now we will use the special form \eqref{eq:covar-cond-extr} of our function $S_N$. In this case we see that $\pa_{a_1} \left[ \log S_N \right] = N F(x_{a_1})$. Therefore,
\begin{multline}
\label{eq:covar-exact-form3}
\eqref{eq:covar-exact-form1} \approx k \sum_{q=0}^{k-1} \sum_{\{a_1, \dots, a_{q+1} \} \subset [N] } \binom{k-1}{q} (q+1)! \ Sym_{a_1, \dots, a_{q+1}} \frac{ x_{a_1}^{k} F (x_{a_1})^{k-1-q} N^{k-1-q} } {(x_{a_1} -x_{a_2} ) \dots (x_{a_1}-x_{a_{q+1}})} \\ \left. \times \pa_{a_1} \left[ \sum_{r=0}^{l} \sum_{\{b_1, \dots, b_{r+1} \} \subset [N] } \binom{l}{r} (r+1)! \ Sym_{b_1, \dots, b_{r+1}} \frac{x_{b_1}^l F (x_{b_1})^{l-r} N^{l-r} }{(x_{b_1}- x_{b_2}) \dots (x_{b_1}- x_{b_{r+1}})} \right] \right|_{\vec{x}=1}.
\end{multline}
Let us analyze this expression for different $a$'s and $b$'s.
Note that we must have $a_1 \in \{b_1, \dots, b_r \}$ in order to get a non-zero contribution. Also we see that if $| \{a_1, \dots, a_{q+1} \} \cap \{b_1, \dots, b_r \}| \ge 2$, then the total $N$-degree is not greater than $(k-1-q)+(l-r)+(q+1)+(r+1)-2 = k+l-1$ ( $k-1-q$ and $l-r$ come from the power of $N$, $q+1$, $r+1$, and $-2$ come from the summation over sets of indices); therefore, these terms do not contribute to the $N$-degree $k+l$. We obtain that only terms with $\{a_1, \dots, a_{q+1} \} \cap \{b_1, \dots, b_r \} = \{a_1\}$ contribute to the limit.

For these terms we use Lemma \ref{lem:basic-2} for the symmetrization over $b$'s and obtain:
\begin{multline}
\label{eq:covar-exact-form4}
\eqref{eq:covar-exact-form3} \approx N^{k+l-q-r-1} k \sum_{q=0}^{k-1} \sum_{\{a_1, \dots, a_{q+1} \} \subset [N] } \binom{k-1}{q} (q+1)! \\ \left. \times Sym_{a_1, \dots, a_{q+1}} \frac{ x_{a_1}^{k} F_{a_1} (x_{a_1})^{k-1-q} } {(x_{a_1} -x_{a_2} ) \dots (x_{a_1}-x_{a_{q+1}})} \pa_{a_1} \left[ \sum_{r=0}^{l} r! \sum_{\{b_2, b_3, \dots, b_{r+1} \} \subset [N] } \binom{l}{r} B_{l,r} (x_{a_1}) \right] \right|_{\vec{x}=1}
\\ = N^{k+l-q-1} k \sum_{q=0}^{k-1} \sum_{\{a_1, \dots, a_{q+1} \} \subset [N] } \binom{k-1}{q} (q+1)! \\ \left. \times Sym_{a_1, \dots, a_{q+1}} \frac{ x_{a_1}^{k} F_{a_1} (x_{a_1})^{k-1-q} } {(x_{a_1} -x_{a_2} ) \dots (x_{a_1}-x_{a_{q+1}})} \left[ \sum_{r=0}^{l} \binom{l}{r} B'_{l,r} (x_{a_1}) \right] \right|_{\vec{x}=1},
\end{multline}
where we use the notation \eqref{eq:def-Blr} with $F_2 (x) = F(x)$. We also use that the summation $\sum_{\{a_1, \dots, a_{q+1} \} \subset [N] }$ contains $\approx N^{q+1} / (q+1)!$ terms. For the symmetrization over $a$'s it is enough to apply Lemma \ref{lem:basic-1}. We obtain
\begin{equation}
\label{eq:last-formula-extr-char}
\eqref{eq:covar-exact-form4} \approx N^{k+l} k \left. \sum_{q=0}^{k-1} \sum_{r=0}^{l} \binom{l}{r} \binom{k-1}{q} \frac{1}{(q+1)!} \partial_x^q \left(x^k F'(x)^{k-1-q} B'_{l,r}(x) \right) \right|_{x=1}.
\end{equation}
Thus,
\begin{equation*}
\lim_{N \to \infty} \frac{\cov \left( p_k^{(N)}, p_l^{(N)} \right)}{N^{k+l}} = k \left. \sum_{q=0}^{k-1} \sum_{r=0}^{l} \binom{l}{r} \binom{k-1}{q} \frac{1}{(q+1)!} \partial_x^q \left(x^k F'(x)^{k-1-q} B'_{l,r}(x) \right) \right|_{x=1}.
\end{equation*}
Now Lemma \ref{lem:extremal-equality} with $F_1(x)=F(x)$ and $F_2(x)=F(x)$ implies the statement of the proposition.
\end{proof}

\subsection{Computation of one-level covariance in the general case}
\label{sec:6-2}

Here we compute the covariance in Theorem \ref{theorem:main-one-level}. We use computations and arguments from the special case considered in the previous section.

\begin{proposition}
\label{prop:main-covar-comput}
Let $\rho = \{\rho_N \}$ be an appropriate sequence of measures on $\GT_N$, $N=1,2,\dots$, and corresponding to functions $F_{\rho}(x)$ and $Q_{\rho} (x,y)$. In notations of Theorem \ref{theorem:main-one-level} we have
\begin{multline*}
\lim_{N \to \infty} \frac{\cov (p_k^{(N)}, p_l^{(N)} )}{N^{k+l}} =
\frac{1}{(2 \pi \ii)^2} \oint_{|z|=\eps} \oint_{|w|=2 \eps} \left( \frac{1}{z} +1 + (1+z) F_{\rho} (1+z) \right)^k \\ \times \left( \frac{1}{w} +1 + (1+w) F_{\rho} (1+w) \right)^l  Q_{\rho} (z, w) dz dw,
\end{multline*}
\end{proposition}

\begin{proof}

For an integer $n$ we denote by $\tilde T_{(n)} (\vec{x})$ any function of $N$ variables which has $N$-degree less than $n$, and which can change from line to line.

We start our analysis with formulas \eqref{eq:covar-ext-lim-beg} and \eqref{eq:covar-exact-form1} for covariance.
Let us fix indices $\{a_1, \dots, a_{q+1} \}$ and $\{b_1, \dots, b_{r+1} \}$, and consider several cases.

{\bf 1}) Assume that $\{a_1, \dots, a_{q+1} \} \cap \{b_1, \dots, b_{r+1} \} = \varnothing$. Then
\begin{multline*}
\pa_{a_1} \left[ \sum_{r=0}^{l} \sum_{\{b_1, \dots, b_{r+1} \} \subset [N] } \binom{l}{r} (r+1)! \ Sym_{b_1, \dots, b_{r+1}} \frac{x_{b_1}^l \left( \pa_{b_1} \left[ \log S_N \right] \right)^{l-r}}{(x_{b_1}- x_{b_2}) \dots (x_{b_1}- x_{b_{r+1}})} \right]
\\ =  \sum_{r=0}^{l} \sum_{\{b_1, \dots, b_{r+1} \} \subset [N] } \binom{l}{r} (l-r) (r+1)! \ Sym_{b_1, \dots, b_{r+1}} \frac{x_{b_1}^l \left( \pa_{b_1} \left[ \log S_N \right] \right)^{l-r-1} \pa_{a_1} \pa_{b_1} \left[ \log S_N \right] }{(x_{b_1}- x_{b_2}) \dots (x_{b_1}- x_{b_{r+1}})}.
\end{multline*}
Note that the definition of an appropriate sequence of Schur generating functions implies that $\left( \pa_{b_1} \left[ \log S_N \right] \right)^{l-r-1}$ has $N$-degree at most $l-r-1$, and $\pa_{a_1} \pa_{b_1} \left[ \log S_N \right]$ has $N$-degree at most $0$. Moreover, we have
\begin{equation}
\label{eq:equiv-approp}
\left( \pa_{b_1} \left[ \log S_N \right] \right)^{l-r-1} = N^{l-r-1} F_{\rho} (x_{b_1})^{l-r-1} + \tilde T_{ (l-r-1)}, \qquad \pa_{a_1} \pa_{b_1} \left[ \log S_N \right] = G_{\rho} (x_{a_1}, x_{b_1}) + \tilde T_{(0)}.
\end{equation}
Using these equalities and Lemma \ref{lem:funct-raznosti}, we get
\begin{multline}
\label{eq:symmetr-b-equiv}
Sym_{b_1, \dots, b_{r+1}} \frac{x_{b_1}^l \left( \pa_{b_1} \left[ \log S_N \right] \right)^{l-r-1} \pa_{a_1} \pa_{b_1} \left[ \log S_N \right] }{(x_{b_1}- x_{b_2}) \dots (x_{b_1}- x_{b_{r+1}})}
\\ = Sym_{b_1, \dots, b_{r+1}} \frac{x_{b_1}^l N^{l-r-1} F_{\rho} (x_{b_1})^{l-r-1} G_{\rho} (x_{a_1}, x_{b_1}) }{(x_{b_1}- x_{b_2}) \dots (x_{b_1}- x_{b_{r+1}})} + \tilde T_{(l-r-1)}.
\end{multline}
Note that the first term in the right-hand side of \eqref{eq:symmetr-b-equiv} depends on $r+2$ variables, not on $N$ variables. The dependence on all $N$ variables is present only in $\tilde T_{(l-r-1)}$; our notion of $N$-degree and Lemma \ref{lem:funct-raznosti} guarantee that eventually this function does not contribute to the covariance.

Using \eqref{eq:equiv-approp}, \eqref{eq:symmetr-b-equiv} and Lemma \ref{lem:funct-raznosti} again, we further obtain
\begin{multline}
\label{eq:brrr1}
Sym_{a_1, \dots, a_{q+1}} \frac{ x_{a_1}^{k} \left( \pa_{a_1} \left[ \log S_N \right] \right)^{k-1-q}} {(x_{a_1} -x_{a_2} ) \dots (x_{a_1}-x_{a_{q+1}})} \\
 \times Sym_{b_1, \dots, b_{r+1}} \frac{(l-r) x_{b_1}^l \left( \pa_{b_1} \left[ \log S_N \right] \right)^{l-r-1} \pa_{a_1} \pa_{b_1} \log S_N }{(x_{b_1}- x_{b_2}) \dots (x_{b_1}- x_{b_{r+1}})} \\=
Sym_{a_1, \dots, a_{q+1}} \frac{ x_{a_1}^{k} N^{k-1-q} F_{\rho} (x_{a_1})^{k-1-q}} {(x_{a_1} -x_{a_2} ) \dots (x_{a_1}-x_{a_{q+1}})}
 \\ \times Sym_{b_1, \dots, b_{r+1}} \frac{(l-r) x_{b_1}^l N^{l-r-1} F_{\rho} (x_{b_1})^{l-r-1} G(x_{a_1}, x_{b_1}) }{(x_{b_1}- x_{b_2}) \dots (x_{b_1}- x_{b_{r+1}})} + \tilde T_{(k+l-r-q-2)} (\vec{x}).
\end{multline}
The summation over non-intersecting sets $\{a_1, \dots, a_{q+1} \}$ and $\{b_1, \dots, b_{r+1} \}$ in \eqref{eq:covar-exact-form1} contributes the $N^{q+r+2}$ terms. Applying Lemma \ref{lem:basic-1} to \eqref{eq:brrr1} and using equality $(l-r) \binom{l}{r} = l \binom{l-1}{r}$, we see that the case of non-intersecting indices contributes the term
\begin{multline}
\label{eq:differ-ind-res1}
N^{k+l} \sum_{q=0}^{k-1} \sum_{r=0}^{l-1} \frac{k l}{(q+1)! (r+1)! } \binom{l-1}{r} \binom{k-1}{q} \\ \left. \times \pa_1^q \left[ \pa_2^r G_{\rho} (x_1, x_2) x_{1}^k F_{\rho} (x_1)^{k-1-q} x_{2}^l F_{\rho} (x_2)^{l-1-r} \right] \right|_{x_1=1, x_2=1}
\end{multline}
into the leading order. With the use of the Cauchy integral formula and the binomial theorem ( which is applied in the same way as in \eqref{eq:binom-theor-add}) one can write it in the form
\begin{multline}
\label{eq:covar-first-case-rez}
\frac{N^{k+l}}{(2 \pi \ii)^2} \oint_{|z|=\eps} \oint_{|w|=2 \eps} \left( \frac{1}{z} +1 + (1+z) F_{\rho} (1+z) \right)^k \left( \frac{1}{w} +1 + (1+w) F_{\rho} (1+w) \right)^l \\ \times G_{\rho} (1+z,1+w) dz dw.
\end{multline}

{\bf 2}) Assume that $|\{a_1, \dots, a_{q+1} \} \cap \{b_1, \dots, b_{r+1} \}|=1$. Without loss of generality we can assume that $a_1=b_1$, and all other indices are distinct. Similarly to the case 1), one can use equality \eqref{eq:equiv-approp} to show that
\begin{multline}
\label{eq:covar-333}
Sym_{a_1, \dots, a_{q+1}} \frac{ x_{a_1}^{k} \left( \pa_{a_1} \left[ \log S_N \right] \right)^{k-1-q}} {(x_{a_1} -x_{a_2} ) \dots (x_{a_1}-x_{a_{q+1}})} \\ \times \pa_{a_1} \left[ Sym_{b_1, \dots, b_{r+1}} \frac{x_{b_1}^l \left( \pa_{b_1} \left[ \log S_N \right] \right)^{l-r} }{(x_{b_1}- x_{b_2}) \dots (x_{b_1}- x_{b_{r+1}})} \right] =
Sym_{a_1, \dots, a_{q+1}} \frac{ x_{a_1}^{k} N^{k-1-q} F_{\rho} (x_{a_1})^{k-1-q}} {(x_{a_1} -x_{a_2} ) \dots (x_{a_1}-x_{a_{q+1}})} \\ \times \pa_{a_1} \left[ Sym_{b_1, \dots, b_{r+1}} \frac{x_{b_1}^l N^{l-r} F_{\rho} (x_{b_1})^{l-r} }{(x_{b_1}- x_{b_2}) \dots (x_{b_1}- x_{b_{r+1}})} \right] + \tilde T_{(k+l-r-q-1)} (\vec{x}).
\end{multline}
Note that the summation over indices produces order $N^{r+q+1}$ terms in this case in \eqref{eq:covar-exact-form1}, so the function $\tilde T_{(k+l-r-q-1)} (\vec{x})$ does not contribute to $N$-degree $k+l$. The first term in the right-hand side of \eqref{eq:covar-333} gives rise to exactly the same computation as in Proposition \ref{prop:cov-extr}. As we proved in Proposition \ref{prop:cov-extr}, the contribution of this term to $N$-degree $k+l$ can be written as
\begin{multline}
\label{eq:covar-second-case-rez}
\frac{N^{k+l}}{(2 \pi \ii)^2} \oint_{|z|=\eps} \oint_{|w|=2 \eps} \left( \frac{1}{z} +1 + (1+z) F'_{\rho} (1+z) \right)^k \left( \frac{1}{w} +1 + (1+w) F'_{\rho} (1+w) \right)^l \\ \times \frac{1}{(z-w)^2} dz dw.
\end{multline}
It is interesting to note that while this term has very similar form to \eqref{eq:covar-first-case-rez}, we obtain it as a result of rather lengthy computations of the whole Section \ref{sec:cov-extr}, though the computation behind \eqref{eq:covar-first-case-rez} in case 1) is much simpler.

{\bf 3}) Assume that $|\{a_1, \dots, a_{q+1} \} \cap \{b_1, \dots, b_{r+1} \}| \ge 2$. Then the same argument as in case 2) shows that for the fixed indices the function in the left-hand side of \eqref{eq:covar-333} has a $N$-degree not greater than $(k+l-r-q-1)$, while the summation over all such indices contributes only $N^{r+q}$. Therefore, all such terms do not contribute to $N^{k+l}$.

\medskip

It remains to conclude that the contribution to the $N$-degree $k+l$ is given by the sum of expression from \eqref{eq:covar-first-case-rez} and \eqref{eq:covar-second-case-rez}. Therefore, recalling the definition of $Q_{\rho}$ given in Definition \ref{def:main}, we are done.

\end{proof}

\subsection{Covariance in Theorems \ref{theorem:main-projections}, \ref{theorem:main-multiplication}, \ref{th:general-for-domino} }
\label{sec:6-3}

The arguments of Section \ref{sec:6-2} need only minor modifications in order to compute the covariance in Theorems \ref{theorem:main-projections}, \ref{theorem:main-multiplication}, \ref{th:general-for-domino}. In each case, we start with a general formula for moments \eqref{eq:prop-gen-form-moments} and analyze it in the same way as in the case of one level.


\textit{Covariance in Theorem \ref{theorem:main-projections}}.

In this case the joint moments on different levels are given by the following differential operators
$$
\mathbf E \left( p_{k_1}^{[a_{t_1} N]} p_{k_2}^{[a_{t_2} N]} \right) = \left. \frac{1}{V_N (\vec{x})} \sum_{i=1}^{[a_{t_1} N]} (x_i \pa_i)^{k_1} \sum_{j=1}^{[a_{t_2} N]} (x_j \pa_j)^{k_2} V_N (\vec{x}) S_N (\vec{x}) \right|_{\vec{x}=1}.
$$

The only difference with computations in Sections \ref{sec:cov-extr} and \ref{sec:6-2} is that in \eqref{eq:covar-exact-form1} the set $\{a_1, \dots, a_{q+1}\}$ is the subset of $\{1,2, \dots, [a_{t_1} N] \}$ and the set $\{b_1, \dots, b_{r+1}\}$ is the subset of $\{1,2, \dots, [a_{t_2} N] \}$.
This leads to the appearance of the factor $a_{t_1}^q a_{t_2}^r$ inside of summations in \eqref{eq:differ-ind-res1} and \eqref{eq:last-formula-extr-char}. The arising modification of computations is given by Lemmas \ref{lem:a-l-r-integrals} and \ref{lem:extremal-equality} with $F_1(x) = \frac{F (x)}{a_{t_1}}$ and $F_2(x) = \frac{F (x)}{a_{t_2}}$.
This gives rise to two functions
$$
\FFF_{1} (z) := \frac{1}{z}+ 1 + \frac{(1+z) F (1+z)}{a_{t_1}}, \qquad  \FFF_{2} (z) := \frac{1}{z}+ 1 + \frac{(1+z) F (1+z)}{a_{t_2} },
$$
instead of one function $\FFF(z)$ (as before, we identify the function $F(x)$ from Section \ref{sec:cov-extr} and $F_{\rho} (x)$). In the end, we obtain
\begin{multline*}
\lim_{N \to \infty} \frac{\cov (p_k^{[a_{t_1} N]}, p_l^{[a_{t_2} N]})}{N^{k+l}} = \frac{a_{t_1}^k a_{t_2}^l}{(2 \pi \ii)^2} \oint_{|z|=\ep} \oint_{|w| = 2 \ep} \FFF_{1} (z)^k \FFF_{2} (w)^l \\ \times \left( G_{\rho} (z,w) + \frac{1}{(z-w)^2} \right) dz dw.
\end{multline*}

\textit{Covariance in Theorem \ref{theorem:main-multiplication}}.

In this case the moments of power sums are given by the following differential operators
$$
\mathbf E \left( p_{k_1;s_1}^{(N)} p_{k_2;s_2}^{(N)} \right) = \left. \frac{1}{V_N (\vec{x}) } \sum_{i=1}^N ( x_i \pa_i)^{k_1} \prod_{r=s_1}^{s_2-1} g_r (\vec{x}) \sum_{j=1}^N (x_j \pa_j)^{k_2} V_N (\vec{x}) H_{s_2} (\vec{x}) \right|_{\vec{x}=1}.
$$
Therefore, the right-hand side of \eqref{eq:covar-exact-form1} has a form
\begin{multline}
\label{eq:covar-exact-form-mult}
k \sum_{q=0}^{k-1} \sum_{\{a_1, \dots, a_{q+1} \} \subset [N] } \binom{k-1}{q} Sym_{a_1, \dots, a_{q+1}} \frac{ x_{a_1}^{k} \left( \pa_{a_1} \left[ \log H_{s_1} \right] \right)^{k-1-q}} {(x_{a_1} -x_{a_2} ) \dots (x_{a_1}-x_{a_{q+1}})} \\ \left. \times \pa_{a_1} \left[ \sum_{r=0}^{l} \sum_{\{b_1, \dots, b_{r+1} \} \subset [N] } \binom{l}{r} Sym_{b_1, \dots, b_{r+1}} \frac{x_{b_1}^l \left( \pa_{b_1} \left[ \log H_{s_2} \right] \right)^{l-r}}{(x_{b_1}- x_{b_2}) \dots (x_{b_1}- x_{b_{r+1}})} \right] \right|_{\vec{x}=1},
\end{multline}
(recall that the functions $H_s$ are defined in \eqref{eq:h-func-def-mult}).

The analysis of this expression goes in exactly the same way as before. The only difference is that
instead of \eqref{eq:equiv-approp} we need to plug
\begin{eqnarray*}
\left( \pa_{b_1} \left[ \log H_{s_1} \right] \right)^{l-r} &=& N^{l-r} F_{\rho; (s_1)} (x_{b_1})^{l-r} + \tilde T_{ (l-r)}, \\
\left( \pa_{a_1} \left[ \log H_{s_2} \right] \right)^{k-q-1} &=& N^{k-q-1} F_{\rho; (s_2)} (x_{a_1})^{k-q-1} + \tilde T_{ (k-q-1)}, \\
\pa_{a_1} \pa_{b_1} \left[ \log H_{s_2} \right] &=& G_{\rho; (s_2)} (x_{a_1}, x_{b_1}) + \tilde
T_{(0)}
\end{eqnarray*}
into \eqref{eq:covar-exact-form-mult}.
The appearance of two different functions $F_{\rho; (s_1)} (x)$ and $F_{\rho; (s_2)} (x)$ instead of one function $F_{\rho} (x)$ leads to a modification of computations of Section \ref{sec:6-2} which is covered by Lemmas \ref{lem:a-l-r-integrals} and \ref{lem:extremal-equality} with $F_1(x) = F_{\rho; (s_1)} (x)$ and  $F_2 (x) = F_{\rho; (s_2)} (x)$. This gives the covariance in Theorem \ref{theorem:main-multiplication}.

\textit{Covariance in Theorem \ref{th:general-for-domino} }

The moments of power sums are given by
\begin{multline*}
\mathbf E \left( p_{k_1;t_1}^{[a_{t_1} N]} p_{k_2;t_2}^{[a_{t_2} N]} \right) =  \frac{1}{V_{[a_{t_2} N]} ( x_1, x_2, \dots, x_{[a_{t_2} N]} )} \sum_{i=1}^{[a_{t_1 N}]} ( x_{i} \pa_{i})^{k_1} \prod_{r=t_1}^{t_2-1} g_r^{(N)} (x_1, x_2, \dots, x_{[a_r N]}) \\ \left. \times \sum_{j=1}^{[a_{t_2} N]} (x_j \pa_j)^{k_2} V_{[a_{t_2} N]} ( x_1, x_2, \dots, x_{[a_{t_2} N]} ) H_{t_2}^{(N)} (x_1, x_2, \dots, x_{[a_{t_2} N]} ) \right|_{\vec{x}=1}.
\end{multline*}
The analysis goes in the same way as in the previous two cases with both changes made simultaneously.



\section{Asymptotic normality}
\label{sec:7}

\subsection{Gaussianity: Theorem \ref{theorem:main-one-level}}
\label{sec:7-1}

In the notations of Theorem \ref{theorem:main-one-level} we prove the asymptotic normality of the vector $\left\{ N^{-k} \left( p_{k}^{(N)} - \E p_{k}^{(N)} \right) \right\}_{k \in \N}$.

Note that for any $k$ we have
$\E p_{k}^{(N)} = \FF_{(k)} (1^N).$  For any $k_1, k_2$ in Section \ref{sec:6-2} we showed that the quantity
$$
C_{k_1,k_2} := \lim_{N \to \infty} \frac{\cov( p_{k_1}^{(N)}, p_{k_2}^{(N)})}{N^{k_1+k_2}}
$$
exists (and also computed it).

\begin{proposition}
\label{prop:gauss-main}
For any positive integers $k_1, \dots, k_s$ we have
$$
\lim_{N \to \infty} \frac{\mathbf E \left( p_{k_1}^{(N)} - \mathbf E p_{k_1}^{(N)} \right) \dots \left( p_{k_s}^{(N)} - \mathbf E p_{k_s}^{(N)} \right)}{N^{k_1 + \dots + k_s}} = 0,
$$
if $s$ is odd, and
$$
\lim_{N \to \infty} \frac{ \mathbf E \left( p_{k_1}^{(N)} - \mathbf E p_{k_1}^{(N)} \right) \dots \left( p_{k_s}^{(N)} - \mathbf E p_{k_s}^{(N)} \right) }{N^{k_1 + \dots + k_s}} = \sum_{P \in \mathcal P^{s}_{\varnothing} } \prod_{(a,b) \in P} C_{k_a, k_b},
$$
where $\mathcal P^{s}_{\varnothing}$ is the set of all pairings of $\{1,2,\dots,s \}$.
\end{proposition}
\begin{proof}

One sees that
\begin{multline*}
\E \left( p_{k_1}^{(N)} - \mathbf E p_{k_1}^{(N)} \right) \dots \left( p_{k_s}^{(N)} - \mathbf E p_{k_s}^{(N)} \right) \\ =
\frac{1}{V_N S_{\rho_N}} \left( \sum_{i_1=1}^N (x_{i_1} \pa_{i_1})^{k_1} - \FF_{(k_1)} (1^N) \right) \left( \sum_{i_2=1}^N (x_{i_2} \pa_{i_2})^{k_2} - \FF_{(k_2)} (1^N) \right) \\ \times \left. \dots \left( \sum_{i_s=1}^N (x_{i_s} \pa_{i_s})^{k_s} - \FF_{(k_s)} (1^N) \right) V_N S_{\rho_N} \right|_{\vec{x}=1}.
\end{multline*}
Therefore, the statement  of the proposition is a direct corollary of Proposition
\ref{prop:gener-gaussianity-final}.
\end{proof}
Proposition \ref{prop:gauss-main} asserts that the joint moments satisfy the Wick formula (see, e.g., \cite[Section 1.2]{Z} for the basic information about Wick formula) which implies the asymptotic normality. Therefore, Propositions \ref{prop:main-covar-comput} and \ref{prop:gauss-main} readily imply Theorem \ref{theorem:main-one-level}.

\subsection{Gaussianity: Theorems \ref{theorem:main-projections}, \ref{theorem:main-multiplication}, \ref{th:general-for-domino} }
\label{sec:7-2}

We discuss the case of Theorem \ref{th:general-for-domino} only, since this theorem implies Theorems \ref{theorem:main-projections} and \ref{theorem:main-multiplication}.

We denote by $\vec{x}_{a}$ the set of variables $(x_1,x_2, \dots, x_a)$.

We use a general formula \eqref{eq:prop-gen-form-moments} for moments. In the notations of Theorem \ref{th:general-for-domino}, for arbitrary $s$ and $k_1,\dots, k_s$, $t_1 \le t_2 \le \dots \le t_s$, we have
\begin{multline*}
\mathbf E \left( p_{k_1;t_1}^{[a_{t_1} N]} p_{k_2;t_2}^{[a_{t_2} N]} \dots p_{k_s;t_s}^{[a_{t_s} N]} \right) =  \frac{1}{V_{[a_{t_s} N]} (\vec{x}_{[a_{t_s} N]}) } \sum_{i_1=1}^{[a_{t_1 N}]} ( x_{i_1} \pa_{i_1})^{k_1} \prod_{r=[a_{t_1} N]}^{[a_{t_2} N]-1} g_r^{(N)} (\vec{x}_{[ a_r N]} ) \\ \times \sum_{i_2=1}^{[a_{t_2} N]} (x_{i_2} \pa_{i_2})^{k_2} \prod_{r=[a_{t_2} N]}^{[a_{t_3} N]-1} g_r^{(N)} (\vec{x}_{[ a_r N]}) \left. \dots  \sum_{i_s=1}^{[a_{t_s N}]} ( x_{i_s} \pa_{i_s})^{k_s} V_{[a_{t_s} N]} (\vec{x}_{[a_{t_s} N]} ) H_{t_s}^{(N)} (\vec{x}_{[a_{t_s} N]}) \right|_{\vec{x}=1}.
\end{multline*}

The analysis of this formula is exactly the same as in Sections \ref{sec:5-4}, \ref{sec:5-5}, and \ref{sec:7-1}. Let us indicate necessary modifications of notations. For $1 \le q \le s$ instead of \eqref{eq:def-Fl} we consider the function
\begin{equation}
\label{eq:def-Fl-gen-domino}
\FF_{(l); t_q} (\vec{x}) := \frac{1}{ H^{(N)}_{t_q} (\vec{x}_{[a_{t_q} N]}) V_{[a_{t_q} N]} (\vec{x}_{[a_{t_q} N]}) } \sum_{i=1}^{[a_{t_q} N]} \left( x_i \pa_i \right)^{l} V_{[a_{t_q} N]} (\vec{x}_{[a_{t_q} N]}) H^{(N)}_{t_q} (\vec{x}_{[a_{t_q} N]}).
\end{equation}
For $1 \le q \le w \le s$ instead of \eqref{eq:def-GGfunc} we use the function
\begin{multline}
\label{eq:def-GGfunc-gen-domino}
\GG_{(l_1, l_2); t_q,t_w} (\vec{x}_{[a_{t_q} N]}) := l_1 \sum_{r=0}^{l_1-1} \binom{l_1-1}{r} \\ \times \sum_{\{ \mathbf{a}_1, \dots, \mathbf{a}_{r+1} \} \subset [ a_{t_q} N ] } Sym_{\mathbf{a}_1, \dots, \mathbf{a}_{r+1}} \frac{ x_{\mathbf{a}_1}^{l_1} \pa_{\mathbf{a}_1} \left[ \FF_{(l_2);t_w} \right] \left( \pa_{\mathbf{a}_1} \left[ \log H_{t_1} \right] \right)^{l_1-1-r}} {(x_{\mathbf{a}_1} -x_{\mathbf{a}_2} ) \dots (x_{\mathbf{a}_1}-x_{\mathbf{a}_{r+1}})}.
\end{multline}
Instead of \eqref{eq:def-e-l-const} we use $E_{l; t_q} := \FF_{(l); t_q} (1^{[a_{t_q} N]})$. With these changes, all the analysis of Sections \ref{sec:5} and \ref{sec:7-1} goes in exactly the same way, which gives us the asymptotic normality of functions $\{ p_{k;t_q}^{[a_{t_q} N]} \}_{q \ge 1, k \ge 1}$.

\section{Asymptotics of Schur functions}
\label{sec:asymp-Schur-func}

In this section we recall and extend certain asymptotics of normalised Schur functions, which were developed in \cite{GM}, \cite{GP}, \cite{BG}.

Recall that we encode a signature $\la = \la_1 \ge \dots \ge \la_N$ by a discrete probability measure on $\R$ via
$$
m [\la] := \frac{1}{N} \sum_{i=1}^N \delta \left( \frac{\la_i+N-i}{N} \right).
$$

We use the notation from Section \ref{sec:statem-requir-formulae}. The following theorem is a special case of Theorem 4.2 from \cite{BG}.

\begin{theorem}[\cite{GM},\cite{GP}, \cite{BG} ]
\label{Theorem_character_asymptotics}
 Suppose that $\lambda(N)\in \GT_N$, $N=1,2,\dots$ is a regular sequence of signatures (see Definition \ref{definition_regularity} ), such that
 $$
  \lim_{N\to\infty} m[\lambda(N)]=\mes.
 $$
 Then for any
 $k=1,2,\dots$ we have
 \begin{equation}
 \label{eq_limit_of_logarithm}
  \lim_{N\to\infty} \frac{1}{N}
  \log \left(\frac{s_{\lambda(N)}(x_1,\dots,x_k,1^{N-k})}{s_{\lambda(N)}(1^N)}\right)=
  H_\mes(x_1)+\dots+ H_\mes(x_k),
 \end{equation}
 where the convergence is uniform over an open complex neighborhood of $(x_1, \dots, x_k)=(1^k)$.
\end{theorem}

\begin{theorem}
\label{Theorem_character_asymtotics_2}

Suppose that $\lambda(N)\in \GT_N$, $N=1,2,\dots$ is a regular sequence of signatures such that
 $$
  \lim_{N\to\infty} m[\lambda(N)]=\mes.
 $$
 Then we have
\begin{multline}
 \label{eq_limit_of_logarithm_2}
  \lim_{N \to \infty}
  \pa_{1} \pa_{2} \log \left(\frac{s_{\lambda(N)}(x_1, x_2, \dots, x_k, 1^{N-k})}{s_{\lambda(N)}(1^N)} \right) \\ =
   \pa_{1} \pa_{2} \log \left( 1 - (x_1-1) (x_2-1) \frac{x_1 H'_\mes(x_1) - x_2 H'_{\mes} (x_2)}{x_1-x_2} \right),
 \end{multline}
and
$$
\lim_{N \to \infty}
  \pa_{1} \pa_{2} \pa_3 \log \left(\frac{s_{\lambda(N)}(x_1, x_2, \dots, x_k, 1^{N-k})}{s_{\lambda(N)}(1^N)} \right) = 0,
$$
where the convergence is uniform over an open complex neighborhood of $(x_1,\dots, x_k)=(1^k)$.
\end{theorem}
\begin{remark}
Note that the limit in \eqref{eq_limit_of_logarithm_2} does not depend on $(x_3, \dots, x_k)$.
\end{remark}

\begin{proof}[Proof of Theorem \ref{Theorem_character_asymtotics_2}]

Let
$$
S_{\lambda} (x_j;N,1) := \frac{s_{\la} (1^{j-1}, x_j, 1^{N-j})}{s_{\la} (1^N)}.
$$
Theorem 3.7 of \cite{GP} asserts that
\begin{multline}
\label{eq:character-formula-GP}
\frac{s_{\lambda(N)}(x_1, x_2, \dots, x_k, 1^{N-k})}{s_{\lambda(N)}(1^N)} = \prod_{i=1}^k \frac{(N-i)!}{(N-1)! (x_i-1)^{N-k}} \\
\times \frac{\det \left[ (x_a \pa_a)^{b-1} \right]_{a,b=1}^k }{\prod_{1 \le a<b \le N} (x_a-x_b)} \prod_{j=1}^k S_{\lambda} (x_j;N,1) (x_j-1)^{N-1}.
\end{multline}

Let us consider the application of the differential operator $\det \left[ (x_a \pa_a)^{b-1} \right]_{a,b=1}^k$. Each differentiation from it can be applied to $S_{\lambda} (x_j;N,1) (x_j-1)^{N-1}$ for some $j$ or to the factors appeared from the other differentiations. Note that
the highest degree in $N$ is obtained when each differentiation is applied to $S_{\lambda} (x_j;N,1) (x_j-1)^{N-1}$.
Using Theorem \ref{Theorem_character_asymptotics} we obtain
\begin{multline*}
\left(x_a \pa_a \right)^{b-1} \prod_{j=1}^k S_{\lambda}(x_j;N,1)(x_j-1)^{N-1} \\ = N^{b-1}\prod_{j=1}^k
S_{\lambda}(x_j;N,1)(x_j-1)^{N-1} \cdot\left(\frac{x}{1-x}+x H'_{\mes} (x) \right)^{b-1} + o( N^{b-1}),
\end{multline*}
(here and below the convergence in $o(\cdot)$ is uniform over an open complex neighborhood of $(x_1, \dots, x_k) = (1^k)$). Hence,
\begin{multline*}
\det \left[ (x_a \pa_a)^{b-1} \right]_{a,b=1}^k \prod_{j=1}^k S_{\lambda} (x_j;N,1) (x_j-1)^{N-1} = \prod_{j=1}^k S_{\lambda} (x_j;N,1) (x_j-1)^{N-1} \\ \times \left( N^{b(b-1)/2} \det \left[ \left( \frac{x_a}{1-x_a}+x_a H'_{\mes} (x_a) \right)^{b-1} \right]_{a,b=1}^k + o \left( N^{b(b-1)/2} \right) \right),
\end{multline*}
where we use that the uniform convergence of analytic functions implies the convergence of its derivatives.
Substituting this formula into \eqref{eq:character-formula-GP}, we get
\begin{multline*}
\pa_1 \pa_2 \log \frac{s_{\lambda(N)}(x_1, x_2, \dots, x_k, 1^{N-k})}{s_{\lambda(N)}(1^N)} \\ = \pa_1 \pa_2 \log \left( \prod_{j=1}^k S_{\lambda}(x_j;N,1) (x_j-1)^{k-1} \left( \frac{ \det \left[ \left( \frac{x_a}{1-x_a}+x_a H'_{\mes} (x_a) \right)^{b-1} \right]_{a,b=1}^k }{\prod_{a<b}^k (x_a-x_b)} + o(1) \right) \right)
\\ = \pa_1 \pa_2 \log \left( \prod_{a<b}^k \frac{ \left( \frac{x_a}{1-x_a}+x_a H'_{\mes} (x_a) \right) - \left( \frac{x_b}{1-x_b}+x_b H'_{\mes} (x_b) \right)}{x_a - x_b} + o(1) \right)
\\ = \pa_1 \pa_2 \log \left( 1 + \frac{x_1 H'_{\mes} (x_1) - x_2 H'_{\mes} (x_2)}{x_1 - x_2} (x_1-1) (x_2 -1) \right) + o(1).
\end{multline*}
Also we see that
\begin{multline*}
\pa_1 \pa_2 \pa_3 \log \frac{s_{\lambda(N)}(x_1, x_2, \dots, x_k, 1^{N-k})}{s_{\lambda(N)}(1^N)}
\\ = \pa_1 \pa_2 \pa_3 \log \left( \prod_{a<b}^k \frac{ \left( \frac{x_a}{1-x_a}+x_a H'_{\mes} (x_a) \right) - \left( \frac{x_b}{1-x_b}+x_b H'_{\mes} (x_b) \right)}{x_a - x_b} + o(1) \right) = o(1). \qedhere
\end{multline*}
\end{proof}

Recall that for a representation $T$ of $U(N)$ we define a probability measure $\rho_T$ on $\GT_N$ with the use of \eqref{eq:repr-meas-sign-def}. The pushforward of $\rho_T$
with respect to the map $\lambda \to m[\lambda]$ is a \emph{random} probability measure on $\mathbb
R$ that we denote $m[\rho_T]$.

\begin{proposition}
\label{prop:Schur-are-approp}

Assume that $\lambda^{(1)} (N), \lambda^{(2)} (N) \in \GT_N$, $N=1,2,\dots$, are two regular sequences of signatures such that
$$
\lim_{N\to\infty} m[\lambda^{(1)} (N)]=\mes_1, \qquad \lim_{N\to\infty} m[\lambda^{(2)} (N)]=\mes_2,
$$
for probability measures $\mes_1$ and $\mes_2$ with compact supports. Then $m [ \rho_{\pi^{\la^{(1)}} \otimes \pi^{\la^{(2)}} } ]$ is an appropriate probability measure on $\GT_{N}$ with functions
\begin{multline*}
F_{\rho} (x) = H'_{\mes_1} (x) + H'_{\mes_2} (x), \\
G_{\rho} (x,y) = \pa_{x} \pa_{y} \log \left( 1 - (x-1) (y-1) \frac{x H'_{\mes_1} (x) - y H'_{\mes_1} (y)}{x-y} \right) \\ + \pa_{x} \pa_{y} \log \left( 1 - (x-1) (y-1) \frac{x H'_{\mes_2} (x) - y H'_{\mes_2} (y)}{x-y} \right).
\end{multline*}

\end{proposition}
\begin{proof}
The Schur generating function of the measure $m [ \rho_{ \pi^{\la^{(1)}} \otimes \pi^{\la^{(2)}} }
]$ is $s_{\la^{(1)}} (\vec{x}) s_{\la^{(2)}} (\vec{x})$. Therefore, the statement of the lemma
follows from Theorems \ref{Theorem_character_asymptotics} and \ref{Theorem_character_asymtotics_2}.
\end{proof}

\section{Proofs of applications}

\subsection{Lozenge tilings and Gaussian Free Field}

\label{Section_tilings_GFF}

In this section we prove Theorem \ref{th:ctl-proj-GFF-1}.

Recall that we study the uniform measure on the set of paths $\mathcal P_N (\la^{(N)})$. The projection of this measure to one level $\GT_M$, $M \le N$, produces a probability measure on $\GT_M$. The branching rule for Schur functions \eqref{eq:Schur-branching-pr-coef} shows that its Schur generating function equals $s_{\la^{(N)}} (x_1, \dots, x_M, 1^{N-M}) / s_{\la^{(N)}} (1^{N})$.

For $a<1$, let $p_k^{[aN]}$ be moments of these measures:
$$
p_k^{[aN]} = \sum_{i=1}^{[aN]} \left( \la_i + [aN]-i \right)^k, \qquad \la \in \GT_{[aN]}.
$$
Note that they are random variables.

\begin{proposition}
\label{prop:restr-moments-gauss}
Let $0 < a_1 \le \dots \le a_s \le 1$ be a collection of reals. Under the assumptions and in the notations of Theorem \ref{th:ctl-proj-GFF-1} the collection of random variables
\begin{equation}
\label{eq:vector-tilings}
\left\{ N^{-k_i} \left( p_{k_i}^{[a_i N]} - \E p_{k_i}^{[a_i N]} \right) \right\}_{i=1, \dots, s}
\end{equation}
converges to the Gaussian vector $(\xi_1, \dots, \xi_s)$ with the covariance
\begin{multline}
\label{eq:prop-tilings-covar}
\cov(\xi_r, \xi_t) = \frac{a_t^{k_t} a_r^{k_r} }{(2 \pi \ii)^2} \oint_{|z|=\ep} \oint_{|w|=2 \ep} \left( \frac{1}{z} +1 + \frac{(1+z) \mathbf H'_{\mes} (1+z)}{a_t} \right)^{k_t} \\ \times \left( \frac{1}{w} +1 + \frac{(1+w) \mathbf H'_{\mes} (1+w)}{a_r} \right)^{k_r}  \pa_z \pa_w \left[ \log \left( \left( \frac{1}{w} +1 + (1+w) \mathbf H'_{\mes} (w) \right) \right. \right. \\ \left. \left. - \left( \frac{1}{z} +1 + (1+z) \mathbf H'_{\mes} (z) \right) \right) \right] dz dw,
\end{multline}
for $1 \le t \le r \le s$, $\ep \ll 1$, with $H'_{\mes} (w)$ given by \eqref{eq:H-S-connection}.
\end{proposition}
\begin{proof}
By Theorems \ref{Theorem_character_asymptotics}, \ref{Theorem_character_asymtotics_2} this model satisfies the conditions of Theorem \ref{theorem:main-projections} with functions $F_{\rho} (x)= \mathbf H'_{\mes} (x)$ and
$$
Q_{\rho} (x,y) = \pa_x \pa_y \left( \log \left( 1 - x y \frac{ (1+x) \mathbf H'_{\mes} (1+x) - (1+y) \mathbf H'_{\mes} (1+y)}{x-y} \right) \right).
$$
Applying Theorem \ref{theorem:main-projections}, we obtain that the Central Limit Theorem holds for the vector \eqref{eq:vector-tilings} with the covariance
\begin{multline}
\label{eq:tilings-covar-interm}
\cov(\xi_r, \xi_t) \\ = \frac{a_t^{k_t} a_r^{k_r}}{(2 \pi \ii)^2} \oint_{|z|=\ep} \oint_{|w|=2 \ep} \left( \frac{1}{z} +1 + \frac{(1+z) \mathbf H'_{\mes} (1+z)}{a_t} \right)^{k_t} \left( \frac{1}{w} +1 + \frac{(1+w) \mathbf H'_{\mes} (1+w)}{a_r} \right)^{k_r} \\ \times \left( \pa_z \pa_w \left[ \log \left( 1 - z w \frac{ (1+z) \mathbf H'_{\mes} (1+z) - (1+w) \mathbf H'_{\mes} (1+w)}{z-w} \right) \right] + \frac{1}{(z-w)^2} \right) dz dw,
\end{multline}
for $1 \le r \le t \le s$, and $\ep \ll 1$. With the use of the equalities $ \pa_z \pa_w \log(z-w) = \frac{1}{(z-w)^2}$ and $\pa_z \pa_w (zw) =0$, we transform \eqref{eq:tilings-covar-interm} into \eqref{eq:prop-tilings-covar}.
\end{proof}

Proposition \ref{prop:restr-moments-gauss} shows that the fluctuations in our model are Gaussian. We next recover the structure of the Gaussian Free Field, for that we transform the expression for the covariance.

\begin{lemma}
\label{lem:two-int-proj}
The expression \eqref{eq:prop-tilings-covar} is equal to
\begin{multline}
\label{eq:lem-two-int-proj}
\frac{1}{(2 \pi \ii)^2 } \oint_{|z|= 2 C} \oint_{|w|= C} \left( z + \frac{1- a_r}{\exp \left( - C_{\mes} \left( z \right) \right) - 1} \right)^{k_r} \left( w + \frac{1- a_t}{\exp \left( - C_{\mes} \left( w \right) \right) - 1} \right)^{k_t} \\ \times \frac{1}{(z-w)^2} dz dw,
\end{multline}
where $C \gg 1$, that is, the contours of integration contain all poles of the integrand (recall that the function $C_{\mes}(z)$ is defined in \eqref{eq:def-sm-function} ).
\end{lemma}
\begin{proof}
Let us make a change of variables in \eqref{eq:prop-tilings-covar}
$$
\tilde z = \frac{1}{C_{\mes}^{(-1)} \left( \log (1+z) \right)}, \qquad \tilde w = \frac{1}{C_{\mes}^{(-1)} \left( \log (1+w) \right)};
$$
this change of variables is well-defined since we are dealing with analytic functions in a neighborhood of the origin.

Using the relation between $C_{\mes} (z)$ and $\mathbf H'_{\mes} (z)$ (see equation \eqref{eq:H-S-connection}), we have
\begin{align*}
\frac{1}{z} + 1 + \frac{(1+z) \mathbf H'_{\mes} (1+z)}{a_t} = \frac{1}{a_t} \left( \frac{1}{\tilde z} + \frac{1- a_t}{\exp( - C_{\mes} (\frac{1}{\tilde z})) - 1} \right), \\
\frac{1}{w} + 1 + \frac{(1+w)\mathbf H'_{\mes} (1+w)}{a_r} = \frac{1}{a_r} \left( \frac{1}{\tilde w} + \frac{1- a_r}{\exp( - C_{\mes} (\frac{1}{\tilde w})) - 1} \right), \\
\log \left( \left( \frac{1}{w} +1 + (1+w) \mathbf H'_{\mes} (1+w) \right) - \left( \frac{1}{z} +1 + (1+z) \mathbf H'_{\mes} (1+z) \right) \right) = \log \left( \frac{1}{\tilde w} - \frac{1}{\tilde z} \right).
\end{align*}
Substituting these equalities, we obtain that \eqref{eq:prop-tilings-covar} equals
\begin{multline*}
\frac{1}{(2 \pi \ii)^2 } \oint_{|\tilde z|=\ep} \oint_{|\tilde w|=2 \ep} \left( \frac{1}{\tilde z} + \frac{1- a_t}{\exp( - C_{\mes} (\frac{1}{\tilde z})) - 1} \right)^{k_t} \left( \frac{1}{\tilde w} + \frac{1- a_r}{\exp( - C_{\mes} (\frac{1}{\tilde  w})) - 1} \right)^{k_r} \\ \times \frac{1}{(\tilde z- \tilde w)^2} d \tilde z d \tilde w.
\end{multline*}
Making a further change of variables $\tilde z \to \frac{1}{\tilde z}$, $\tilde w \to \frac{1}{\tilde w}$, we arrive at \eqref{eq:lem-two-int-proj}.

\end{proof}

\begin{proof}[Proof of Theorem \ref{th:ctl-proj-GFF-1} ]

We recall that the functions $y_{\mes} (z)$ and $\eta_{\mes} (z)$ were defined in Section \ref{sec:restrictions}.

For $0<a < 1$ let $\mathbf{C}_{a;\mes}$ be the union of the set $\{ z \in \HH: \eta_{\mes}(z)= a \}$ and its conjugate.
A direct check shows that if $z \to \infty$ then $\eta_{\mes} (z) \to 0$. Therefore, for $a_r < a_t$ the contour $\mathbf{C}_{a_r;\mes}$ contains the contour $\mathbf{C}_{a_t;\mes}$. Thus, in \eqref{eq:lem-two-int-proj} we can deform the contour $|w|=C$ to $\mathbf{C}_{a_t;\mes}$ and the contour $|z|=2C$ to $\mathbf{C}_{a_r;\mes}$ without meeting poles of the integrand. We obtain
\begin{multline}
\label{eq:trans-gff-end-1}
\frac{1}{(2 \pi \ii)^2} \oint_{|z|= 2 C} \oint_{|w|= C} \left( z + \frac{1- a_r}{\exp \left(-C_{\mes} \left( z \right) \right) - 1} \right)^{k_r} \left( w + \frac{1- a_t}{\exp \left(-C_{\mes} \left( w \right) \right) - 1} \right)^{k_t} \\ \times \frac{1}{(z-w)^2} dz dw = \frac{1}{(2 \pi \ii)^2} \oint_{\mathbf{C}_{a_r;\mes} } \oint_{\mathbf{C}_{a_t;\mes}} y_{\mes} (z)^{k_r} y_{\mes} (w)^{k_t} \frac{1}{(z-w)^2} dz dw.
\end{multline}
Recall that the values of $y_{\mes} (z)$ are real. Using this fact and the equality
$$
2 \log \left| \frac{z-w}{z- \bar w} \right| = \log \frac{(z-w) (\bar z - \bar w)}{(z - \bar w) (\bar z - w)},
$$
we can rewrite this expression as
\begin{equation}
\label{eq:trans-gff-end-2}
\frac{1}{(2 \pi \ii)^2} \oint_{z \in \HH: \eta_{\mes}(z)= a_r } \oint_{w \in \HH: \eta_{\mes}(w)= a_t } y_{\mes} (z)^{k_r} y_{\mes} (w)^{k_t} \pa_z \pa_w \left[ 2 \log \left| \frac{z-w}{z- \bar w} \right| \right] dz dw,
\end{equation}
(in this equation we are integrating over curves in the upper half-plane only).

Let us now transform the quantities involved in the statement of the theorem. An integration by parts gives us
\begin{multline*}
M_{\eta,k}^{pr} = \int_{-\infty}^{+\infty} y^k \left( H^{\la^{(N)}} ( N y, N \eta) - \mathbf E H^{\la^{(N)}} ( N y, N \eta) \right) dy = \frac{N^{-(k+1)}}{k+1} \left( p_{k+1}^{[N \eta]} - \E p_{k+1}^{[N \eta]} \right).
\end{multline*}
Therefore, Proposition \ref{prop:restr-moments-gauss}, Lemma \ref{lem:two-int-proj} and equations \eqref{eq:trans-gff-end-1}, \eqref{eq:trans-gff-end-2} show that the collection $\{M_{\eta,k}^{pr}\}_{\eta \ge 0; k \in \mathbb Z_{\ge 0}}$ converges to the Gaussian limit and the limit covariance is given by
\begin{multline*}
\lim_{N \to \infty} \cov \frac{\left( M_{a_r,k_r}^{pr}, M_{a_t,k_t}^{pr} \right)}{N^{(k_r+1)+(k_t+1)}} = \frac{-1}{4 \pi^2 (k_r+1) (k_t+1)} \\ \times \oint_{z \in \HH: \eta_{\mes}(z)= a_r } \oint_{w \in \HH: \eta_{\mes}(w)= a_t } y_{\mes} (z)^{k_r+1} y_{\mes} (w)^{k_t+1} \pa_z \pa_w \left[ 2 \log \left| \frac{z-w}{z- \bar w} \right| \right] dz dw.
\end{multline*}
The definition of the Gaussian Free Field implies
\begin{multline*}
\cov \left( \mathcal M_{a_r,k_r}^{pr}, \mathcal M_{a_t, k_t}^{pr} \right) = \int_{z \in \HH: a_r = \eta_{\mes} (z)} \int_{w \in \HH; a_t = \eta_{\mes} (w)} y_{\mes} (z)^{k_r} y_{\mes} (w)^{k_t} \\
 \times \frac{d y_{\mes} (z)}{d z}   \frac{d y_{\mes} (w)}{d w}   \left[ \frac{-1}{2 \pi} \log \left| \frac{z-w}{z-\bar w} \right| \right] dz dw.
\end{multline*}
An integration by parts shows that the right-hand sides of two equations differ by a factor $\pi$.
This concludes the proof of the theorem.
\end{proof}

\subsection{Extreme characters of $U(\infty)$ and Gaussian Free Field}
\label{sec:ext-GFF-proof}

In this Section we prove Proposition \ref{prop:comp-struct-ext} and Theorem \ref{theorem:extr-char-gff-1}.

Recall that $C_{\mes} (z)$ is a Stieltjes transform of  measure $\mes$ on the real line (see
\eqref{eq:def-sm-function}).
\begin{lemma}
\label{lem:Voic-lim}
Assume that a sequence of extreme characters $\omega(N)$ satisfies the condition \eqref{eq:ext-basic-cond} with the limiting sextuple $\mathbf J = \{\mathcal A^+, \mathcal B^+, \mathcal A^-, \mathcal B^{-}, \Gamma^+, \Gamma^-)$. Then we have the convergence
$$
\lim_{N \to \infty} \frac{\pa_t \log \Phi^{\omega (N)} (1+t)}{N} = \F_{\mathbf J} (1+t), \qquad \mbox{uniformly in $|t|<\ep, \ \ \ep >0$},
$$
where $\F_{\mathbf J} (1+t)$ is given by the formula
\begin{multline}
\label{eq:limit-Voiculescu}
\F_{\mathbf J} (1+t) := \frac{1}{t^2} C_{\mathcal A^+} \left( \frac{1}{t} \right) - \frac{\mathcal A^+ (\R)}{t} + \frac{1}{t^2} C_{\mathcal B^+} \left( - \frac{1}{t} \right) + \frac{\mathcal B^+ (\R)}{t} - \frac{1}{t^2} C_{\mathcal A^-} \left( - \frac{1+t}{t} \right) \\ - \frac{\mathcal A^- (\R)}{t(1+t)} - \frac{1}{t^2} C_{\mathcal B^{-}} \left( \frac{1+t}{t} \right) + \frac{\mathcal B^- (\R)}{t(1+t)} + \Gamma^+ - \frac{\Gamma^-}{(1+t)^2}.
\end{multline}

\end{lemma}
\begin{proof}
The explicit formula for $\F_{\mathbf J} (1+t)$ comes as a direct computation from the Voiculescu formula \eqref{eq:Voic-formula}.
\end{proof}

We will need the following elementary technical statement about the measures on $\R$; we omit its proof.
\begin{lemma}
\label{lem:measures-easy-tech}
For each finite measure $\mu$ on $\R$ with compact support there exists a sequence of measures $\mu_K$ such that
$$
\lim_{K \to \infty} C_{\mu_K} (z) = C_{\mu} (z), \qquad \mbox{as $K \to \infty$,}
$$
and $\mu_K$ has a density with respect to Lebesgue measure which does not exceed $K^{1/10}$.
\end{lemma}


For a measure $\mes$ and $a \in \R$ let $\mathrm{sh}_{a} (\mes)$ be a shift of $\mes$ into $a \in \R$, that is
$$
\mathrm{sh}_{a} (\mes) (A+a) = \mes (A), \qquad \mbox{for any measurable $A \subset
\R$}.
$$
For a measure $\mes$ and $c \in \R$ we denote by $c \mes$ the measure
$$
(c \mes) (A) := c \cdot \mes(A), \qquad \mbox{for any measurabe $A \subset \R$}.
$$
For a set $A \subset \R$ let $\mathrm{sym} (A)$ be a set obtained from $A$ by reflecting with respect to $0$. For a measure $\mes$ we denote by $\mathrm{sym} (\mes)$ the measure
$$
(\mathrm{sym} (\mes)) (A) := \mes ( \mathrm{sym} (A) ), \qquad \mbox{for any
measurable $A \subset \R$}.
$$
We denote by $\mes_1 \cup \mes_2$ the union (equivalently, the sum )of measures
$\mes_1$ and $\mes_2$.

\begin{lemma}
\label{lem:predel-til-ext}
Assume that $\mathbf J = (\mathcal A^+, \mathcal B^+, \mathcal A^-, \mathcal B^-, \Gamma^+, \Gamma^- )$ is a sextuple that appears in the limit in the condition \eqref{eq:ext-basic-cond}. Then there exists a sequence of probability measures $\mu_{\mathbf J;K}$ with bounded by $1$ densities with respect to the Lebesgue measure on $\R$ such that theirs Stieltjes' transforms satisfy:
\begin{multline}
\label{eq:ext-stielt-k}
C_{\mu_{\mathbf J;K}} (z) = \log z - \log (z-1) +\frac{1}{K} \left( \left( C_{\mathcal A^+} (z-1) - \frac{\mathcal A^+ (\R)}{(z-1)} \right) \right. \\ \left. + \left( \frac{\mathcal B^+ (\R) }{z-1} + C_{\mathcal B^+} (1-z) \right) + \left( - C_{\mathcal A^-} (-z) - \frac{\mathcal A^- (\R) }{z} \right) + \left( - C_{\mathcal B^-} (z) + \frac{\mathcal B^- (\R) }{z} \right) \right. \\ \left. + \frac{\Gamma^+}{(z-1)^2} - \frac{\Gamma^-}{z^2} \right) + o \left( \frac{1}{K} \right),
\end{multline}
as $K \to \infty$.
\end{lemma}
\begin{proof}
First, note that $\log(z) - \log(z-1)$ is the Stieltjes transform of the uniform
measure on $[0;1]$. Next, we will consider several other signed measures with total
weight $0$ which give rise to other terms in expression \eqref{eq:ext-stielt-k}. It
follows that the union of the uniform measure on $[0;1]$ and all signed measures
will have weight $1$, and we will check that it is a probability measure.

Let $I_1$ be a (negative) measure with the density $ - \mathcal A^+ (\R) K^{-9/10}$
on the segment $[1- K^{-1/10}; 1]$. Let $\tilde A^+ := \frac{1}{K}
\mathrm{sh}_{+1}(\mathcal A^+)$ (note that the total weight of $\tilde A^+$ is
$1/K$). Then the measure $\tilde A^+ \cup I_1$ has total zero weight and the
Stieltjes transform of the form $\frac{1}{K} \left( C_{\mathcal A^+} (z-1) -
\frac{\mathcal A^+ (\R)}{(z-1)} \right) + o(1/K)$.

Let $I_2$ be a (positive) measure with the density $ \mathcal B^+ (\R) K^{-9/10}$ on
the segment $[1; 1+K^{-1/10}]$. Let $\tilde B^+ (K)$ be a sequence of measures given
by Lemma \ref{lem:measures-easy-tech} applied to the measure $\mathrm{sym}
(\mathrm{sh}_{-1} (\mathcal B^+))$. Then the measure $- \frac{1}{K} \tilde B^+ (K)
\cup I_2$ has total zero weight and the Stieltjes transform of the form $\frac{1}{K}
\left( C_{\mathcal B^+} (1-z) + \frac{\mathcal B^+ (\R)}{(z-1)} \right) + o(1/K)$.

The measures for $\mathcal A^-$ and $\mathcal B^-$ are constructed in an analogous way. In order to obtain the term $ \frac{\Gamma^+}{K (z-1)^2} + o(1/K)$ let us consider the measure which has density $\Gamma^+ K^{-8/10}$ on the interval $[1;1+K^{-1/10}]$ and density $(- \Gamma^+ K^{-8/10})$ on the interval $[1- K^{-1/10};1]$. In an analogous way we obtain the term $- 1/K \frac{\Gamma^-}{z^2} + o(1/K)$.

Finally, let us notice that all negative measures in the construction above are placed on the
segment from $0$ to $1$ (recall that beta parameters are bounded by $1$, see
\eqref{eq:Voic-formula} ) and has densities which decrease with $K$. In the same time, all positive
parts in the constructed signed measures lie outside of the segment $[0;1]$. Therefore, for large
$K$ the union of all these 6 measures with total weight zero and the uniform measure (with weight
$1$) on the segment $[0;1]$ forms a probability measure which has a required form of the Stieltjes
transform, and the density of this measure does not exceed $1$.
\end{proof}

\textit{Proof of Proposition \ref{prop:comp-struct-ext}}.

We prove this Proposition by a limit transition from Proposition \ref{prop:comp-struct-tilings}.

Let $K>0$ be a large real number. Let us consider the probability measure $\mu_{\mathbf J;K}$ on $\R$ which is given by Lemma \ref{lem:predel-til-ext}. Let us apply Proposition \ref{prop:comp-struct-tilings} to the measure $\mu_{\mathbf J;K}$. As a result, we obtain a diffeomorphism between $D_{\mu_{\mathbf J;K}}$ and $\HH$. For a fixed pair $(x, \alpha) \in D_{\mu_K}$ it is given by a unique root of the equation
\begin{equation}
\label{eq:Djk-bijection}
x = z+ \frac{1-\alpha}{e^{-C_{\mu_{\mathbf J;K}} (z)}-1},
\end{equation}
which lies in $\HH$ (the uniqueness of such a root is a part of Proposition \ref{prop:comp-struct-tilings}). We can rewrite it in the form
$$
C_{\mu_{\mathbf J;K}} (z) = \log(z-x) - \log(z - x+\alpha -1).
$$

Let us now set $x = \frac{X}{K}$, $\alpha = \frac{A}{K}$, for some fixed $X \in \R$ and $A>0$. 
For large $K$ we obtain
\begin{equation}
\label{eq:trrr2}
C_{\mu_{\mathbf J;K}} (z) = \log(z) - \log(z-1) - \frac{X}{Kz} - \frac{A-X}{K(z-1)} + o \left( \frac{1}{K} \right).
\end{equation}
Note that Lemma \ref{lem:Voic-lim} and Lemma \ref{lem:predel-til-ext} show that
\begin{equation}
\label{eq:trrr1}
C_{\mu_{\mathbf J;K}} (z) = \log(z) - \log(z-1) + \frac{1}{K (z-1)^2} \F_{\mathbf J} \left( 1+\frac{1}{z-1} \right) + o \left( \frac{1}{K} \right).
\end{equation}
Plugging \eqref{eq:trrr1} into \eqref{eq:trrr2}, cancelling $\log(z)-\log(z-1)$ and multiplying by $K$, we get
\begin{equation}
\label{eq:dikoe-ravenstvo}
X = A z + \frac{z}{z-1} \F_{\mathbf J} \left( 1+\frac{1}{z-1} \right) + o \left( 1 \right).
\end{equation}
Let us do a change of variables $t = \frac{1}{z-1}$.
Equation \eqref{eq:dikoe-ravenstvo} shows that
\begin{equation}
\label{eq:ext-main-coord}
\frac{X}{A} =  (1+t) \left( \frac{1}{t}+ \frac{ \F_{\mathbf J} (1+t)}{A} \right),
\end{equation}
which has a form given by Proposition \ref{prop:comp-struct-ext}.

Note that the function
$$
K \left( C_{\mu_{\mathbf J;K}} (z) - \log(z-X/K) + \log(z - X/K + A/K -1) \right)
$$
is analytic and converges uniformly on compact sets inside $\HH$ as $K \to \infty$. Therefore, the number of zeros of this function inside $\HH$ cannot increase in the limit as $K \to \infty$. Thus, for any pair $(X, A)$ the equation \eqref{eq:ext-main-coord} has no more than one solution in $\HH$, which shows the existence of $D_{\mathbf F}$ from the statement of Proposition \ref{prop:comp-struct-ext} and the existence of the map $D_{\mathbf F} \to \HH$.

On the other side, the bijection $\HH \to D_{\mu_{\mathbf J;K}}$ is given by the explicit formulas:
$$
x_{\mu_{\mathbf J;K}} (z) = z+  \frac{(z- \bar z) (\exp ( C_{\mu_{\mathbf J;K}} ( \bar z) )-1) \exp( C_{\mu_{\mathbf J;K}} (z))}{\exp(C_{\mu_{\mathbf J;K}} (z)) - \exp( C_{\mu_{\mathbf J;K}} (\bar z))},
$$
$$
\al_{\mu_{\mathbf J;K}} (z) = 1 + \frac{(z- \bar z) (\exp ( C_{\mu_{\mathbf J;K}} ( \bar z) )-1) ( \exp( C_{\mu_{\mathbf J;K}} (z))-1)}{\exp(C_{\mu_{\mathbf J;K}} (z)) - \exp( C_{\mu_{\mathbf J;K}} (\bar z))},
$$
( these functions are solutions to \eqref{eq:Djk-bijection} ). In the limit $K \to \infty$ and with the change of variables $t=\frac{1}{z-1}$ as above, these functions converge to the functions
$$
X(t) = \frac{(1+t) (1 + \bar{t} ) (t \F_{\mathbf J} (1+t) - \bar t \F_{\mathbf J} (1 + \bar{t} ))}{t - \bar{t}},
$$
$$
A(t) = \frac{(1+\bar t) ( t \F_{\mathbf J} (1+t) - \bar t \F_{\mathbf J} (1+\bar t))}{t - \bar t} t - t \F_{\mathbf J} (1+t).
$$

Note that for any $t \in \HH$ these limiting functions are solutions to \eqref{eq:ext-main-coord} with $X=X(t)$ and $A=A(t)$. Therefore, the map $D_{\mathbf F} \to \HH$ is a bijection. Moreover, this is a diffeomorphism since the functions $X(t)$ and $A(t)$ are differentiable, and the differentiability of the map $D_{\mathbf F} \to \HH$ is provided by the Implicit Function theorem.

\begin{proof}[Proof of Theorem \ref{theorem:extr-char-gff-1}]

Recall that we have a probability measure $\mu_{\chi^{w(N)}}$ on the set of paths $\mathcal P$ in the Gelfand-Tsetlin graph.
Let $p_{A;k}$ be the shifted moments of the random signature $\la^{([A N])}$ distributed according to this measure:
$$
p_{A;k} = \sum_{i=1}^{A N} (\la^{(A N)}_i + [A N] - i)^k.
$$

Our probabilistic model clearly satisfies assumptions of Theorem \ref{theorem:main-projections}, with $F_{\rho} (z) =  \F_{\mathbf J} (z)$, and $G_{\rho} (z) =0$. Applying it, we obtain that the random variables $\{ p_{A;k} - \E p_{A;k} \}_{A>0;k \ge 1}$ converge to the jointly Gaussian limit with zero mean and covariance
\begin{multline}
\label{eq:extr-gauss-proof}
\lim_{N \to \infty} \frac{\cov(p_{A_1;k_1}, p_{A_2, k_2})}{N^{k_1+k_2}} = \frac{1}{(2 \pi \ii)^2} \oint_{|z|=\ep} \oint_{|w| = 2 \ep} \left( \frac{1}{z} +1 + \frac{(1+z) \F_{\mathbf J} (1+z)}{A_1} \right)^{k_1} \\ \times \left( \frac{1}{w} +1 + (1+w) \frac{(1+z) \F_{\mathbf J} (1+z)}{A_2} \right)^{k_2} \frac{1}{(z-w)^2} dz dw,
\end{multline}
where $0 < A_1 \le A_2$ and $\ep \ll 1$.

Note that the formula \eqref{eq:extr-gauss-proof} for covariance already contains the cross factor
$\frac{1}{(z-w)^2}$ --- this is a key indication of the presence of the Gaussian Free Field. The
derivation of Theorem \ref{theorem:extr-char-gff-1} from \eqref{eq:extr-gauss-proof} is completely
analogous to the derivation of Theorem \ref{th:ctl-proj-GFF-1} from \eqref{eq:lem-two-int-proj}
modulo the fact that one needs to use Proposition \ref{prop:comp-struct-ext} instead of Proposition
\ref{prop:comp-struct-tilings} in order to deal with the arising level curves.
\end{proof}

\subsection{Domino tilings of Aztec diamond and Gaussian Free Field}
\label{Section_Aztec_proof}

In this section we prove Theorem
\ref{theorem:extr-char-aztec}.

Let us formally describe the probability measure on a particle system which turns
out to be equivalent to the domino tiling model.

For each $t=1,2,\dots, N$ let $\la^{(t)}$, $\upsilon^{(t)}$ be signatures of length
$t$, and let $\beta := \frac{\mathfrak{q}}{\mathfrak{q}+1}$, where $\mathfrak{q}$ is a parameter from Section \ref{Section_Aztec_statement}.

Define the coefficients $\kappa (\la^{(t)} \to \upsilon^{(t)})$ via
\begin{equation*}
\frac{s_{\la^{(t)}} (x_1, \dots, x_t)}{s_{\la^{(t)}} (1^t)} \prod_{i=1}^t
(1-\beta + \beta x_i) = \sum_{\upsilon^{(t)} \in \GT_t} \kappa (\la^{(t)} \to
\upsilon^{(t)}) \frac{s_{\upsilon^{(t)}} (x_1, \dots, x_t)}{s_{\upsilon^{(t)}}
(1^t)}.
\end{equation*}
Recall that the coefficients $\pr_{t \to t-1} (\upsilon^{(t)} \to \la^{(t-1)})$ were
defined in Section \ref{Section_multilevel_specified}. The branching rule and the
Pieri rule for Schur polynomials imply that the coefficients $\kappa (\la^{(t)} \to
\upsilon^{(t)})$ and $\pr_{t \to t-1} (\upsilon^{(t)} \to \la^{(t-1)})$ are nonnegative.

Define the probability measure on the sets of signatures of the form $(\la^{(N)},
\upsilon^{(N)}, \la^{(N-1)}, \upsilon^{(N-1)}, \dots, \la^{(2)}, \upsilon^{(2)},
\la^{(1)})$ by the formula
\begin{multline}
\label{eq:meas-Aztec} \mathrm{Prob} (\la^{(N)}, \upsilon^{(N)}, \la^{(N-1)},
\upsilon^{(N-1)}, \dots, \la^{(2)}, \upsilon^{(2)}, \la^{(1)}) \\ := 1_{\la^{(N)} =
(0^N) } \prod_{i=2}^{N} \kappa ( \la^{(i)} \to \upsilon^{(i)}) \pr_{i \to (i-1)}
(\upsilon^{(i)} \to \la^{(i-1)}),
\end{multline}
(it can be directly checked by induction that the total sum of these weights is 1).
Let $\mathbb S_N$ be the set of such configurations that has a nonzero probability
measure.

\begin{proposition}
\label{prop:part-aztec-biject} There is a bijection between $\mathbb S_N$ and the set of domino
tilings of the Aztec diamond of size $N$. Moreover, under this bijection the measure
\eqref{eq:meas-Aztec} turns into the measure $\mathfrak{q}^{\frac{1}{2}(\text{number of horizontal dominos})}
\cdot {(1+\mathfrak{q})^{-N(N+1)/2}}$ on the set of domino tilings of the Aztec diamond of size
$N$.
\end{proposition}
\begin{proof}
This fact is well-known, and essentially two sequences of signatures are yellow and
green particles in Figure \ref{Fig_Domino_1}. It was implicitly used in \cite{Joh}
and \cite{BorFer}; for a recent exposition, see \cite{BCC}. See also \cite{BK},
where a generalization of this construction is used for a study of domino tilings of
more general domains.
\end{proof}

The bijection described in Proposition \ref{prop:part-aztec-biject} allows to
translate all results about a two-dimensional particle array into the geometric
language of domino tilings. We will now proceed in the language of arrays.




\begin{proof}[Proof of Theorem \ref{theorem:extr-char-aztec}]

Let $\la^{(1)}, \la^{(2)}, \dots, \la^{(N)}$ be random signatures distributed according to the measure \eqref{eq:meas-Aztec}. For $a<1$ let $p_{k;t}$ be the $k$-th power degree of the coordinates of the signature $\la^{([a N])}$. That is, we have
$$
p_{k;t}^{[a N]} := \sum_{i=1}^{[a N]} (\la_i^{([a N])} + [a N] -i)^k.
$$
\begin{proposition}
\label{prop:aztec-gff-2}
In the notations of Theorem \ref{theorem:extr-char-aztec}, the collection of random variables $\{ N^{-k} p_{k;t}^{[a N]} \}_{k \in \N; 0<a \le 1}$ is asymptotically Gaussian with the limit covariance
\begin{multline}
\label{eq:aztec-moment-formula}
\lim_{N \to \infty} \frac{\cov \left( p_{k_1}^{[a_1 N]}, p_{k_2}^{[a_2 N]} \right)}{N^{k_1+k_2}} = \frac{ a_1^{k_1} a_2^{k_2} }{(2 \pi \ii)^2} \oint_{|z|=\ep} \oint_{|w| = 2 \ep} \left( \frac{1}{z} +1 + \frac{(1+z)(1-a_1) \beta}{a_1 (1- \beta + \beta (z+1) )} \right)^{k_1} \\ \times \left( \frac{1}{w} +1 + \frac{(1+w)(1-a_2)\beta}{a_2 (1 - \beta + \beta (w+1))} \right)^{k_2} \frac{1}{(z-w)^2} dz dw,
\end{multline}
where $0< a_1 \le a_2$ and $\ep \ll 1$.
\end{proposition}
\begin{proof}
By construction, the probability measure \eqref{eq:meas-Aztec} satisfies the assumptions of Theorem \ref{th:general-for-domino}. Note that the Schur generating function on the level $[a N]$ (that is, on signatures of length $[a N]$) is equal to
$$
 \prod_{i=1}^{[a N]} (1 - \beta + \beta x_i)^{N-[a N]}.
$$
Thus, the application of Theorem \ref{th:general-for-domino} implies this proposition.
\end{proof}

Notice that the equation \eqref{eq:aztec-moment-formula} contains the cross factor $\frac{1}{(z-w)^2}$. Using $\beta = \frac{\mathfrak{q}}{\mathfrak{q}+1}$, one directly checks that the equation
$$
 \frac{1}{z} +1 + \frac{(1+z)(1-a_1) \beta}{a_1 (1- \beta + \beta (z+1) )} = \frac{y}{\eta}
$$
coincides with the equation given in Proposition \ref{prop:aztec-struct}. Theorem
\ref{theorem:extr-char-aztec} can be obtained from Proposition \eqref{prop:aztec-gff-2} with the
use of Proposition \ref{prop:aztec-struct} in exactly the same way as in the previous two sections.
\end{proof}

\subsection{Tensor products and degeneration to random matrices}
\label{sec:9-4}

There exists a way to degenerate the tensor products of representations into sums of
Hermitian matrices, see e.g. Section 1.3 of \cite{BG} for details. Our goal is to
show that under this degeneration the covariance for tensor products (given in
Theorem \ref{theorem:tensor}) turns into the covariance for the sum of random
matrices.

Let $a_i (N)$ and $b_i (N)$, $i=1, \dots ,N$, be two sets of reals, let $A(N)$ be a
diagonal $N \times N$ matrix with eigenvalues $\{ a_i (N) \}_{i=1}^N$, and let
$B(N)$ be a diagonal $N \times N$ matrix with eigenvalues $\{ b_i (N) \}_{i=1}^N$.
Assume that $U_N$ is a uniformly (=Haar-distributed) random unitary $N\times N$
matrix. Let
$$
H_N := A(N) + U_N^{-1} B(N) U_N,
$$
and let $\la_1 (H_N) \ge \dots \ge \la_N (H_N)$ be (random) eigenvalues of $H_N$.
Set
$$
p_k (H_N) := \sum_{i=1}^N \la_i^k (H_N).
$$
Assume that
$$
\frac{1}{N} \sum_{i=1}^N \delta(a_i(N)) \xrightarrow[N \to \infty]{} \hat \mes_1, \qquad \frac{1}{N} \sum_{i=1}^N \delta(b_i(N) ) \xrightarrow[N \to \infty]{} \hat \mes_2, \qquad \mbox{weak convergence},
$$
where $\hat \mes_1, \hat \mes_2$ are probability measures on $\R$ with compact supports.

Let
$$
\hat f (z) := \left( \frac{-1}{z} + R_{\hat \mes_1} (-z) + R_{\hat \mes_2} (-z) \right)^{(-1)},
$$
where by $\mathbf F^{(-1)} (z)$ we mean the functional inverse of the function $\mathbf F(z)$, and the function $R_{\mes} (z)$ was introduced in Section \ref{sec:statem-requir-formulae}.

In the limit regime $N \to \infty$ the covariance of the functions $p_k (H)$ is given by the following formula, see \cite[Chapter 10]{PS}
\begin{multline*}
 \lim_{N\to\infty} \cov(p_k(H_N), p_l(H_N) ) =
 \frac{1}{(2\pi\ii)^2} \oint_z \oint_w  z^k
 w^{l}  \frac{\partial^2}{\partial z \partial
 w} \left(
 \log \left( R_{\hat \mes^2} (- \hat f(z))- \frac{1}{\hat f(z)} \right. \right. \\ \left. \left. - R_{\hat \mes^2} (- \hat f(w))+ \frac{1}{\hat f(w)} \right)
 +\log
 \left( R_{\hat \mes^1} (- \hat f(z))- \frac{1}{\hat f(z)} - R_{\hat \mes^1} (- \hat f(w))+
\frac{1}{\hat f(w)} \right)  \right. \\ \left. -\log (z-w)-\log \left( \frac{1}{\hat f(w)}- \frac{1}{\hat f(z)} \right)
 \right) dz dw,
\end{multline*}
where the contours encircle infinity and no other poles of the integrand.

Let us make a change of variables $z=\hat f^{(-1)}(-\hat z)$ (i.e.\ $\hat z=-f(z)$), which, in
particular, swaps $0$ with $\infty$. Note that conveniently $\frac{\partial^2}{\partial z \partial
w} F(z,w) \cdot dz dw$ is a differential form for an arbitrary function $F(z,w)$, and thus it does not change at all. Therefore, we obtain
\begin{multline}
\label{eq:degen-matrix}
 \lim_{N\to\infty}\cov(p_k (H_N), p_l (H_N))
 =\frac{1}{(2\pi\ii)^2} \oint_{\hat z} \oint_{\hat w}  \left( \hat f^{(-1)}(-\hat z) \right)^k
 \left( \hat f^{(-1)}(- \hat w ) \right)^{l} \\ \times \frac{\partial^2}{\partial \hat z \partial
 \hat w} \left(
 \log \left( R_B(\hat z) + \frac{1}{\hat z} - R_B (\hat w)- \frac{1}{\hat w} \right)
 +\log \left( R_A (\hat z)  + \frac{1}{\hat z} - R_A (\hat w) -1/\hat w \right)  \right. \\  \left.
  -\log \left( \hat f^{(-1)}(-\hat z) - \hat f^{(-1)}(-\hat z) \right)-\log \left( -\frac{1}{\hat w} +\frac{1}{\hat
z} \right)
 \right) d\hat z d\hat w,
\end{multline}
where contours of integration encircle zero and no other poles.

\medskip

Recall the setting and notations of Theorem \ref{theorem:tensor}. Let $\mes^1, \mes^2$ be two limiting measures for signatures.
Theorem \ref{theorem:tensor} asserts that random variables $p_k$ which corresponds to the measure $m[T^{\la^1(N)} \otimes T^{\la^2 (N)}]$ have the covariance given by the formula
\begin{multline}
\label{eq:degen-tensor}
 \lim_{N \to \infty} N^{-k-l} \cov \left( p_k, p_l \right) \\ = \frac{1}{(2 \pi \ii)^2} \oint_{|z|=\ep} \oint_{|w|=2 \ep} \left( \frac{1}{z} +1 + (1+z) \left( H'_{\mes^1} (1+z)+H'_{\mes^2} (1+z) \right) \right)^{k_1} \\ \times \left( \frac{1}{w} + 1 + (1+w) \left( H'_{\mes^1} (1+w)+ H'_{\mes^2} (1+w) \right) \right)^{k_2} Q_{\mes^1, \mes^2}^{\otimes} (z, w) dz dw,
 \end{multline}
where contours of integration encircle zero and no other poles.

\begin{proposition}
\label{prop-degener-to-matr}
The right-hand side of \eqref{eq:degen-tensor} converges to the right-hand side of \eqref{eq:degen-matrix} in the limit
$$
\mes^i = \hat \mes^i \delta^{-1}, \quad z = \delta \hat z, \quad w= \delta \hat w, \quad i=1,2,
$$
where positive real $\delta$ tends to 0.
\end{proposition}
\begin{proof}
By a straightforward computation we have
$$
\lim_{\delta \to 0} H'_{\mes^i} (1+ z) = R_{\hat \mes^i} (\hat z), \quad i=1,2.
$$
We can further transform as $\delta\to 0$ the $Q^{\otimes}_{m^1,m^2}(z,w)$ part of (9.18) (without changing the integral) to
\begin{multline*}
\partial_{\hat z} \partial_{\hat w} \left( \log \left( 1 - \hat z \hat w \frac{R_{\hat \mes^1} (\hat z) - R_{\hat \mes^1} (w)}{\hat z-\hat w} \right) \right.  \\ \left.
 + \log \left( 1 - \hat z \hat w \frac{R_{\hat \mes^2} (\hat z) - R_{\hat \mes^2} (w)}{\hat z-\hat w} \right)
 - \log \left( 1 - \hat z \hat w \frac{R_{\hat \mes^1} (\hat z)+R_{\hat \mes^2} (\hat z) - R_{\hat \mes^1} (w) - R_{\hat \mes^2} (w)}{\hat z-\hat w} \right) \right)
\end{multline*}

Plugging these limit relations into the right-hand side of \eqref{eq:degen-tensor}, we obtain that as $\delta \to 0$ we have
\begin{multline}
\frac{1}{(2 \pi \ii)^2} \oint_{\hat z} \oint_{\hat w} \left( \frac{1}{z} + R_{\hat \mes^1} (\hat z) + R_{\hat \mes^2} (\hat z) +o(1) \right)^k \left( \frac{1}{w} + R_{\hat \mes^1} (\hat w) + R_{\hat \mes^2} (\hat w) +o(1) \right)^l \\ \times \partial_{\hat z} \partial_{\hat w} \left( \log \left( 1 - \hat z \hat w \frac{R_{\hat \mes^1} (\hat z) - R_{\hat \mes^1} (w)}{\hat z-\hat w} +o(1) \right)
 + \log \left( 1 - \hat z \hat w \frac{R_{\hat \mes^2} (\hat z) - R_{\hat \mes^2} (w)}{\hat z-\hat w} +o(1) \right)
 \right.  \\ \left. - \log \left( 1 - \hat z \hat w \frac{R_{\hat \mes^1} (\hat z)+R_{\hat \mes^2} (\hat z) - R_{\hat \mes^1} (w) - R_{\hat \mes^2} (w)}{\hat z-\hat w} + o(1) \right) \right) d \hat z d \hat w.
\end{multline}

Therefore, in the limit we obtain the right-hand side of \eqref{eq:degen-matrix}.
\end{proof}

\begin{remark}
This limit regime of Proposition \ref{prop-degener-to-matr} is closely related to
the \textit{semi-classical limit}; see Section 1.3 of \cite{BG} for more details on
this transition.
\end{remark}

\section{Appendix: Law of Large Numbers}
\label{sec:4-2}

In this section we prove Theorem \ref{theorem:LLN-general}. In fact, this is \cite[Theorem 5.1]{BG}, and we comment on slight differences here.

The first difference is that \cite[Theorem 5.1]{BG} requires that the Schur generating function $S_{\rho}$ of a probability measure $\rho = \rho(N)$ satisfies the condition
$$
\lim_{N \to \infty} \frac{\log S_{\rho} (x_1, \dots, x_k, 1^{N-k})}{N} = U (x_1)+ \dots + U (x_k), \qquad \mbox{for any fixed $k \ge 1$},
$$
where $U$ is a holomorphic function and the convergence is uniform in an open neighborhood of $(x_1, \dots, x_k) = 1^{N-k}$. It is clear that this condition implies the properties of Definition \ref{def:LLN-appr_schur} with $\pa_z^l U (z) = \c_l$, for $l \ge 1$. The uniform convergence of holomorphic functions implies the convergence of Taylor coefficients, though the opposite is not always correct. However, in the proof of Theorem 5.1 from \cite{BG} we use only the convergence of Taylor coefficients, so the same proof (see Section 5.2 of \cite{BG} ) goes without any changes.

The second difference is that the right-hand side of equation
\eqref{eq:LLN-theorem-statem} in Theorem \ref{theorem:LLN-general} is written in an
integral form rather than in a summation form. The computation which shows the
equivalence of these two expressions is given in equation (6.2) of \cite{BG}.

\end{document}